\documentclass[12pt]{amsart}
\usepackage{amsmath}
\usepackage{amsxtra}
\usepackage{amscd}
\usepackage{amsthm}
\usepackage{amsfonts}
\usepackage{amssymb}
\usepackage{eucal}

\input xy
\xyoption{all}

\newtheorem{thm}{Theorem}
\newtheorem{conj}[thm]{Conjecture}
\newtheorem{cor}[thm]{Corollary}
\newtheorem{lem}[thm]{Lemma}

\theoremstyle{remark}

\theoremstyle{definition}
\newtheorem{definition}{Definition}

\numberwithin{equation}{section}

\newcommand{\secref}[1]{Section~\ref{#1}}
\newcommand{\lemref}[1]{Lemma~\ref{#1}}

\newcommand{\corref}[1]{Corollary~\ref{#1}}

\newcommand{\conjref}[1]{Conjecture~\ref{#1}}
\newcommand{\nc}{\newcommand}
\nc{\on}{\operatorname}
\nc{\ch}{\mbox{ch}}
\nc{\Z}{{\mathbb Z}}
\nc{\C}{{\mathbb C}}
\nc{\pone}{{\mathbb P}^1}
\nc{\pa}{\partial}
\nc{\F}{{\mathcal F}}
\nc{\arr}{\rightarrow}
\nc{\larr}{\longrightarrow}
\nc{\al}{\alpha}
\nc{\ri}{\rangle}
\nc{\lef}{\langle}
\nc{\W}{{\mathbb W}}
\nc{\la}{\lambda}
\nc{\ep}{\epsilon}
\nc{\su}{\widehat{{\mathfrak s}{\mathfrak l}}_2}
\nc{\sw}{{\mathfrak s}{\mathfrak l}}
\nc{\g}{{\mathfrak g}}
\nc{\h}{{\mathfrak h}}
\nc{\n}{{\mathfrak n}}
\nc{\N}{\widehat{\n}}
\nc{\G}{\widehat{\g}}
\nc{\De}{\Delta}
\nc{\gt}{\widetilde{\g}}
\nc{\Ga}{\Gamma}
\nc{\one}{{\mathbf 1}}
\nc{\z}{{\mathfrak Z}}
\nc{\La}{\Lambda}
\nc{\wt}{\widetilde}
\nc{\wh}{\widehat}
\nc{\cri}{_{\kappa_c}}
\nc{\kk}{\underline{\mathbf C}}
\nc{\sun}{\widehat{\sw}_N}
\nc{\si}{\sigma}
\nc{\el}{\ell}
\nc{\bi}{\bibitem}
\nc{\om}{\omega}
\nc{\ol}{\overline}
\nc{\ds}{\displaystyle}
\nc{\dzz}{\frac{dz}{z}}
\nc{\Res}{\on{Res}}
\nc{\mc}{\mathcal}
\nc{\Cal}{\mathcal}
\nc{\bb}{{\mathfrak b}}
\nc{\ot}{\otimes}
\nc{\R}{{\mathbb R}}
\nc{\yy}{{\mc Y}}
\nc{\ga}{\gamma}

\nc{\us}{\underset}
\nc{\opl}{\oplus}

\nc{\Fq}{{\mathbb F}_q}
\nc{\Mq}{{\mathcal M}}
\nc{\Rep}{\on{Rep}}
\nc{\sssec}{\subsubsection}
\nc{\ssec}{\subsection}
\nc{\lan}{\langle}
\nc{\ran}{\rangle}

\nc{\D}{\mathcal D}
\nc{\Vect}{\on{Vect}}
\nc{\ghat}{\G}
\nc{\T}{\mc T}
\nc{\Tloc}{\T^\g_{\on{loc}}}
\nc{\vac}{|0\ran}
\nc{\Wick}{{\mb :}}
\nc{\mb}{\mathbf}
\nc{\delz}{\partial_z}
\nc{\K}{{\mathbb K}}
\nc{\cali}{\mathcal}
\nc{\li}{\mathfrak l}
\nc{\lt}{\widetilde{\li}}
\nc{\astar}{a^*}
\nc{\cA}{{\mc A}}
\nc{\ka}{\kappa}

\nc{\OO}{{\mc O}}
\nc{\AutO}{\on{Aut}\OO}
\nc{\DerO}{\on{Der}\OO}
\nc{\DerpO}{\on{Der}_+\OO}
\nc{\Au}{{\mc A}ut}
\nc{\mf}{\mathfrak}
\nc{\V}{{\mathcal V}}
\nc{\hh}{\wh{\h}}

\nc{\pp}{{\mathfrak p}}
\nc{\mm}{{\mathfrak m}}
\nc{\rr}{{\mathfrak r}}
\nc{\ket}{\rangle}
\nc{\zz}{{\mathfrak z}}
\nc{\gr}{\on{gr}}
\nc{\Spe}{\on{Spec}}
\nc{\rv}{\crho}
\nc{\can}{\on{can}}
\nc{\CC}{\on{Op}_G(D))}
\nc{\Op}{\on{Op}_G(D)}
\nc{\MOp}{\on{MOp}_G(D)}
\nc{\Db}{{\mathbb D}}
\nc{\ww}{w}

\nc{\af}{{\mathbb A}^1}
\nc{\bs}{\backslash}
\nc{\laa}{(\la_i)}
\nc{\zn}{(z_i)}

\nc{\cla}{\check{\la}}
\nc{\cmu}{\check{\mu}}
\nc{\crho}{\check{\rho}}
\nc{\chal}{\check{\al}}
\nc{\cc}{{\mathfrak c}}

\nc{\MM}{{\mathbb M}}

\nc{\ZZ}{{\mc Z}}

\nc{\UU}{{\mathbb U}}

\nc{\Conn}{\on{Conn}(\Omega^{\crho})}
\nc{\Con}{\on{Conn}(\Omega^{-\rho})}
\nc{\Co}{\on{Conn}(\Omega^{\rho})}

\nc{\ppart}{(\!(t)\!)}
\nc{\pparl}{(\!(\la)\!)}
\nc{\zpart}{(\!(z)\!)}
\nc{\ppzi}{(\!(t-z_i)\!)}
\nc{\ppinf}{(\!(t^{-1})\!)}
\nc{\Ind}{\on{Ind}}
\nc{\I}{{\mathbb I}}

\nc{\Bun}{\on{Bun}}

\newcommand {\IC}{\on{IC}}

\nc{\gtil}{\wt{\g}}
\nc{\ntil}{\wt{\n}}
\nc{\htil}{\wt{\h}}
\nc{\gbar}{\ol{\g}}
\nc{\nbar}{\ol{\n}}
\nc{\bbar}{\ol{\bb}}
\nc{\lhat}{\wh{\mf l}}

\nc{\ovc}{\overset{\circ}}
\nc{\Gr}{\on{Gr}}

\nc{\AD}{{\mathbb A}}
\nc{\Gm}{{\mathbb G}_m}
\nc{\Ql}{{\mathbb Q}_\ell}

\nc{\Loc}{\on{Loc}}

\nc\GG{\mathbb G}
\nc\Xb{\mathbf X}
\nc\inv{{\rm inv}}
\nc\isom{=}
\nc\Hc{{\mathcal H}}
\nc\ovl{\overline}

\newcommand{\BD}{{}}
\newcommand\Oc{\mathcal{O}}
\nc\Lc{{\mathcal L}}
\nc\Ec{{\mathcal E}}
\nc\wF{{\mc K}}
\nc\pr{{\rm pr}}
\nc\SL{{\rm SL}}
\nc\PGL{{\rm PGL}}
\nc\GL{{\rm GL}}
\nc\LG{{}^L\negthinspace G}
\nc\LH{{}^L\negthinspace H}
\nc\Pc{{\mc P}}
\nc\Ac{{\mc A}}
\nc\dv{{\rm div}}
\nc\Fc{{\mc F}}
\nc\M{{\mc M}}
\nc\Q{{\mathbb Q}}
\nc\Fp{{\mathbb F}_p}

\nc{\und}{\underline}

\def\H{{\mathbb H}}

\begin{document}

\title{Langlands Program, Trace Formulas, and their Geometrization}

\dedicatory{Notes for the AMS Colloquium Lectures at the Joint
  Mathematics Meetings in Boston, \\ January 4--6, 2012}

\author[Edward Frenkel]{Edward Frenkel}\thanks{Supported by DARPA
  under the grant HR0011-09-1-0015}

\address{Department of Mathematics, University of California,
Berkeley, CA 94720, USA}

\begin{abstract}

  The Langlands Program relates Galois representations and automorphic
  representations of reductive algebraic groups. The trace formula is
  a powerful tool in the study of this connection and the Langlands
  Functoriality Conjecture. After giving an introduction to the
  Langlands Program and its geometric version, which applies to curves
  over finite fields and over the complex field, I give a survey of my
  recent joint work with Robert Langlands and Ng\^o Bao Ch\^au
  \cite{FLN,FN} on a new approach to proving the Functoriality
  Conjecture using the trace formulas, and on the geometrization of
  the trace formulas. In particular, I discuss the connection of the
  latter to the categorification of the Langlands correspondence.

\end{abstract}

\maketitle

\tableofcontents

\newpage

\vspace*{10mm}

\section{Introduction}

\setcounter{footnote}{1}

The Langlands Program was initiated by Robert Langlands in the late
60s in order to connect number theory and harmonic analysis
\cite{L}. In the last 40 years a lot of progress has been made in
proving the Langlands conjectures, but much more remains to be
done. We still don't know the underlying reasons for the deep and
mysterious connections suggested by these conjectures. But in the
meantime, these ideas have propagated to other areas of mathematics,
such as geometry and representation theory of infinite-dimensional Lie
algebras, and even to quantum physics, bringing a host of new ideas
and insights. There is hope that expanding the scope of the Langlands
Program will eventually help us get the answers to the big questions
about the Langlands duality.

In this lectures I will give an overview of this subject and describe
my recent joint work with Robert Langlands and Ng\^o Bao Ch\^au
\cite{FLN,FN} on the Functoriality Principle and the geometrization of
the trace formulas.

The key objects in the Langlands Program are {\em automorphic
  representations} of a reductive algebraic group $G$ over a global
field $F$, which is either a number field or the field of rational
functions on a smooth projective curve $X$ over a finite field. These
are the constituents in the decomposition of the space $L_2(G(F)\bs
G(\AD_F))$ under the right action of $G(\AD_F)$, where $\AD_F$ is the
ring of ad\`eles of $F$ (see \secref{lc} for details). Let now $G$ and
$H$ be two reductive algebraic groups over $F$, and $\LG$ and $\LH$
their Langlands dual groups as defined in \cite{L}. The Langlands {\em
  Functoriality Principle} \cite{L} states that for each admissible
homomorphism
$$\LH \to {}\LG$$ there exists a {\em transfer} of automorphic
representations, from those of $H(\AD_F)$ to those of $G(\AD_F)$,
satisfying some natural properties. Functoriality has been established
in some cases, but is still unknown in general (see
\cite{Arthur:funct} for a survey).

In \cite{FLN}, following \cite{L:BE,L:PR} (see also
\cite{L:ST,L:IAS}), a strategy for proving functoriality was
proposed. In the space of automorphic functions on $G(F)\bs G(\AD_F)$
we construct a family of integral operators ${\mb K}_{d,\rho}$, where
$d$ is a positive integer and $\rho$ is a finite-dimensional
representation of $\LG$, which for sufficiently large $d$ project onto
the automorphic representations of $G(\AD_F)$ that come by
functoriality from automorphic representations of certain groups $H$
determined by $\rho$. We then apply the {\em Arthur--Selberg trace
  formula}
\begin{equation}    \label{tr}
\on{Tr} {\mb K}_{d,\rho} = \int K_{d,\rho}(x,x) \; dx,
\end{equation}
where $K_{d,\rho}(x,y)$ is the kernel of ${\mb K}_{d,\rho}$, a
function on the square of $G(F)\bs G(\AD_F)$. The left hand, {\em
  spectral}, side of \eqref{tr} may be written as the sum over
irreducible automorphic representations $\pi$ of $G(\AD_F)$:
\begin{equation}    \label{LHS}
\sum_\pi m_\pi \on{Tr}({\mb K}_{d,\rho},\pi),
\end{equation}
where $m_\pi$ is the multiplicity of $\pi$ in the space of $L_2$
functions on $G(F)\bs G(\AD_F)$) (here we ignore the continuous part
of the spectrum). The right hand, {\em orbital}, side of \eqref{tr}
may be written as a sum over the conjugacy classes $\gamma$ in $G(F)$:
\begin{equation}    \label{RHS}
\sum_{\gamma \in G(F)/\on{conj.}} a_\gamma O_\gamma(K_{d,\rho}),
\end{equation}
where $O_\gamma(K_{d,\rho})$ is an ``orbital integral'': an integral
over the conjugacy class of $\gamma$ in $G(\AD_F)$ (see \secref{trace
  for}).

The idea is to analyze the orbital side of the trace formula and
compare the corresponding orbital integrals of the group $G$ to those
of the groups $H$. This way one hopes to connect the spectral sides of
the trace formulas for $G$ and $H$ and hence prove functoriality. In
\cite{FLN,FN} we related this to the geometric and categorical forms
of the Langlands correspondence and made the first steps in developing
the geometric methods for analyzing these orbital integrals in the
case of the function field of a curve $X$ over a finite field $\Fq$.

\medskip

By a {\em geometrization} of the trace formula \eqref{tr} we
understand representing each side as the trace of the Frobenius
automorphism on a vector space equipped with an action of
$\on{Gal}(\ol{\mathbb F}_q/\Fq)$ (the Galois group of the finite field
$\Fq$, over which our curve $X$ is defined).

Such a reformulation is useful because, first of all, unlike mere
numbers, these vector spaces may well carry additional structures that
could help us understand the connections we are looking for. For
example, in B.C. Ng\^o's beautiful recent proof of the fundamental
lemma \cite{Ngo:FL}, he used the \'etale cohomologies of certain
moduli spaces (the fibers of the so-called Hitchin map). Ng\^o showed
that these cohomologies carry natural actions of finite groups and he
used these actions to isolate the ``right'' pieces of these
cohomologies and to prove their isomorphisms for different groups.
The equality of the traces of the Frobenius on these vector spaces
then yields the fundamental lemma.

The second reason why geometrization is useful is that we expect that,
unlike the corresponding numbers, these vector spaces would also make
sense for curves defined over $\C$, so that we would obtain a version
of the trace formula for complex curves.

\medskip

As the first step in the program of geometrization of trace formulas,
we showed in \cite{FN} that the kernels $K_{d,\rho}$ constructed in
\cite{FLN} may be obtained using the Grothendieck {\em
  faisceaux--fonctions} dictionary from perverse sheaves ${\mc
  K}_{d,\rho}$ on a certain algebraic stack over the square of
$\Bun_G$, the moduli stack of $G$-bundles on $X$. Hence the right hand
side of the trace formula \eqref{tr} may indeed be written as the
trace of the Frobenius on the \'etale cohomology of the restriction of
${\mc K}_{d,\rho}$ to the diagonal in $\Bun_G \times \Bun_G$. This may
be further rewritten as the cohomology of a sheaf defined on the
moduli stack of ``$G$-pairs'', which is closely related to the Hitchin
moduli stack of Higgs bundles on $X$ (see \cite{FN} and \secref{def
  stack}). Thus, we obtain a geometrization of the orbital side of the
trace formula \eqref{tr}.

The idea then is to use the geometry of these moduli stacks to prove
the desired identities of orbital integrals for $G$ and $H$ by
establishing isomorphisms between the corresponding cohomologies. As
we mentioned above, a geometric approach of this kind turned out to be
very successful in B.C. Ng\^o's recent proof of the fundamental lemma
\cite{Ngo:FL}. The elegant argument of \cite{Ngo:FL} takes advantage
of the decomposition of the cohomology of the fibers of the Hitchin
map under the action of finite groups. In our case, the decomposition
of the cohomology we are looking for does not seem to be due to an
action of a group. Hence we have to look for other methods. Some
conjectures in this direction were formulated in \cite{FN}. We discuss
them in \secref{conj gen}.

\medskip

The moduli stacks and the sheaves on them that appear in this picture
have natural analogues for curves over $\C$, and hence the
geometrization allows us to include complex curves into
consideration. We can then use the methods of complex algebraic
geometry (some of which have no obvious analogues over a finite field)
to tackle the questions of functoriality that we are interested in.

\medskip

Thus, the geometrization of the right hand (orbital) side \eqref{RHS}
of \eqref{tr} is the cohomology of a sheaf on the moduli stack of
$G$-pairs. In \cite{FN} we also looked for a geometrization of the
left hand (spectral) side of \eqref{tr}, trying to interpret the sum
\eqref{LHS} as the Lefschetz trace formula for the trace of the
Frobenius on the \'etale cohomology of an $\ell$-adic sheaf. It is not
obvious how to do this, because the set of the $\pi$'s appearing in
\eqref{LHS} is not the set of points of a moduli space (or stack) in
any obvious way.

However, according to the Langlands correspondence \cite{L}, reviewed
in \secref{lc}, the $L$-packets of irreducible (tempered) automorphic
representations of $G(\AD_F)$ are supposed to be parametrized by
the homomorphisms
$$
\sigma: W(F) \to \LG,
$$
where $W(F)$ is the Weil group of the function field $F$ and $\LG$ is
the Langlands dual group to $G$. Therefore, assuming the Langlands
correspondence, we may rewrite the sum \eqref{LHS} as a sum over such
$\sigma$. Unfortunately, if $k$, the field of definition of our
curve $X$, is a finite field, there is no reasonable algebraic stack
whose $k$-points are the equivalence classes of homomorphisms
$\sigma$. But if $k=\C$, such a stack exists! If we restrict
ourselves to the unramified $\sigma$, then we can use the algebraic
stack $\Loc_{\LG}$ of flat $\LG$-bundles on $X$. Hence we can pose
the following question: define a sheaf on this stack such that its
cohomology (representing the left hand side of \eqref{tr} in the
complex case) is isomorphic to the cohomology representing the right
hand side of \eqref{tr}. This isomorphism would then be a {\em
geometrization of the trace formula} \eqref{tr}.

The idea of Ng\^o and myself \cite{FN} is that the answer may be
obtained in the framework of a {\em categorical form of the geometric
  Langlands correspondence}, which we review below. It is a
conjectural equivalence between derived categories of $\OO$-modules
on the moduli stack $\Loc_{\LG}$ and ${\mc D}$-modules on the moduli
stack $\Bun_G$ of $G$-bundles on $X$. Such an equivalence has been
proved in the abelian case by G. Laumon \cite{Laumon} and M. Rothstein
\cite{Rothstein}, and in the non-abelian case it has been suggested as
a conjectural guiding principle by A. Beilinson and V. Drinfeld (see,
e.g., \cite{F:houches,Laf,LL} for an exposition). This categorical
version of the geometric Langlands correspondence also appears
naturally in the $S$-duality picture developed by A. Kapustin and
E. Witten \cite{KW} (see \cite{F:bourbaki} for an exposition).

At the level of objects, the categorical Langlands correspondence
assigns to the skyscraper $\OO$-module supported at a given flat
$\LG$-bundle ${\mc E}$ on $X$ a {\em Hecke eigensheaf} on $\Bun_G$
with ``eigenvalue'' ${\mc E}$. But an equivalence of categories also
gives us non-trivial information about morphisms; namely, the Hom's
between the objects corresponding to each other on the two sides
should be isomorphic. The idea of \cite{FN}, which we review below, is
that for suitable objects the isomorphism of their Hom's yields the
sought-after geometric trace formula. This led us to propose in
\cite{FN} a conjectural geometrization of the trace formula \eqref{tr}
in this framework, which we review in \secref{leap}.

In deriving this geometric trace formula, we work with the Hom's in
the categories of sheaves on the squares of $\Bun_G$ and $\Loc_{\LG}$,
because this is where the kernels of our functors ``live''.

We also obtained in \cite{FN} an analogous statement in the categories
of sheaves on the stacks $\Bun_G$ and $\Loc_{\LG}$ themselves.  The
result is a geometrization of the {\em relative trace formula}, also
known as the Kuznetsov trace formula, see, e.g., \cite{Jacquet}. This
formula has some favorable features compared to the usual trace
formula (for example, only tempered automorphic representations, and
only one representation from each $L$-packet -- the ``generic'' one --
are expected to contribute). But there is a price: in the sum
\eqref{LHS} appears a weighting factor, the reciprocal of the value of
the $L$-function of $\pi$ in the adjoint representation at $s=1$. The
insertion of this factor in this context has been considered
previously in \cite{Sarnak} and \cite{Ven} for the group $GL_2$. Ng\^o
and I showed in \cite{FN} (see \secref{abl fixed}) that this factor
also has a natural geometric interpretation, as coming from the {\em
  Atiyah--Bott--Lefschetz fixed point formula}.

The geometric trace formulas proposed by Ng\^o and myself in \cite{FN}
are still in a preliminary form, because several important issues
need to be worked out. Nevertheless, we believe that they contain
interesting features and even in this rough form might provide a
useful framework for a better geometric understanding of the trace
formula as well as the geometric Langlands correspondence.

\medskip

These notes are organized as follows. We start with a brief
introduction to the classical Langlands correspondence in \secref{lc}
and its geometric and categorical forms in \secref{GLC}. In
\secref{FTF} we survey the trace formula and its applications to the
functoriality of automorphic representations, following \cite{FLN}. In
\secref{gos} we describe, following \cite{FN}, the geometrization of
the orbital side of the trace formula in terms of the cohomology of
certain sheaves on the moduli stacks which are group-like analogues of
the Hitchin moduli stacks of Higgs bundles. The geometrization of the
spectral side of the trace formula is discussed in \secref{leap}, and
the relative trace formula and its geometrization in \secref{rel
  trace}. Here we follow closely \cite{FN}.

\medskip

\noindent {\bf Acknowledgments.} I am grateful to Robert Langlands and
Ng\^o Bao Ch\^au for their collaboration on the papers \cite{FLN,FN}
which are reviewed in these notes.

I thank Ivan Fesenko for his comments on a draft of this paper.

\section{The Classical Langlands Program}    \label{lc}

The Langlands correspondence, in its original form, manifests a deep
connection between number theory and representation theory. In
particular, it relates subtle number theoretic data (such as the
numbers of points of a mod $p$ reduction of an elliptic curve defined
by a cubic equation with integer coefficients) to more easily
discernible data related to automorphic forms (such as the
coefficients in the Fourier series expansion of a modular form on the
upper half-plane). In this section we give an outline this
correspondence.

\subsection{The case of $GL_n$}

Let $\Q$ be the field of rational numbers. Denote by $\ol{\Q}$ its
algebraic closure, the field of algebraic numbers, obtained by
adjoining to $\Q$ the roots of all polynomial equations in one
variable with coefficients in $\Q$. The arithmetic questions about
algebraic numbers may be expressed as questions about the {\em Galois
  group} $\on{Gal}(\ol{\Q}/\Q)$ of all field automorphisms of $\Q$.

A marvelous insight of Robert Langlands was to conjecture \cite{L}
that there exists a connection between $n$-{\em dimensional
  representations} of $\on{Gal}(\ol{\Q}/\Q)$ and irreducible
representations of the group $GL_n({\mathbb A}_{\Q})$ which occur in
the space of $L_2$ functions on the quotient $GL_n(\Q)\bs
GL_n({\mathbb A}_{\Q})$.

Here $\AD_{\Q}$ is the restricted product of all completions of $\Q$:
the field $\Q_p$ $p$-adic numbers, where $p$ runs over all primes, and
the field $\R$ of real numbers. Elements of $\AD_{\Q}$ are infinite
collections
$$
((x_p)_{p \on{prime}},x_\infty),
$$
where for all but finitely many primes $p$ we have $x_p \in \Z_p
\subset \Q_p$, the ring of $p$-adic integers. (This is the meaning of
the word ``restricted'' used above.) We have a diagonal embedding $\Q
\hookrightarrow \AD_{\Q}$, and hence an embedding of groups $GL_n(\Q)
\hookrightarrow GL_n({\mathbb A}_{\Q})$.

Under the right action of $GL_n({\mathbb A}_{\Q})$ on $L_2(GL_n(\Q)\bs
GL_n({\mathbb A}_{\Q}))$, we have a decomposition into irreducible
representations appearing both discretely and continuously:
$$
L_2(GL_n({\Q})\bs GL_n({\mathbb A}_{\Q})) = \bigoplus \; \pi \; \bigoplus
\; \on{continuous} \; \on{spectrum}
$$
The irreducible representations $\pi$ appearing in the discrete
spectrum, as well as the suitably defined constituents of the
continuous spectrum (we are not going to give the precise definition)
are called {\em automorphic representations}. The theory of Eisenstein
series reduces the study of automorphic representations to the study
of those occurring in the discrete spectrum.

Schematically, the {\em Langlands correspondence} for $GL_n$ may be
formulated as a correspondence between the equivalence classes of the
following data:

\begin{equation}    \label{LC n}
\boxed{\begin{matrix} n\text{-dimensional representations} \\
    \text{of } \on{Gal}(\ol{{\Q}}/{\Q}) \end{matrix}} \quad \longrightarrow
    \quad \boxed{\begin{matrix} \text{irreducible automorphic }
    \\ \text{representations of } GL_n({\mathbb A}_{\Q})
\end{matrix}}
\end{equation}
\bigskip

The Langlands correspondence is very useful for understanding deep
questions in number theory. First of all, according to the ``Tannakian
philosophy'', one can reconstruct a group from the category of its
finite-dimensional representations, equipped with the structure of the
tensor product. Describing the equivalence classes of $n$-dimensional
representations of the Galois group may be viewed as a first step
towards understanding its structure.

\medskip

\noindent {\bf Technical point.} In order to describe the tensor
category of representations of $\on{Gal}(\ol{{\Q}}/{\Q})$, one needs
to do much more: one has to consider the categories of automorphic
representations of $GL_n({\mathbb A}_{\Q})$ for all $n$ at once and
define various functors between them corresponding to taking the
direct sums and tensor products of Galois representations. This is
closely related to the functoriality of automorphic representations
that we discuss in \secref{FTF} below.\qed

\medskip

Second, there are many interesting representations of Galois groups
arising in ``nature''. Indeed, the group $\on{Gal}(\ol\Q/\Q)$ acts on
the geometric invariants (such as the \'etale cohomologies) of an
algebraic variety defined over $\Q$. For example, if we take an
elliptic curve $E$ over $\Q$, then we will obtain a representation of
$\on{Gal}(\ol\Q/\Q)$ on its first \'etale cohomology, which is a
two-dimensional vector space (much like the first cohomology of an
elliptic curve defined over $\C$).

\medskip

Under the Langlands correspondence \eqref{LC n}, some important
invariants attached to Galois representations and automorphic
representations have to to match. These are the so-called {\em
  Frobenius eigenvalues} for the former and the {\em Hecke
  eigenvalues} for the latter. They are attached to all but finitely
many primes. We will discuss the latter in more detail in
\secref{hecke eig}, and for the former, see \cite{F:houches},
Section 1.5.

\subsection{Examples}

The correspondence \eqref{LC n} is well understood for $n=1$. This is
the abelian case. Indeed, one-dimensional representations of any group
factor through its maximal abelian quotient. In the case of
$\on{Gal}(\ol{{\Q}}/{\Q})$, the maximal abelian quotient is the Galois
group of the maximal abelian extension of $\Q$. According to the
Kronecker--Weber theorem, this is the field obtained by adjoining to
$\Q$ all roots of unity. One derives from this description that the
corresponding Galois group is isomorphic to the group of connected
components of
$$
GL_1(\Q) \bs GL_1(\AD_{\Q}) = \Q^\times\bs {\mathbb
  A}_{\Q}^\times.
$$
This is what \eqref{LC n} boils down to for $n=1$. The Abelian Class
Field Theory gives a similar adelic description of the maximal abelian
quotient of the Galois group $\on{Gal}(\ol{F}/F)$, where $F$ is a
general number field (see, e.g., \cite{F:houches}, Section 1.2, for
more details).

\medskip

Suppose next that $n=2$. Let $\sigma$ come from the first \'etale
cohomology of the smooth elliptic curve $E$ defined by the equation
\begin{equation}    \label{ell curve}
y^2 = x^3 + ax + b,
\end{equation}
where $a, b \in \Z$ are such that the discriminant is non-zero, $4a^3
+ 27b^2 \neq 0$. The representation $\sigma$ contains a lot of
important information about the curve $E$. The corresponding Frobenius
eigenvalues encode, for each prime $p$ not dividing the discriminant,
the number of points of the reduction of $E$ modulo $p$, $\#
E(\Fp)$. This is simply the number of solutions of the equation
\eqref{ell curve} mod $p$ plus one, corresponding to the point at
infinity (our $E$ is a projective curve).

According to the Langlands correspondence, $\sigma$ should correspond
to a cuspidal automorphic representation $\pi$ of $GL_2(\AD_{\Q})$. To
make things more concrete, we assign to this automorphic
representation in a standard way (see, e.g., \cite{F:houches}, Section
1.6) a modular cusp form
$$
f(q) = \sum_{n=1}^\infty a_n q^n
$$
on the upper half-plane $\{ \tau \in \C \, | \, \on{Im} \tau > 0 \}$,
where $q=e^{2\pi i \tau}$. The matching of the Frobenius and Hecke
eigenvalues under the Langlands correspondence now becomes the
statement of the Shimura--Taniyama--Weil conjecture (now a theorem
\cite{Wiles}): for each $E$ as above there exist a modular cusp form
$f_E(q)$ with $a_1=1$ and
\begin{equation}    \label{matching}
a_p = p + 1 - \# E(\Fp)
\end{equation}
for all primes $p$ not dividing the discriminant of $E$ (and also,
$a_{mn} = a_m a_n$ for all relatively prime $m$ and $n$).

This is a stunning result: the modular form $f_E$ serves as a
generating function of the numbers of points of $E$ mod $p$ for almost
all $p$.

It implies, according to a result of K. Ribet, Fermat's Last Theorem.

One obtains similar statements by analyzing from the point of view of
the Langlands correspondence the Galois representations coming from
other algebraic varieties, or more general motives. This shows the
great power of the Langlands correspondence: it translates difficult
questions in number theory to questions in harmonic analysis.

\subsection{Function fields}

The correspondence \eqref{LC n} is not a bijection. But it becomes a
bijection if we replace $\Q$ (or a more general number field) in
\eqref{LC n} by a function field.

Let $X$ be a smooth projective connected curve over a finite field
$k=\Fq$. The field $\Fq(X)$ of ($\Fq$-valued) rational functions on
$X$ is called the {\em function field} of $X$.

For example, suppose that $X = \pone$. Then $\Fq(X)$ is just the field
of rational functions in one variable. Its elements are fractions
$P(t)/Q(t)$, where $P(t)$ and $Q(t) \neq 0$ are polynomials over $\Fq$
without common factors, with their usual operations of addition and
multiplication.

It turns out that there are many similarities between function fields
and number fields. For example, let's look at the completions of the
function field $\Fq(\pone)$. Consider the field $\Fq\ppart$ of formal
Laurent power series in the variable $t$. An element of this
completion is a series of the form $\sum_{n \geq N} a_n t^n$, where $N
\in \Z$ and each $a_n$ is an element of $\Fq$. Elements of
$\Fq(\pone)$ are rational functions $P(t)/Q(t)$, and such a rational
function can be expanded in an obvious way in a formal power series in
$t$. This defines an embedding of fields $\Fq(\pone) \hookrightarrow
\Fq\ppart$, which makes $\Fq\ppart$ into a completion of $\Fq(\pone)$,
with respect to a standard norm.

Observe that the field $\Fp\ppart$ looks very much like the field
$\Q_p$ of $p$-adic numbers. Likewise, the field $\Fq\ppart$, where $q
= p^n$, looks like a degree $n$ extension of $\Q_p$.

The completion $\Fp\ppart$ corresponds to the maximal ideal in the
ring $\Fq[t]$ generated by $A(t) = t$ (note that $\Fq[t] \subset
\Fq(\pone)$ may be thought of as the analogue of $\Z \subset
\Q$). Other completions of $\Fq(\pone)$ correspond to other maximal
ideals in $\Fq[t]$, which are generated by irreducible monic
polynomials $A(t)$. (Those are the analogues of the ideals $(p)$ in
$\Z$ generated by the prime numbers $p$). There is also a completion
corresponding to the point $\infty \in \pone$, which is isomorphic to
$\Fq(\!(t^{-1})\!)$.

If the polynomial $A(t)$ has degree $m$, then the corresponding
residue field is isomorphic to ${\mathbb F}_{q^m}$, and the
corresponding completion is isomorphic to ${\mathbb
  F}_{q^m}(\!(\wt{t})\!)$, where $\wt{t}$ is the ``uniformizer'',
$\wt{t} = A(t)$. One can think of $\wt{t}$ as the local coordinate
near the ${\mathbb F}_{q^m}$-point corresponding to $A(t)$, just like
$t-a$ is the local coordinate near the $\Fq$-point $a$ of ${\mathbb
  A}^1$. The difference with the number field case is that
all of these completions are non-archimedian; there are no analogues
of the archimedian completions $\R$ or $\C$ that we have in the case
of number fields.

For a general curve $X$, completions of $\Fq(X)$ are also labeled by
its closed points (forming the set denoted by $|X|$), and the
completion corresponding to a point $x$ with residue field ${\mathbb
  F}_{q^n}$ is isomorphic to ${\mathbb F}_{q^n}(\!(t_x)\!)$, where
$t_x$ is the ``local coordinate'' near $x$ on $X$. The subring
${\mathbb F}_{q^m}[[t_x]]$ consisting of the formal Taylor series
(no negative powers of $t_x$) will be denoted by $\OO_x$.

\subsection{The Langlands correspondence}

Let $F$ be the function field of a curve $X$ over $\Fq$. The ring
$\AD_F$ of {\em ad\`eles} of $F$ is by definition the {\em restricted}
product of the fields $F_x$, where $x$ runs over the set of all closed
points of $X$:
$$
\AD_F = \prod_{x \in |X|}{}' \quad F_x.
$$
The word ``restricted'' (reflected by the prime in the above product)
means that we consider only the collections $(f_x)_{x \in X}$ of
elements of $F_x$ in which $f_x \in \OO_x$ for all but finitely many
$x$. The ring $\AD_F$ contains the field $F$, which is embedded into
$\AD_F$ diagonally, by taking the expansions of rational functions on
$X$ at all points. One defines {\em automorphic representations} of
$GL_n(\AD_F)$ as constituents of the space of $L_2$ functions on the
quotient $GL_n(F)\bs GL_n(\AD_F)$ defined as in the number field
case.

The objects that will appear on the right hand side of the Langlands
correspondence are the so-called {\em tempered} automorphic
representations. These are the representations for which the
``Ramanujan hypothesis'' is expected to hold (hence they are sometimes
called ``Ramanujan representations''). This means that each of the
conjugacy classes $\nu_x$ in the complex group $GL_n$ (or $\LG$ in
general) encoding the Hecke eigenvalues of $\pi$ (see \secref{hecke
  eig} below) is unitary. For $G=GL_n$ it is known that all cuspidal
automorphic representations are tempered.

Now let $\ol{F}$ be the separable closure of $F$. We have the Galois
group $\on{Gal}(\ol{F}/F)$ and a natural homomorphism
\begin{equation}    \label{hom Gal}
\on{Gal}((\ol{F}/F) \to \on{Gal}(\ol{{\mathbb F}}_q/\Fq),
\end{equation}
due to the fact that $\Fq$ is the subfield of scalars in $F$. Now,
$\on{Gal}(\ol{{\mathbb F}}_q/\Fq)$ is isomorphic to the pro-finite
completion of $\Z$,
$$
\wh\Z = \underset{\longleftarrow}\lim \; \Z/N\Z =
\underset{\longleftarrow}\lim \; \on{Gal}({\mathbb F}_{q^N}/\Fq).
$$
The preimage of $\Z \subset \wh\Z$ in $\on{Gal}((\ol{F}/F)$ under the
homomorphism \eqref{hom Gal} is called the {\em Weil group} of $F$ and
is denoted by $W(F)$.

On the left hand side of the Langlands correspondence we take the
equivalence classes $n$-dimensional representations of the Weil group
$W(F)$.

The function field analogue of the Langlands correspondence \eqref{LC
  n} is then given by the following diagram:

\begin{equation}    \label{LC n fun}
\boxed{\begin{matrix} n\text{-dimensional} \\
    \text{representations of } W(F) \end{matrix}} \quad
    \longleftrightarrow \quad \boxed{\begin{matrix}
    \text{irreducible tempered automorphic} \\ \text{representations
    of } GL_n({\mathbb A}_F) \end{matrix}}
\end{equation}

\bigskip

$$
\sigma \quad \longleftrightarrow \quad \pi
$$

\bigskip

Under the correspondence the Frobenius eigenvalues of $\sigma$ should
match the Hecke eigenvalues of $\pi$.

This correspondence has been proved by V. Drinfeld \cite{Dr1,Dr2} for
$n=2$ and by L. Lafforgue \cite{LLaf} for $n>2$.

\medskip

\noindent {\bf Technical point.} The representations of $W(F)$
appearing on the left hand side of \eqref{LC n fun} should be
continuous with respect to the Krull topology on $W(F)$. However,
continuous complex representations necessarily factor through a finite
group. In order to obtain a large enough class of such representations
we should consider the so-called $\ell$-adic representations, defined
not over $\C$, but over a finite extension of the field $\Q_\ell$ of
$\ell$-adic numbers, where $\ell$ is a prime that does not divide
$q$. Frobenius eigenvalues will then be $\ell$-adic numbers, and to
match them with the Hecke eigenvalues, which are complex numbers, we
need to choose once and for all an isomorphism of $\C$ and
$\overline\Q_\ell$ as abstract fields (extending the identification of
their subfield $\ol\Q$). But it follows from the theorem of Drinfeld
and Lafforgue that actually they all belong to $\ol\Q$, so this
isomorphism is never used, and the left hand side of \eqref{LC n fun}
is independent of $\ell$. For more on this, see, e.g.,
\cite{F:houches}, Section 2.2.\qed

\subsection{Langlands dual group}   \label{L dual grp}

An important insight of Langlands \cite{L} was that the correspondence
\eqref{LC n fun} may be generalized by replacing the group $GL_n$ on
the right hand side by an arbitrary connected reductive algebraic
group $G$ over the field $F$. This necessitates introducing the
so-called {\em Langlands dual group}.

The simplest case to consider is that of a reductive group $G$ that is
defined over the finite field of scalars $k=\Fq$ and is split over
$k$. This means that it contains a maximal torus that is split
(isomorphic to the direct product of copies of the multiplicative
group) over $k$. In these notes we will consider slightly more general
groups of the following kind: $G$ is a group scheme over $X$ which
contains a Borel subgroup scheme $B$ also defined over $X$, and there
exists an \'etale cover $X' \to X$ such that the pull-back of $G$ to
$X'$ (resp., $B$) is isomorphic to the constant group scheme $X'
\times \mathbb G$, where $\mathbb G$ is a split reductive group over
$k$ (resp., $X' \times \mathbb B$, where $\mathbb B$ is a Borel
subgroup of $\mathbb G$).

For example, let $\mathbb G$ be the multiplicative group ${\mathbb
  G}_m = GL_1$. Let $X' \to X$ be an \'etale double cover of $X$. The
group $\Z_2 = \{ \pm 1 \}$ acts fiberwise on $X'$ and on $\mathbb
G_m$ by the formula $x \to x^{\pm 1}$. Define $H$ as the following
group scheme over $X$:
$$
H = X' \underset{\Z_2}\times \mathbb G_m.
$$
This is an example of a ``twisted torus''. Its pull-back to $X'$ is
isomorphic to $X' \times \mathbb G_m$. We can now talk about the group
over any subscheme of $X$. For example, we have the group $H(F)$ of
sections of $H$ over the ``generic point'' of $X$, $\on{Spec} F$,
the local group for each $x \in |X|$, which is the group of sections
of $H$ over the formal punctured disc $D_x^\times = \on{Spec} F_x$,
its subgroup of sections over the formal disc $D_x = \on{Spec} \OO_x$.
The ad\'elic group is now the restricted product of the local groups
for all $x \in |X|$, etc.

Likewise, in general we have a group scheme
$$
G = X' \underset{\Gamma}\times \mathbb G,
$$
where $\Gamma$ is the group of deck transformations of $X'$, acting
on $\mathbb G$.

\medskip

Let $\mathbb T$ be a maximal split torus in $\mathbb G$. We associate
to it two lattices: the weight lattice $X^*(\mathbb T)$ of
homomorphisms $\mathbb T \to \mathbb G_m$ and the coweight lattice
$X_*(\mathbb T)$ of homomorphisms ${\mathbb G_m} \to \mathbb T$. They
contain the sets of roots $\Delta \subset X^*(\mathbb T)$ and coroots
$\Delta^\vee \subset X_*(\mathbb T)$ of $G$, respectively. The
quadruple $(X^*(\mathbb T),X_*(\mathbb T),\Delta,\Delta^\vee)$ is
called the root data for $\mathbb G$. It determines the split group
$\mathbb G$ up to an isomorphism.

The action of the group $\Gamma$ on $\mathbb G$ gives rise to its
action on the root data.

Let us now exchange the lattices of weights and coweights and the sets
of roots and coroots. Then we obtain the root data
$$
(X_*({\mathbb T}),X^*({\mathbb T}),\Delta^\vee,\Delta)
$$
of another reductive algebraic group over $\C$ (or any other
algebraically closed field, like $\ol{\mathbb Q}_\ell$) which is
denoted by $\check{\mathbb G}$. The action of $\Gamma$ on the root
data gives rise to its action on $\check{\mathbb G}$. We then define
the Langlands dual group of $G$ as
$$
\LG = \Gamma \ltimes \check{\mathbb G}.
$$
Note that $\Gamma$ is a finite quotient of the Galois group
$\on{Gal}(\ol{F}/F)$.

For example, if $G$ is a twisted torus described above, then $\LG =
\Z_2 \ltimes {\mathbb G}_m$, where $\Z_2$ acts on ${\mathbb G}_m$ by
the formula $x \mapsto x^{\pm 1}$.

There is a variant of the above definition in which $\Gamma$ is
replaced by $\on{Gal}(\ol{F}/F)$ (acting on the right factor through
the surjective homomorphism $\on{Gal}(\ol{F}/F) \to \Gamma$) or by the
Weil group $W(F)$. The definition may be generalized to an arbitrary
reductive group $G$ over $F$.

\medskip

Now the conjectural Langlands correspondence \eqref{LC n fun} takes
the following form:

\begin{equation}    \label{LC G fun}
\boxed{\begin{matrix} \text{homomorphisms} \\ W(F) \to
      \LG \end{matrix}}
\quad
    \longleftrightarrow \quad \boxed{\begin{matrix}
    \text{irreducible tempered automorphic} \\ \text{representations
    of } G({\mathbb A}_F) \end{matrix}}
\end{equation}

\bigskip

Note that if $G=GL_n$, then $\LG$ is also $GL_n$, and so the
homomorphisms $W(F) \to \LG$ appearing on the left hand side are the
same as $n$-dimensional representations of $W(F)$.

\medskip

\noindent {\bf Technical point.} We should consider here $\ell$-adic
homomorphisms, as in the case of $GL_n$. Also, to a homomorphism $W(F)
\to \LG$ in general corresponds not a single irreducible automorphic
representation of $G(\AD_F)$, but a set of such representations,
called an $L$-{\em packet}.\qed

\medskip

Under the correspondence \eqref{LC G fun}, the same kind of
compatibility between the Hecke and Frobenius eigenvalues should hold
as in the case of $GL_n$. The key point here (which comes from
Langlands' interpretation \cite{L} of the description of the spherical
Hecke algebra for general reductive groups due to Satake) is that the
Hecke eigenvalues of automorphic representations may be interpreted as
conjugacy classes in the Langlands dual group $\LG$ (see \secref{hecke
  eig} below and \cite{F:houches}, Section 5.2, for more details).

For $G=GL_n$ the Frobenius eigenvalues completely determine $\sigma$
and the Hecke eigenvalues completely determine an irreducible
automorphic representation $\pi$. All automorphic representations
occur with multiplicity 1 in $L_2(G(F) \bs G(\AD_F))$. For general
groups, this is not so. There may be several inequivalent
homomorphisms $\sigma: W(F) \to \LG$ (with the same collection of
Frobenius eigenvalues), all corresponding to the same $\pi$ (or the
same $L$-packet). In this case the multiplicity of $\pi$ is expected
to be greater than 1.

\section{The geometric Langlands correspondence} \label{GLC}

Now we wish to reformulate the Langlands correspondence in such a way
that it would make sense not only for curves defined over a
finite field, but also for curves over the complex field.

Thus, we need to find geometric analogues of the notions of Galois
representations and automorphic representations.

\subsection{$\LG$-bundles with flat connection}

The former is fairly easy. Let $X$ be a curve over a field $k$ and $F
= k(X)$ the field of rational functions on $X$. If $Y \to X$ is a
covering of $X$, then the field $k(Y)$ of rational functions on $Y$ is
an extension of the field $F = k(X)$ of rational functions on $X$, and
the Galois group $\on{Gal}(k(Y)/k(X))$ may be viewed as the group of
``deck transformations'' of the cover. If our cover is {\em
  unramified}, then this group is isomorphic to a quotient of the
(arithmetic) fundamental group of $X$. For a cover ramified at points
$x_1,\ldots,x_n$, it is isomorphic to a quotient of the (arithmetic)
fundamental group of $X \bs \{ x_1,\ldots,x_n \}$. From now on we will
focus on the {\em unramified case}. This means that we replace
$\on{Gal}(\ol{F}/F)$ by its maximal unramified quotient, which is
isomorphic to the (arithmetic) fundamental group of $X$. Its geometric
analogue, for $X$ defined over $\C$, is $\pi_1(X,x)$, with respect to
a reference point $x \in X$.

The choice of a reference point could be a subtle issue in
general. However, since these groups are isomorphic to each other for
different choices of the reference point, we obtain canonical
bijections between the sets of equivalence classes of homomorpisms
from these groups to $\LG$, which is what we are interested in
here. Henceforth we will suppress the reference point in our notation
and write simply $\pi_1(X)$.

Thus, the geometric counterpart of a (unramified) homomorphism
$\on{Gal}(\ol{F}/F) \to {}\LG$ is a homomorphism $\pi_1(X) \to
{}\LG$. If we replace $\on{Gal}(\ol{F}/F)$ by the Weil group $W(F)$,
then we should replace $\pi_1(X)$ by a similarly defined subgroup.

{}Let $X$ be a smooth projective connected algebraic
curve defined over $\C$. Let $G$ be a complex reductive algebraic
group and $\LG$ its Langlands dual group. Then homomorphisms $\pi_1(X)
\to {}\LG$ may be described in differential geometric terms as
(smooth) principal $\LG$-bundles on $X$ with a flat
connection. Indeed, the monodromy of the flat connection gives rise to
a homomorphism $\pi_1(X) \to {}\LG$, and this gives rise to an
equivalence of the appropriate categories and a bijection of the
corresponding sets of equivalence classes.

Let $E$ be a smooth principal $\LG$-bundle on $X$. A flat connection
on $E$ has two components. The $(0,1)$ component, with respect to the
complex structure on $X$, defines holomorphic structure on $E$, and
the $(1,0)$ component defines a holomorphic connection $\nabla$. Thus,
a principal $\LG$-bundle with a flat connection on $X$ is the same as
a pair $(E,\nabla)$, where $E$ is a holomorphic (equivalently,
algebraic) principal $\LG$-bundle on $X$ and $\nabla$ is a holomorphic
(equivalently, algebraic) connection on $E$.

Thus, for complex curves the objects on the left hand side of the
Langlands correspondence \eqref{LC G fun} should be the equivalence
classes of flat (holomorphic or algebraic) $\LG$-bundles $(E,\nabla)$.

\subsection{Sheaves on $\Bun_G$}

We consider next the right hand side of \eqref{LC G fun}. Here the
answer is not quite as obvious. We sketch it briefly referring the
reader to \cite{F:houches}, Section 3, for more details.

Recall that automorphic representations of $G(\AD_F)$ (where $F$ is a
function field of a curve $X$ defined over $\Fq$) are realized in
functions on the quotient $G(F) \bs G(\AD_F)$. Let us restrict
ourselves to those irreducible automorphic representations that
correspond to unramified homomorphisms $W(F) \to {}\LG$. It is known
that they contain a one-dimensional subspace stable under the subgroup
$G(\OO_F) \subset G(\AD_F)$, where
$$
\OO_F = \prod_{x \in |X|} \OO_x.
$$
These representation are also called {\em unramified}. Any vector in
the $G(\OO_F)$-stable line in such a
representation $\pi$ gives rise to a function on the double
quotient
\begin{equation}    \label{double quotient}
G(F) \bs G(\AD_F)/G(\OO_F).
\end{equation}
This function, which is called the {\em spherical function}, contain
all information about $\pi$, because the right translates by $g \in
G(\AD_F)$ of this function pulled back to $G(F)\backslash G(\AD_F)$
span $\pi$.

Now, a key observation, due to Andr\'e Weil, is that in the case of
$G=GL_n$ this double quotient is precisely the set of isomorphism
classes of rank $n$ bundles on our curve $X$. This statement is true
if the curve $X$ is defined over a finite field or the complex field.

For a general reductive group $G$ this double quotient is the set of
isomorphism classes of principal $G$-bundles on $X$ if $X$ is over
$\C$. This is true in Zariski, \'etale, or analytic topology.

If $X$ is defined over a finite field, the situation is more
subtle.\footnote{I thank Yevsey Nisnevich for a discussion of
  this issue.} Then the double quotient \eqref{double quotient} is the
set of equivalence classes of principal $G$-bundles in Zariski
topology as well as Nisnevich topology \cite{Nis1,Nis2}. In the
\'etale topology, this is true only if the group
$$
\on{Ker}^1(F,G) = \on{Ker}(H^1(F,G) \to \prod_{x\in |X|} H^1(F_x,G))
$$
is trivial. In this case, it is sometimes said that $G$ ``satisfies the
Hasse principle''. This holds, for example, in the case that $G$ is
semi-simple and split over $\Fq$, see \cite{Har} and \cite{BeDh},
Corollary 4.2. Otherwise, the set of equivalence classes of principal
$G$-bundles in the \'etale topology (equivalently, the {\em fppf}
topology) is a union over $\xi \in \on{Ker}^1(F,G)$ of double
quotients like \eqref{double quotient} in which $G(F)$ is replaced by
its form corresponding to $\xi$.

\medskip

{}From now on we will assume for simplicity that the Hasse principle
holds for $G$. Then the geometric analogues of unramified automorphic
representations should be some geometric objects that ``live'' on some
kind of moduli space of principal $G$-bundles on $X$.

If $G=GL_1$, the Picard variety is an algebraic variety that serves as
the moduli space of principal $G$-bundles on $X$, which are the same
as line bundles on $X$ in this case.

Unfortunately, for a non-abelian group $G$ there is no algebraic
variety whose set of $k$-points is the set of isomorphism classes of
principal $G$-bundles on $X$ (where $k$ is the field of definition of
$X$). The reason is that $G$-bundles have groups of automorphisms,
which vary from bundle to bundle (in the case of $GL_1$-bundles, the
group of automorphisms is the same for all bundles; it is the
multiplicative group acting by rescalings). However, there is an {\em
  algebraic stack} that parametrizes principal $G$-bundles on $X$. It
is denoted by $\Bun_G$. It is not an algebraic variety, but it looks
locally like the quotient of an algebraic variety by the action of an
algebraic group. These actions are not free, and therefore the
quotient is no longer an algebraic variety. However, the structure of
the quotient allows us to define familiar objects on it. For instance,
a coherent sheaf on the quotient stack $Y/H$ of this kind is nothing
but an $H$-equivariant coherent sheaf on $Y$. It turns out that this
is good enough for our purposes.

In the classical story, when $X$ is defined over $\Fq$, an unramified
automorphic representation may be replaced by a non-zero spherical
function (which is unique up to a scalar) on the above double quotient
which is the set of $\Fq$-points of $\Bun_G$. Hence in the geometric
theory we need to find geometric analogues of these functions.

According to Grothen\-dieck's philosophy, the ``correct'' geometric
counterpart of the notion of a function on the set of $\Fq$-points of
$V$ is the notion of an {\em $\ell$-adic sheaf} on $V$. We will not
attempt to give a precise definition here, referring the reader to
\cite{Milne,Weil}. Let us just say that the simplest example of an
$\ell$-adic sheaf is an $\ell$-adic local system, \index{local
  system!$\ell$-adic} which is, roughly speaking, a compatible system
of locally constant $\Z/\ell^n \Z$-sheaves on $V$ for $n\geq 1$ (in
the \'etale topology).

The important property of the notion of an $\ell$-adic sheaf $\F$ on
$V$ is that for any morphism $f: V' \to V$ from another variety $V'$
to $V$ the group of symmetries of this morphism will act on the
pull-back of $\F$ to $V'$. In particular, let $x$ be an $\Fq$-point of
$V$ and $\ol{x}$ the $\ol{{\mathbb F}}_q$-point corresponding to an
inclusion $\Fq \hookrightarrow \ol{{\mathbb F}}_q$. Then the pull-back
of $\F$ with respect to the composition $\ol{x} \to x \to V$ is a
sheaf on $\ol{x}$, which is nothing but the fiber $\F_{\ol{x}}$ of
$\F$ at $\ol{x}$, which is a $\ol\Q_\ell$-vector space. But the Galois
group $\on{Gal}(\ol{{\mathbb F}}_q/\Fq)$ is the symmetry of the map
$\ol{x} \to x$, and therefore it acts on $\F_{\ol{x}}$.

Let $\on{Fr}_{\ol{x}}$ be the (geometric) Frobenius element, which is
the inverse of the automorphism $y \mapsto y^q$ of $\ol{{\mathbb
    F}}_q$. It is a generator of $\on{Gal}(\ol{{\mathbb F}}_q/\Fq)$
and hence acts on $\F_{\ol{x}}$. Taking the trace of
$\on{Fr}_{\ol{x}}$ on $\F_{\ol{x}}$, we obtain a number
$\on{Tr}(\on{Fr}_{\ol{x}},{\mc F}_{\ol{x}})$, which we will also
denote by $\on{Tr}(\on{Fr}_{x},{\mc F}_{x})$.

Hence we obtain a function on the set of ${\mathbb F}_q$-points of
$V$. One assigns similarly a function to a complex of $\ell$-adic
sheaves, by taking the alternating sums of the traces of
$\on{Fr}_{\ol{x}}$ on the stalk cohomologies of $\K$ at $\ol{x}$.  The
resulting map intertwines the natural operations on complexes of
sheaves with natural operations on functions (see \cite{Laumon:const},
Section 1.2). For example, pull-back of a sheaf corresponds to the
pull-back of a function, and push-forward of a sheaf with compact
support corresponds to the fiberwise integration of a function.

Thus, because of the existence of the Frobenius automorphism in the
Galois group $\on{Gal}(\ol{{\mathbb F}}_q/{\mathbb F}_{q})$ (which is
the group of symmetries of an $\Fq$-point) we can pass from
$\ell$-adic sheaves to functions on any algebraic variety over $\Fq$.
This suggests that the proper geometrization of the notion of a
function in this setting is the notion of $\ell$-adic sheaf.

The naive abelian category of $\ell$-adic sheaves is not a good choice
for various reasons; for instance, it is not stable under the Verdier
duality. The correct choice turns out to be another abelian category
of the so-called {\em perverse sheaves}. These are actually complexes
of $\ell$-adic sheaves on $V$ satisfying certain special
properties. Examples are $\ell$-adic local systems on a smooth variety
$V$, placed in cohomological degree equal to $-\dim V$. General
perverse sheaves are ``glued'' from such local systems defined on the
strata of a particular stratification of $V$ by locally closed
subvarieties.

Experience shows that many ``interesting'' functions on the set
$V({\mathbb F}_{q})$ of points of an algebraic variety $V$ over $\Fq$
come from perverse sheaves ${\mc K}$ on $V$. Hence it is natural to
expect that unramified automorphic functions on $$G(F)\bs
G(\AD_F)/G(\OO_F),$$ which is the set of $\Fq$-points of
$\Bun_G$, come from perverse sheaves on $\Bun_G$.

The concept of perverse sheaf makes perfect sense for varieties over
$\C$ as well, and this allows us to formulate the geometric Langlands
conjecture when $X$ (and hence $\on{Bun}_n$) is defined over $\C$. And
over the field of complex numbers there is one more reformulation that
we can make; namely, we can pass from perverse sheaves to $\D$-{\em
  modules}.

Recall (see, e.g., \cite{KS,GM}) that a ${\mc D}$-module on
a smooth algebraic variety $Z$ is a sheaf of modules over the sheaf
${\mc D}_Z$ of differential operators on $Z$. An example of a ${\mc
D}$-module is the sheaf of sections of a flat vector bundle on
$Z$. The sheaf of functions on $Z$ acts on sections by multiplication,
so it is an $\OO_Z$-module. But the flat connection also allows us to
act on sections by vector fields on $Z$. This gives rise to an action
of the sheaf ${\mc D}_Z$, because it is generated by vector fields and
functions. Thus, we obtain the structure of a ${\mc D}$-module.

In our case, $\Bun_G$ is not a variety, but an algebraic stack.  The
suitable (derived) category of ${\mc D}$-modules on it has been
defined in \cite{BD}.

${\mc D}$-modules on $\Bun_G$ will be the objects that we will
consider as the replacements for the unramified spherical functions in
the complex case.

\subsection{Hecke functors: examples}    \label{hecke functors}

There is more: an unramified spherical function attached to an
unramified automorphic representation has a special property; it is an
eigenfunction of the Hecke operators. These are integral operators
that are cousins of the classical Hecke operators one studies in the
theory of modular forms. The eigenvalues of these operators are
precisely what we referred to earlier as ``Hecke eigenvalues''. For a
general automorphic representation, these are defined for all but
finitely many closed points of $X$. But for the unramified automorphic
representations they are defined for all points. In this case the
Hecke operators may be defined as integral operators acting on the
space of functions on the set of $\Fq$-points of $\Bun_G$, if the
curve $X$ is defined over $\Fq$.

The ${\mc D}$-modules on $\Bun_G$ we are looking for, in the case that
$X$ is defined over $\C$, should reflect this Hecke property.

The analogues of the Hecke operators are now the so-called {\em Hecke
  functors} acting on the derived category of ${\mc D}$-modules on
$\Bun_G$. They are labeled by pairs $(x,V)$, where $x \in X$ and $V$
is a finite-dimensional representation of the dual group $\LG$, and
are defined using certain modifications of $G$-bundles.

Before giving the general definition, consider two examples. First,
consider the abelian case with $G = GL_1$ (thus, we have $G(\C) =
\C^\times$). In this case $\Bun_G$ may be replaced by the Picard
variety $\on{Pic}$ which parametrizes line bundles on $X$. Given a
point $x \in X$, consider the map $p'_x: \on{Pic} \to \on{Pic}$
sending a line bundle ${\mc L}$ to ${\mc L}(x)$ (the line bundle whose
sections are sections of ${\mc L}$ which are allowed to have a pole of
order $1$ at $x$). By definition, the Hecke functor ${\mathbb
  H}_{1,x}$ corresponding to $x$ and $1 \in \Z$ (which we identify
with the set of one-dimensional representations of $\LG = GL_1$) is
given by the formula
$$
{\mathbb H}_{1,x}({\mc F}) = p'_x{}^*({\mc F}).
$$

Next, consider the case of $G=GL_n$ and $V=V_{\check\omega_1}$, the
defining $n$-dimensional representation of $\LG = GL_n$. In this case
$\Bun_{GL_n}$ is the moduli stack $\Bun_n$ of rank $n$ bundles on
$X$. There is an obvious analogue of the map sending a rank $n$ bundle
${\mc M}$ to ${\mc M}(x)$. But then the degree of the bundle jumps by
$n$. It is possible to increase it by $1$, but we need to choose a
line $\ell$ in the fiber of ${\mc M}$ at $x$. We then define a new
rank $n$ bundle ${\mc M}'$ by saying that its sections are the
sections of ${\mc M}$ having a pole of order $1$ at $x$, but the polar
part has to belong to $\ell$. Then $\on{deg} {\mc M}' = \on{deg} {\mc
  M} + 1$. However, we now have a ${\mathbb P}^{n-1}$ worth of
modifications of ${\mc M}$ corresponding to different choices of the
line $\ell$. The Hecke functor ${\mathbb H}_{V_{\check\omega_1,x}}$ is
obtained by ``integrating'' over all of them.

More precisely, let ${\mc H}_{\check\omega_1,x}$ be the moduli
stack of pairs $(\M,\M')$ as above. It defines a correspondence over
$\on{Bun}_n \times \on{Bun}_n$:
\begin{equation}    \label{Hecke cor}
\begin{array}{ccccc}
& & {\mc H}_{\check\omega_1,x} & & \\
& \stackrel{p_x}\swarrow & & \stackrel{p'_x}\searrow & \\
\Bun_n & & & & \Bun_n
\end{array}
\end{equation}
By definition,
\begin{equation} \label{formula H1}
{\mathbb H}_{\check\omega_1,x}({\mc
    F}) = p_{x*} \; p'_x{}^*({\mc F}).
\end{equation}

\subsection{Hecke functors: general definition}   \label{gen hecke}

For irreducible representations $\rho_{\mu}$ of $\LG$ with general
dominant integral highest weights $\mu$ there is an analogous
correspondence in which the role of the projective space ${\mathbb
  P}^{n-1}$ is played by the Schubert variety in the affine
Grassmannian of $G$ corresponding to $\mu$.

We explain this in the split case (so that $G=\GG$ and
$\LG=\check\GG$). First, observe that if we have two $G$-bundles
$E,E'$ on the (formal) disc $\on{Spec} k[[t]]$ which are identified
over the punctured disc $\on{Spec} k\ppart$, we obtain a point in the
double quotient
$$G[[t]]\bs G\ppart/G[[t]],$$ or, equivalently, a $G[[t]]$-orbit
in the affine Grassmannian $$\Gr = G\ppart/G[[t]]$$ which is an
ind-scheme over $k$ \cite{BD,MV}. These orbits are called Schubert
cells, and they are labeled by elements $\mu$ the set $\Xb_+$ of
dominant weights of the maximal torus in the dual group $\LG$. We
denote the orbit corresponding to $\mu$ by $\Gr_\mu$. We will write
$\inv(E,E')=\mu$ if the pair $(E,E')$ belongs to $\Gr_\mu$. Note
that $\Gr_{\mu'}$ is contained in the closure $\ovl\Gr_\mu$ of
$\Gr_\mu$ if and only if $\mu\geq\mu'$.

Following Beilinson and Drinfeld \cite{BD}, introduce the {\em Hecke
  stack} ${\mc H} = \Hc^\BD(X,G)$ that classifies quadruples
$$(x,E,E',\phi),$$ where $x\in X$, $E,E'\in \Bun_{G}$ and $\phi$ is
an isomorphism
$$
E|_{X{-}\{x\}}\simeq E'|_{X{-}\{x\}}.
$$
We have two natural morphisms $p, p': {\mc H}^{\BD} \to \Bun_{G}$
sending such a quadruple to $E$ or $E'$ and the morphism $s: {\mc H}
\to X$. Since $\Bun_{G}$ is an algebraic stack, so is
$\Hc^\BD(X,G)$. However, if we fix $E'$, then we obtain an
ind-scheme over $X$, which is called the {\em Beilinson-Drinfeld
  Grassmannian} (see \cite{BD,MV}).

Let ${\mc H}^{\BD}{}'(X,G)$ be the stack classifying the quadruples
$$(x,E,E',\phi),$$ where $x\in X$, $E \in \Bun_{G}$, $E'$ is a
$G$-bundle on the disc $D_x$ around the point $x$, and $\phi$ is
an isomorphism
$$
E|_{D_x^\times}\simeq E'|_{D_x^\times},
$$
where $D_x^\times$ is the punctured disc around $x$. We have a
natural morphism $${\mc H}^{\BD}(X,G) \to {\mc H}^{\BD}{}'(X,G)$$
(restricting $E'$ to $D_x$ and $\phi$ to $D_x^\times$), which is in
fact an isomorphism, according to a strong version of a theorem of
Beauville--Laszlo \cite{BL} given in \cite{BD}, Section 2.3.7. Therefore
we obtain that a morphism $$s \times p: {\mc H}^{\BD}(X,G) \to X
\times \Bun_{G}$$ sending the above quadruple to $(x,E)$ is a
locally trivial fibration with fibers isomorphic to the affine
Grassmannian $\Gr=G\ppart/G[[t]]$.

For every dominant integral weight $\mu \in \Xb_+$ we define the
closed substack $\Hc_\mu$ of ${\mc H}^{\BD}(X,G)$ by imposing the
inequality
\begin{equation} \label{mu}
\inv_x(E,E')\leq \mu.
\end{equation}
It is a scheme over $X\times \Bun_{G}$ with fibers isomorphic
to $\ovl\Gr_\mu$.

\medskip

Recall the {\em geometric Satake correspondence} \cite{MV}, which is
an equivalence of tensor categories between the category of
finite-dimensional representations of $\LG$ and the category of
$G[[t]]$-equivariant perverse sheaves on $\Gr$ (see \cite{F:houches},
Sects. 5.4--5.6 for an exposition). It sends the irreducible
finite-dimensional representation $\rho=\rho_\mu$ of $\LG$ to the
irreducible perverse sheaf $\IC(\ol\Gr_\mu)$ supported on
$\ol\Gr_\mu$.

Let ${\mc K}_\rho$ by the perverse sheaf on ${\mc H}^{\BD}(X,G)$,
supported on $\Hc_\mu$, which is constant along $X\times \Bun_G$ with
the fibers isomorphic to $\IC(\ol\Gr_\mu)$.

We now define the {\em Hecke functor} ${\mathbb H}_\rho={\mathbb
  H}_\mu$ as the integral transform corresponding to the kernel ${\mc
  K}_\rho$ (see \cite{BD}):
$$
{\mathbb H}_\rho({\mc F}) = (s \times p)_*(p'{}^*({\mc F}) \otimes
{\mc K}_\rho).
$$

For $x \in |X|$, let $\Hc_x$ be the fiber of $\Hc$ over $x$, and
$p_x,p'_x: \Hc_x \to \Bun_G$ the corresponding morphisms. Denote by
${\mc K}_{\rho,x}$ the restriction of ${\mc
K}_\rho$ to $\Hc_x$. Define the functor ${\mathbb H}_{\rho,x}$ by the
formula
$$
{\mathbb H}_{\rho,x}({\mc F}) = p_{x*}(p'_x{}^*({\mc F}) \otimes {\mc
  K}_{\rho,x}).
$$

If $X$ is defined over $\Fq$, then one can show that the function
corresponding to the sheaf ${\mc K}_{\rho,x}$ via the Grothendieck
dictionary is the kernel $K_{\rho,x}$ of the Hecke operator
corresponding to $\rho$ and $x$ (see \cite{F:houches},
Section 5.4). Therefore the functor ${\mathbb H}_{\rho,x}$ is a
geometric analogue of the Hecke operator ${\mb H}_{\rho,x}$.

\subsection{Hecke eigensheaves}    \label{hecke eig}

Let ${\mc E}$ be a flat $\LG$-bundle on $X$. Then
$$
\rho_{\mc E} = {\mc E} \underset{\LG}\times \rho
$$
is a flat vector bundle on $X$, hence a ${\mc D}$-module on $X$. The
following definition is due to \cite{BD}.

\begin{definition}
  A sheaf ${\mc F}$ on $\Bun_G$ is called a {\em Hecke eigensheaf with
    the eigenvalue ${\mc E}$} if for any representation $\rho$ of
  $\LG$ we have an isomorphism
\begin{equation}    \label{iso}
{\mathbb H}_\rho({\mc F}) \simeq  \rho_{\mc E} \boxtimes
{\mc F},
\end{equation}
and these isomorphisms are compatible for different $\rho$ with
respect to the structures of tensor categories on both sides.
\end{definition}

\medskip

By base change, it follows from the above identity that for every
$x\in X$, we have an isomorphism
\begin{equation}    \label{isox}
{\mathbb H}_{\rho,x}({\mc F}) \simeq \rho \otimes
{\mc F}.
\end{equation}

If our curve $X$ is defined over $\Fq$, we can pass from a Hecke
eigensheaf ${\mc F}$ on $\Bun_G$ to a function $f$ on $G(F) \bs
G(\AD_F)/G(\OO_F)$. Then this function will be an eigenfunction of the
Hecke operators ${\mb H}_{\rho,x}$:
\begin{equation}    \label{Hecke ef}
{\mb H}_{\rho,x}(f) = h_{\rho,x} f,
\end{equation}
According to the {\em Satake isomorphism} (see, e.g.,
\cite{F:houches}, Section 5.4), the map $$[\rho] \mapsto {\mb
  H}_{\rho,x},$$ where $\rho$ runs over all finite-dimensional
representations of $\LG$, defines an isomorphism between the
representation ring $\Rep \LG$ of $\LG$ and the spherical Hecke
algebra generated by ${\mb H}_{\rho,x}$. Hence the collection of
eigenvalues $\{ h_{\rho,x} \}$ defines a point in the spectrum of
$\on{Rep} \LG$, that is, a semi-simple conjugacy class $\nu_x$ in
$\LG$.

If the (unramified) automorphic representation $\pi$ generated by the
spherical function $f$ corresponds to a homomorphism $\sigma: W(F) \to
\LG$ under the Langlands correspondence \eqref{LC G fun}, then $\nu_x
= \sigma(\on{Fr}_x)$, where $\on{Fr}_x$ is the Frobenius conjugacy
class associated to $x$ (see \cite{F:houches}, Section 2.2). In other
words, in this case
\begin{equation}    \label{hx}
h_{\rho,x} = \on{Tr}(\sigma(\on{Fr}_x),\rho)
\end{equation}
(up to a power of $q$). A general automorphic representation $\pi$
would be ramified at finitely many points $x \in |X|$. Then this
condition would only be satisfied away from those points. This is the
precise meaning of the ``matching'' between the Hecke and Frobenius
eigenvalues that we mentioned above.

\subsection{Geometric Langlands correspondence}

Now we can state the geometric Langlands correspondence as the
following diagram:

\begin{equation}    \label{LC}
\boxed{\begin{matrix} \text{flat} \\
    \LG\text{-bundles on } X \end{matrix}} \quad
    \longrightarrow \quad \boxed{\begin{matrix}
    \text{Hecke eigensheaves} \\ \text{on } \Bun_G \end{matrix}}
\end{equation}

$$
{\mc E} \quad \longrightarrow \quad {\mc F}_{\mc E}.
$$

This correspondence has been constructed in many cases. For $G=GL_n$
the Hecke eigensheaves corresponding to irreducible ${\mc E}$ have
been constructed in \cite{FGV,Ga}, building on the work of P. Deligne
for $n=1$ (explained in \cite{Laumon:duke} and \cite{F:houches}),
V. Drinfeld \cite{Dr1} for $n=2$, and G. Laumon \cite{Laumon:duke}
(this construction works for curves defined both over $\Fq$ or $\C$).

For any split simple algebraic group $G$ and $X$ defined over $\C$,
the Hecke eigensheaves have been constructed in a different way by
A. Beilinson and V. Drinfeld \cite{BD} in the case that ${\mc E}$ has
an additional structure of an {\em oper} (this means that ${\mc E}$
belongs to a certain half-dimensional locus in $\Loc_{\LG}$). It is
interesting that this construction is closely related to the 2D
Conformal Field Theory and representation theory of affine Kac--Moody
algebras of critical level. For more on this, see Part III of
\cite{F:houches}.

\subsection{Categorical version}    \label{cat ver}

Looking at the correspondence \eqref{LC}, we notice that there is an
essential asymmetry between the two sides. On the left we have flat
$\LG$-bundles, which are points of a moduli stack $\Loc_{\LG}$ of
flat $\LG$-bundles (or local systems) on $X$. But on the right we
have Hecke eigensheaves, which are objects of a category; namely, the
category of ${\mc D}$-modules on $\Bun_G$. Beilinson and Drinfeld have
suggested how to formulate it in a more symmetrical way.

The idea is to replace a point ${\mc E} \in \Loc_{\LG}$ by an object
of another category; namely, the skyscraper sheaf ${\mc O}_{\mc E}$ at
${\mc E}$ viewed as an object of the category of coherent ${\mc
O}$-modules on $\Loc_{\LG}$. A much stronger, categorical, version of
the geometric Langlands correspondence is then a conjectural
equivalence of derived categories\footnote{\label{foot} It is expected
(see \cite{FW}, Section 10) that there is in fact a $\Z_2$-gerbe of such
equivalences. This gerbe is trivial, but not canonically
trivialized. One gets a particular trivialization of this gerbe, and
hence a particular equivalence, for each choice of the square root of
the canonical line bundle $K_X$ on $X$.}

\begin{equation}    \label{na fm}
\boxed{\begin{matrix} \text{derived category of} \\
    \OO\text{-modules on } \on{Loc}_{\LG} \end{matrix}} \quad
    \longleftrightarrow \quad \boxed{\begin{matrix}
    \text{derived category of} \\ \D\text{-modules on } \Bun_G
    \end{matrix}}
\end{equation}

\bigskip

This equivalence should send the skyscraper sheaf ${\mc O}_{\mc E}$ on
$\on{Loc}_{\LG}$ supported at ${\mc E}$ to the Hecke eigensheaf ${\mc
F}_E$. If this were true, it would mean that Hecke eigensheaves
provide a good ``basis'' in the category of $\D$-modules on $\Bun_G$,
so we would obtain a kind of spectral decomposition of the derived
category of $\D$-modules on $\Bun_G$, like in the Fourier
transform. (Recall that under the Fourier transform on the real line
the delta-functions $\delta_x$, analogues of ${\mc O}_{\mc E}$, go to
the exponential functions $e^{itx}$, analogues of ${\mc F}_{\mc E}$.)

This equivalence has been proved by G. Laumon \cite{Laumon} and
M. Rothstein \cite{Rothstein} in the abelian case, when $G=GL_1$ (or a
more general torus). They showed that in this case this is nothing but
a version of the Fourier--Mukai transform. Thus, the categorical
Langlands correspondence may be viewed as a kind of non-abelian
Fourier--Mukai transform (see \cite{F:houches}, Section 4.4).

In the non-abelian case, this has not yet been made into a precise
conjecture in the literature.\footnote{After I presented these
  Colloquium Lectures, the paper \cite{AG} appeared in which a precise
  formulation of this conjecture was proposed.} Nevertheless, the
diagram \eqref{na fm} gives us a valuable guiding principle to the
geometric Langlands correspondence. In particular, it gives us a
natural explanation as to why the skyscraper sheaves on $\Loc_{\LG}$
should correspond to Hecke eigensheaves.

The point is that on the category of ${\mc O}$-modules on $\Loc_{\LG}$
we also have a collection of functors ${\mathbb W}_{V}$, parametrized by the
same data as the Hecke functors ${\mathbb H}_{V}$. We will call them
the {\em Wilson functors} (because of the close connection between
them and the Wilson line operators in 4D gauge theory). These functors
act from the category of $\OO$-modules on $\on{Loc}_{\LG}$ to the
category of sheaves on $X \times \on{Loc}_{\LG}$, which are ${\mc
  D}$-modules along $X$ and $\OO$-modules along $\on{Loc}_{\LG}$.

To define them, observe that we have a tautological $\LG$-bundle
${\mc T}$ on $X \times \on{Loc}_{\LG}$, whose restriction to $X
\times {\mc E}$, where ${\mc E} = (E,\nabla)$, is $E$. Moreover,
$\nabla$ gives us a partial connection on ${\mc T}$ along $X$. For a
representation $V$ of $\LG$, let ${\mc T}_V$ be the associated vector
bundle on $X \times \on{Loc}_{\LG}$, with a connection along $X$.

Let $p_2: X \times \on{Loc}_{\LG} \to \on{Loc}_{\LG}$ be the projection
onto the second factor. By definition,
\begin{equation}    \label{wilson}
{\mathbb W}_V({\mc F}) = {\mc T}_V \otimes p_2^*({\mc F})
\end{equation}
(note that by construction ${\mc T}_V$ carries a connection along $X$
and so the right hand side really is a ${\mc D}$-module along $X$).

Now, the conjectural equivalence \eqref{na fm} should be compatible
with the Wilson/Hecke functors in the sense that
\begin{equation}    \label{WH}
C({\mathbb W}_V) \simeq {\mathbb H}_V, \qquad V
\in \on{Rep} {}\LG,
\end{equation}
where $C$ denotes this equivalence (from left to right).

In particular, observe that the skyscraper sheaf $\OO_{\mc E}$ at
${\mc E} \in \on{Loc}_{\LG}$ is obviously an eigensheaf of the Wilson
functors:
$$
{\mathbb W}_V(\OO_{\mc E}) = V_{\mc E} \boxtimes \OO_{\mc E}.
$$
Indeed, tensoring a skyscraper sheaf with a vector bundle is
the same as tensoring it with the fiber of this vector bundle at the
point of support of this skyscraper sheaf. Therefore \eqref{WH}
implies that ${\mc F}_{\mc E} = C(\OO_{\mc E})$ must satisfy the
Hecke property \eqref{iso}. In other words, ${\mc F}_{\mc E}$ should
be a Hecke eigensheaf on $\Bun_G$ with eigenvalue ${\mc E}$. Thus, we
obtain a natural explanation of the Hecke property of ${\mc F}_{\mc
  E}$: it follows from the compatibility of the categorical Langlands
correspondence \eqref{na fm} with the Wilson/Hecke functors.

Thus, the conjectural equivalence \eqref{na fm} gives us a natural and
convenient framework for the geometric Langlands correspondence.

The equivalence \eqref{na fm} also arises in the study of S-duality
of the maximally 4D supersymmetric gauge theories with the gauge
groups being the compact forms of $G$ and $\LG$. As shown by Kapustin
and Witten \cite{KW} (see \cite{F:bourbaki} for an exposition), the
S-duality of boundary conditions in these theories yields an
equivalence closely related to \eqref{na fm}, in which the category of
${\mc D}$-modules on $\Bun_G$ is replaced by the category of
$A$-branes on the cotangent bundle of $\Bun_G$.

\section{Langlands Functoriality and Trace Formula}    \label{FTF}

The Langlands correspondence \eqref{LC G fun} is subtle, because it
involves objects from two different worlds: automorphic
representations and Galois representations. However, there is a
closely related correspondence that may be formulated entirely in the
world of automorphic representations.

\subsection{The Langlands Functoriality Principle}    \label{LFC}

Let $G$ and $H$ be two reductive algebraic groups over the function
field $F$ of smooth projective curve $X$ over a finite field, and
assume that $G$ is quasi-split (that is, contains a Borel subgroup
defined over $F$). Let $\LG$ and $\LH$ be their Langlands dual
groups and
$$
a: \LH \to \LG
$$
a homomorphism between them that induces the identity on the Galois
parts. Such homomorphisms are called {\em admissible}.

Given a homomorphism $\sigma: W(F) \to \LH$, we obtain a homomorphism
$a \circ \sigma: W(F) \to \LG$, and hence a natural map of sets of
equivalence classes

\medskip

$$
\begin{matrix}
\boxed{\begin{matrix} \text{homomorphisms} \\ W(F) \to
  \LH \end{matrix}}
\\ \\
    \downarrow \\ \\
\boxed{\begin{matrix} \text{homomorphisms} \\ W(F) \to
    \LG \end{matrix}}
\end{matrix}
$$

\medskip

Taking into account the Langlands correspondence \eqref{LC G fun} for
the group $G$ and the one for $H$, we conclude that to each $L$-packet
of tempered automorphic representations of $H(\AD_F)$ should
correspond an $L$-packet of automorphic representations of
$G(\AD_F)$. In fact, there should be a map

\medskip

\begin{equation}    \label{funct pr}
\begin{matrix}
\boxed{\begin{matrix}
\text{{\em L}-packets of automorphic} \\ \text{representations
    of } H({\mathbb A}_F) \end{matrix}} \\ \\
    \downarrow \\ \\
\boxed{\begin{matrix}
\text{{\em L}-packets of automorphic} \\ \text{representations
    of } G({\mathbb A}_F) \end{matrix}}
\end{matrix}
\end{equation}

\medskip

It is called the {\em Langlands Functoriality} or {\em transfer} of
automorphic representations.

The existence of such a map is non-trivial and surprising, because
even though we have a homomorphism of dual groups $a: \LH \to \LG$,
there is {\em a priori} no connection between the groups $G$ and $H$.

\medskip

These transfers should have the following associativity property: if
$M$ is another reductive group and we have a chain of homomorphisms:
$$
{{}^L\negthinspace M} \to \LH \to \LG,
$$
then the two transfers of automorphic representations from $M(\AD_F)$
to $G(\AD_F)$, one obtained as the composition of the transfers from
$M(\AD_F)$ to $H(\AD_F)$ and from $H(\AD_F)$ to $G(\AD_F)$, and the
other obtained directly from the composition ${{}^L\negthinspace M}
\to \LG$, should coincide.

In addition, we require that under the transfer the Hecke
eigenvalues of automorphic representations should match, in
the following sense. Recall from \secref{hecke eig} that the
eigenvalues of the Hecke operators assign to an automorphic
representation of $G(\AD_F)$ a collection of conjugacy classes
$(\nu_x)$, for all but finitely many $x \in |X|$, in the Langlands
dual group $\LG$. It is known that for all automorphic representations
$\pi_j$ of $G(\AD_F)$ that belong to a given $L$-packet $\{ \pi_j \}$
these conjugacy classes are the same.

Now, let $\{ \pi_i^H \}$ be an $L$-packet of automorphic
representations of $H(\AD_F)$ and $(\nu_x^H)$ the collection of
$\LH$-conjugacy classes assigned to it. Let $\{ \pi_j^G \}$ be the
$L$-packet of automorphic representations which is the transfer of $\{
\pi_i^H \}$ under \eqref{funct pr}, and $(\nu_x^G)$ the collection of
$\LG$-conjugacy classes assigned to it. Then they should be defined
for the same $x \in |X|$ as the $\nu_x^H$ and for each of these $x \in
|X|$ we must have
$$
\nu_x^G = a(\nu_x^H).
$$

Functoriality has been established in some cases, but is unknown in
general (see \cite{Arthur:funct} for a survey).

\subsection{Geometric Functoriality}

Given a homomorphism $\LH \to \LG$, we obtain a natural morphism of
algebraic stacks $\Loc_{\LH} \to \Loc{\LG}$ and hence a natural
functor (direct image) from the category of $\OO$-modules on
$\Loc_{\LH}$ to that on $\Loc_{\LG}$. Hence, in view of the
categorical Langlands correspondence \eqref{na fm}, we should also
have a functor from the derived category of ${\mc D}$-modules on
$\Bun_H$ to that on $\Bun_G$, making the following diagram
commutative:

\medskip

$$
\begin{CD}
\boxed{\OO\text{-modules on } \Loc_{\LH}} @>>> \boxed{{\mc
    D}\text{-modules on } \Bun_H} \\
@VVV @VVV \\
\boxed{\OO\text{-modules on } \Loc_{\LG}} @>>> \boxed{{\mc
    D}\text{-modules on } \Bun_G}
\end{CD}
$$

\medskip

The right vertical arrow is the ``geometric functoriality''
functor. Examples have been constructed in \cite{Lys1,Lys2} (see also
\cite{LL}) using a geometric version of the theta-correspon\-dence. In
\cite{FW} the geometric functoriality for endoscopic groups was
analyzed using the Mirror Symmetry approach to the geometric Langlands
correspondence (in this setting the category of ${\mc D}$-modules on
$\Bun_G$ is replaced by the category of $A$-branes on the cotangent
bundle of $\Bun_G$).

\subsection{Non-tempered representations}    \label{arthur funct}

In \cite{L:BE,L:PR} Langlands proposed a strategy for proving the
Functoriality Conjecture which is based on the use of the {\em trace
  formula}. This was further developed in \cite{FLN} and
\cite{L:ST}. In order to explain this, we need to discuss first the
non-tempered representations and a modified version of the
functoriality transfer. Then, in the next subsection, we will
introduce the trace formula, and in Sections \ref{str}--\ref{decomp
  trace} explain the strategy of \cite{FLN} to use the trace formula
to prove the Functoriality Conjecture.

Recall that according to the Langlands correspondence (formula
\eqref{LC G fun}), the $L$-packets of tempered automorphic
representations are supposed to be parametrized by the equivalence
classes of homomorphisms $W(F) \to \LG$. According to the conjectures
of Arthur, for general automorphic representations of $G(\AD_F)$ the
homomorphisms $W(F) \to \LG$ should be replaced by the {\em Arthur
  parameters}: equivalence classes of homomorphisms
\begin{equation}    \label{Arthur parameter}
\sigma: \SL_2 \times W(F) \to {}\LG
\end{equation}
that induce the canonical map $W(F) \to \on{Gal}(\ol{F}/F) \to
\Gamma$.

If an irreducible automorphic representation $\pi$ is tempered, then
$\sigma|_{\SL_2}$ is trivial. Those $\sigma$ for which
$\sigma|_{\SL_2}$ is non-trivial, correspond to non-tempered
automorphic representations.

If an automorphic representation $$\pi = \bigotimes_{x \in |X|}{}' \;
\pi_x$$ of the adelic group $G(\AD_F)$ has the Arthur parameter
$\sigma$, then for all closed points $x \in |X|$ where $\pi_x$ is
unramified the restriction $\sigma|_{W(F)}$ is also unramified and the
conjugacy class $\nu_x$ in $\LG$ corresponding to $\pi_x$ is the
conjugacy class of
\begin{equation}    \label{matching arthur}
\sigma\left(\begin{pmatrix} q^{1/2} & 0 \\ 0 & q^{-1/2} \end{pmatrix}
\times \on{Fr}_x \right) \in {}\LG.
\end{equation}

Let $\sigma$ be an Arthur parameter. We attach to it two subgroups of
$\LG$: $^\la G = {}^\la G_\sigma$ is the centralizer of the image of
$\SL_2$ in $\LG$ under $\sigma$, and $^\la H = {}^\la H_\sigma$ is
the Zariski closure of the image of $W(F)$ in $^\la G_\sigma$ under
$\sigma$.

\medskip

\noindent {\bf Technical point.} The group $^\la H$ may not be a
Langlands dual group, but the image of an admissible homomorphism $\LH
\to \LG$ (so that we have a surjection $\LH \to {}^\la H$). As
explained on \cite{L:BE}, Sections 1.4 and 1.7, we may enlarge $^\la H$
by a central torus to get a Langlands dual group $\LH$. In what follows
we will ignore this issue.\qed

\medskip

The idea of Langlands \cite{L:BE} (see also \cite{FLN,L:IAS}) is to
assign (bypassing Arthur parameters) to each irreducible automorphic
representations $\pi$ of $G(\AD_F)$ (more precisely, an $L$-packet)
directly the data of
\begin{equation}    \label{phi}
\phi: \SL_2 \times \LH \to \SL_2 \times {}^\la H \to \LG,
\end{equation}
and a tempered irreducible automorphic representation $\pi_H$ (or an
$L$-packet) of the group $H(\AD_F)$ whose dual group is $\LH$, which
Langlands calls {\em hadronic}. Note that the same $\pi$ might
correspond to several inequivalent homomorphisms $\phi$, but this is
expected to be related to the multiplicity of $\pi$ in the space of
automorphic functions.

{}From this point of view, $\pi$ should be thought of as a transfer of
$\pi_H$ with respect to $\phi$, so we obtain a more convenient notion
of transfer for non-tempered representations which explicitly involves
Arthur's group $\SL_2$ (note that if we restrict ourselves to a
temptered automorphic representation $\pi$ of $G$, then we expect
$\phi|_{\SL_2}$ to be trivial and therefore this transfer should agree
with the transfer discussed in \secref{LFC}). According to formula
\eqref{matching arthur}, under this transfer the Hecke eigenvalues
$(\nu_x^H)$ of $\pi_H$ and $(\nu_x)$ of $\pi$ should be matched by the
formula
\begin{equation}
\nu_x = \phi\left(\begin{pmatrix} q^{1/2} & 0 \\ 0 & q^{-1/2} \end{pmatrix}
\times \nu_x^H \right).
\end{equation}

Thus, we obtain a conjectural description of the equivalence classes
of automorphic representations of all reductive groups $G$ in terms of
the pairs $(\phi,\pi_H)$, where $\pi_H$ is hadronic, and $\phi$ is a
homomorphism \eqref{phi}. In \cite{L:BE,FLN} (see also
\cite{L:ST,L:IAS}) a strategy for proving it using the trace formula
was suggested. We discuss it next.

\subsection{Trace Formula}    \label{trace for}

Let $f$ be a smooth compactly supported function on $G(\AD_F)$. We
choose a Haar measure on $G({\mathbb A}_F$ normalized so that the
volume of the fixed maximal compact subgroup
$$
G(\OO_F) = \prod_{x \in |X|} G(\OO_x)
$$
is equal to $1$. Denote by ${\mb K}_f$ the operator on the space of
automorphic functions on $G(F)\bs G(\AD_F)$ acting by the formula
$$
({\mathbf K}_f \cdot \psi)(x) = \underset{G({\mathbb A)}}\int \psi(xy)
f(y) dy.
$$
Thus, we ``average'' the right action of $y \in G(\AD_F)$ with the
``weight'' $f(y)$.

We rewrite ${\mb K}_f$ as an integral operator corresponding to the
kernel
\begin{equation}    \label{K and H}
K_f(x,y) = \sum_{a \in G(F)} f(x^{-1} a y),
\end{equation}
acting as follows:
$$
({\mb K}_f \cdot \psi)(x) = \underset{G(F)\bs G({\mathbb A)}}\int
K_f(x,y) \psi(y) dy.
$$

The {\em Arthur--Selberg trace formula} for ${\mb K}$ reads, formally,
\begin{equation}    \label{general}
\on{Tr} {\mb K}_f = \underset{G(F)\bs G({\mathbb A)}}\int K_f(x,x) dx
\end{equation}
This is correct if $G(F) \bs G(\AD_F)$ is compact, in which case there
is no continuous spectrum; otherwise, some important modifications
need to be made. We will not discuss this here, referring the reader
to \cite{Arthur:trace} and references therein.

The left hand side of this formula, called the {\em spectral side},
may be rewritten as follows (we ignore the continuous spectrum):
\begin{equation}    \label{s1}
\sum_\pi m_\pi \on{Tr}({\mb K}_f,\pi),
\end{equation}
where the sum is over the irreducible automorphic representations
$\pi$ of $G(\AD_F)$ and $m_\pi$ is the multiplicity of $\pi$ in the
space of automorphic functions.

The right hand side of \eqref{general} may be rewritten as (see
\cite{Arthur:trace}, Section 1)
\begin{equation}    \label{sum over gamma}
\sum_{\gamma \in G(F)/\on{conj.}} a_\gamma O_\gamma(f),
\end{equation}
where $\gamma$ runs over the set of conjugacy classes in $G(F)$,
$O_\gamma(f)$ is the global {\em orbital integral} defined by the
formula
$$
O_\gamma(f) = \underset{G_\gamma(\AD_F)\bs G(\AD_F)}\int
f(g^{-1}\gamma g) dg
$$
and
$$
a_\gamma = \on{vol}(G_\ga(F)\bs G_\ga(\AD_F)).
$$
Here $G_{\gamma}(F)$ (resp., $G_\ga(\AD_F)$) denotes the stabilizer of
$\gamma$ in $G(F)$ (resp., $G(\AD_F)$).

The sum \eqref{sum over gamma} is usually called the ``geometric
side'' of the trace formula, but we will call it the {\em orbital
  side}, because by its ``geometrization'' we will understand its
representation as the trace of the Frobenius on a vector space.

Thus, the trace formula \eqref{general} takes the form
\begin{equation}    \label{trace formula}
\sum_\pi m_\pi \on{Tr}({\mb K}_f,\pi) = \sum_{\gamma \in
  G(F)/\on{conj.}} a_\gamma O_\gamma(f).
\end{equation}

We want to use this formula to establish the functoriality transfer
\eqref{funct pr}. The idea is to find enough relations between
$\on{Tr} {\mb K}_f$ and $\on{Tr} {\mb K}_{f^H}$ for a sufficiently
large class of functions $f$ on $G(\AD_F)$ and a suitable map $$f
\mapsto f^H,$$ where $f^H$ is a function on $H(\AD_F)$. Though the
formula \eqref{trace formula} for a single function $f$ does not
necessarily pin down a particular irreducible representation $\pi$ --
the left hand side of \eqref{trace formula} is a sum over those -- if
we have at our disposal formulas for the trace $\on{Tr} {\mb K}_f$ for
a sufficiently large class of functions $f$, then we can separate
different irreducible representations. So if we can prove enough
relations between $\on{Tr} {\mb K}_f$ and $\on{Tr} {\mb K}_{f^H}$,
then we can often derive the existence of an irreducible automorphic
representation $\pi$ of $G(\AD_F)$ whose Hecke eigenvalues match those
of an irreducible automorphic representation $\pi_H$ of $H(\AD_F)$
(or, in general, $L$-packets of those).

In order to find these relations between the traces, we look for
relations between the global orbital integrals $O_\gamma(f)$ and
$O_{\gamma^H}(f^H)$ appearing on the right hand side of \eqref{trace
  formula} for $G(\AD_F)$ and $H(\AD_F)$, respectively. This should
eventually be reduced to proving relations between local orbital
integrals of the local factors $f_x$ and $f^H_x$ of $f = \bigotimes_x
f_x$ and $f^H = \bigotimes_x f^H_x$, respectively. Thus, everything
should boil down to a problem in local harmonic analysis. This is the
basic strategy used to prove the Functoriality Conjecture.

This strategy has been successfully employed in a number of
cases. Perhaps, the most famous (and historically one of the first)
examples is the {\em Jacquet--Langlands theory} \cite{JL}. Here
$G=GL_2$ and $H$ is the multiplicative group of a quaternion algebra
over $F$, which is an inner form of $G$. Thus, $\LG=\LH=GL_2$ and we
take the identity as the homomorphism $a$ between them. Jacquet and
Langlands give a very explicit construction of the transfer of
functions and conjugacy classes under which the orbital integrals for
the two groups are equal. This, together with the strong multiplicity
one theorem for $GL_2$ (which says that the collection of conjugacy
classes $(\nu_x)$ at almost all points $x \in |X|$ uniquely determine
an irreducible automorphic representation of $GL_2(\AD_F)$), allows
them to prove the existence of the transfer $\pi_H \mapsto \pi$
satisfying the above properties.

\medskip

\noindent {\bf Technical point.} Labesse and Langlands have shown in
\cite{LL1} that the same strategy fails already in the case of
$G=\SL_2$. It turns out that for general groups on needs first to
``stabilize'' the trace formula. Roughly speaking, this allows to
write the orbital side in terms of the products, over $x \in |X|$, of
local {\em stable} orbital integrals. Here ``stable'' refers to
``stable conjugacy class'' in $G(F_x)$, the union of the conjugacy
classes in $G(F_x)$ that are conjugate to each other over the
algebraic closure of $F_x$ (if $G=GL_n$, each stable conjugacy class
consists of one conjugacy class, but for other groups it may consist
of several conjugacy classes). Stable orbital integrals for different
groups may then be compared. Fortunately, the stabilization of
trace formulas has now been achieved thanks to Ng\^o's recent proof of
the fundamental lemma \cite{Ngo:FL} and other important results.

One of the benefits of the stabilization of the trace formula is that
it leads to the proof of the Functoriality Conjecture for the
so-called {\em endoscopic} groups (at least, those of classical
types). These are the the groups $H$ whose Langlands dual groups $\LH$
are, roughly speaking, the centralizers of semi-simple elements in
$\LG$ (see, e.g., \cite{Ngo:FL} for a precise definition).\qed

\subsection{Strategy}    \label{str}

In \cite{FLN}, following \cite{L:BE,L:PR} (see also
\cite{L:ST,L:IAS,FN}), the following strategy for proving
functoriality has been proposed.

Suppose we had at our disposal an explicitly defined family of
integral operators $\{ {\mathbf K}_{f_i} \}$ on the space of
automorphic functions on $G(F)\bs G(\AD_F)$ which annihilate all
automorphic representations of $G(\AD_F)$ that do {\em not} come by
transfer from automorphic representations of $H(\AD_F)$. Then we will
be able to isolate in $L_2(G(F)\bs G(\AD_F))$ those representations
that come by functoriality from automorphic representations of
$H(\AD_F)$. Hence we may compare the traces of these operators with
the traces of some operators $\{ {\mathbf K}_{f^H_i} \}$ acting on
$L_2(H(F)\bs H(\AD_F))$ using the corresponding orbital integrals.

While it is not known how to do this literally for any given pair of
groups $G$ and $H$, operators with a similar property have indeed
been constructed in \cite{L:BE,FLN}.

In what follows we will describe a class of such operators, denoted by
${\mb K}_{d,\rho}$, defined in \cite{FLN}, restricting ourselves for
simplicity to unramified automorphic representations (for the general
case, see \cite{FLN}). In this case $f$ is a finite linear combination
of functions of the form $\bigotimes_x f_x$, where each $f_x$ is in
the spherical Hecke algebra of $G(F_x)$ (with respect to
$G(\OO_x)$), and $f_x=1$ for all but finitely many $x \in |X|$.

The operators ${\mb K}_{d,\rho}$ depend on a positive integer $d$ and
an irreducible representation $\rho$ of $\LG$ and are expected to have
the following property: for sufficiently large $d$, the
representations of $G(\AD_F)$ that contribute to the trace of ${\mb
  K}_{d,\rho}$ are those coming by functoriality from the groups $H$
satisfying the following:

\medskip

\noindent {\bf Property $(G,\rho)$:} the pull-back $a^*(\rho)$ of
$\rho$ to $\LH$ under the homomorphism $a: \LH \to \LG$ has non-zero
invariant vectors.

\medskip

This suggests a path to proving functoriality: we need to express the
right hand side of the trace formula for ${\mb K}_{d,\rho}$ as the sum
of orbital integrals and compare these orbital integrals for $G$ and
the groups $H$ satisfying property $(G,\rho)$.

\medskip

We will now give a simple example of a group satisfying this property
and then define the operators ${\mb K}_{d,\rho}$. Then we will compute
the eigenvalues of these operators on the Hecke eigenfunctions in
\secref{eig K}. Using this computation, we will derive the crucial
\lemref{odin} which shows that for large enough $d$ the operator ${\mb
  K}_{d,\rho}$ annihilates the Hecke eigenfunctions that do not come
by functoriality from the groups $H$ satisfying property $(G,\rho)$.

\subsection{Example}

What are the possible groups $H$ with the $(G,\rho)$ property?
Consider the case of $G=GL_2$. Then $\LG=GL_2$ as well. Let
$\rho=\rho_1$, the defining two-dimensional representation of
$GL_2$. We are interested in the reductive subgroups $\LH$ that
stabilize proper non-zero subspaces of $\rho$.  In this case this
subspace has to be a line in the two-dimensional vector space. The
group $\LH$ is then a one-dimensional torus. This is not a very
interesting example, because functoriality for tori is established via
the Eisenstein series.

Let us now consider the three-dimensional representation $\rho_2 =
\on{Sym}^2(\rho_1)$. If we choose a basis $\{ v_1,v_2 \}$ of $\rho_1$,
then $\rho_2$ has the basis
$$
\{ v_1 \otimes v_1, v_1 \otimes v_2 + v_2 \otimes v_1, v_2 \otimes v_2
\}.
$$
Then in addition to the one-dimensional torus stabilizing the vector
$v_1 \otimes v_1$, we will have another group with the $(GL_2,\rho_2)$
property, $O_2 = \Z_2 \ltimes GL_1$,
stabilizing the line spanned by the second vector. This group consists
of the matrices
$$
\begin{pmatrix} x & 0 \\ 0 & x^{-1} \end{pmatrix},
\qquad \begin{pmatrix}   x & 0 \\ 0 &
  x^{-1} \end{pmatrix} \begin{pmatrix} 0 & 1 \\ 1 & 0 \end{pmatrix}.
$$
It is the Langlands dual group of the twisted tori described in
\secref{L dual grp}. Hence the corresponding groups $H$ are the
twisted tori in this case.

\subsection{Definition of ${\mb K}_{d,\rho}$}

Let $X$ be defined over $k=\Fq$. Assume for simplicity that $G$ is
split over $k$ (a more general case is considered in
\cite{FN}). Recall the Hecke operator ${\mb H}_{\rho,x}, x \in |X|,
\rho \in \on{Rep} \LG$, and its kernel $K_{\rho,x}$, which is a
function on $\Bun_G(k) \times \Bun_G(k)$.

We define the kernel $K_{d,\rho}$ for $d \geq 1$ on $\Bun_G(k) \times
\Bun_G(k)$.

\medskip

For $d=1$ it is simply the sum of $K_{\rho,x}$ over all $x \in X(k)$:
$$
K_{1,\rho} = \sum_{x \in |X|} K_{\rho,x}.
$$

For $d=2$, we want to define the ``symmetric square'' of
$K_{1,\rho}$. In other words, we sum over the degree two effective
divisors $D$ -- these are the $k$-points in the symmetric square
$X^{(2)} = X^2/S_2$ of our curve $X$. There are three types of such
divisors: $D=(x)+(y)$, where $x,y \in X(k), x \neq y$ -- to which we
assign $K_{\rho,x} K_{\rho,y}$; $D=x$, where $x \in X({\mathbb
  F}_{q^2})$ -- we assign $K_{\rho,x}$; and $D=2(x)$ -- then naively
we could assign $K_{\rho,x}^2 = K_{\rho^{\otimes 2},x}$, but since we
want the symmetric product, we assign instead $K_{\rho^{(2)},x}$,
where
$$
\rho^{(2)} = \on{Sym}^2(\rho).
$$

Similarly, for $d>2$ we set
\begin{equation}    \label{kernel of kdrho}
K_{d,\rho} = \sum_{D \in X^{(d)}(k)} \; \prod_i
K_{\rho^{(n_i)},x_i}, \qquad D = \sum_i n_i[x_i],
\end{equation}
where
$$
\rho^{(n)} = \on{Sym}^n(\rho).
$$

Let ${\mb K}_{d,\rho}$ be the integral operator on functions on
$\Bun_G(k)$ corresponding to the kernel $K_{d,\rho}$. Thus,
\begin{equation}    \label{kernel of kdrho1}
{\mb K}_{d,\rho} = \sum_{D \in
 X^{(d)}(k)} \; \prod_i {\mb H}_{\rho^{(n_i)},x_i}, \qquad D = \sum_i
n_i[x_i].
\end{equation}

\subsection{Eigenvalues of ${\mb K}_{d,\rho}$}    \label{eig K}

Now let $f_\sigma$ be a Hecke eigenfunction on $\Bun_G(k)$ with
respect to an unramified homomorphism $\sigma: W(F) \to \LG$. Recall
from formulas \eqref{Hecke ef} and \eqref{hx} that
$$
{\mb H}_{V,x} \cdot f_\sigma =
\on{Tr}(\sigma(\on{Fr}_x),V) \; f_\sigma
$$
for any finite-dimensional representation $V$ of $\LG$.

Therefore we find from formula \eqref{kernel of kdrho1} that
\begin{equation}    \label{eig Kdrho}
{\mb K}_{d,\rho} \cdot f_\sigma = l_{d,\rho} f_\sigma,
\end{equation}
where
\begin{equation}    \label{ldrho}
l_{d,\rho} = \sum_{D \in
 X^{(d)}(k)} \quad \prod_i
\on{Tr}(\sigma(\on{Fr}_{x_i}),\rho^{(n_i)}), \qquad D = \sum_i
n_i[x_i].
\end{equation}

Consider the generating function of these eigenvalues:
\begin{equation}    \label{Lf}
L(\sigma,\rho,t) = \sum_{d \geq 0} l_{d,\rho} t^d = \prod_{x \in |X|}
\on{det}(1-t^{\deg(x)} \sigma(\on{Fr}_x),\rho)^{-1}.
\end{equation}
If we substitute $t=q^{-s}$, we obtain the $L$-{\em function}
$L(\sigma,\rho,q^{-s})$ attached to $\sigma$ and $\rho$. Thus, formula
\eqref{eig Kdrho} implies that the eigenvalues of ${\mb K}_{d,\rho}$
are the coefficients of this $L$-function:
\begin{equation}    \label{eig Kdrho1}
{\mb K}_{d,\rho} \cdot f_\sigma = \left( q^{-ds}\text{-coefficient} \;
  \text{of} \; L(\sigma,\rho,q^{-s}) \right) \; f_\sigma.
\end{equation}

\begin{lem} \label{odin}
  Suppose that the spaces of invariants and coinvariants of the
  representation $\rho \circ \sigma$ of the Weil group $W(F)$ are
  equal to $0$. Then the corresponding Hecke eigenfunction $f_\sigma$
  satisfies
\begin{equation}    \label{odin formula}
{\mb K}_{d,\rho} \cdot f_\sigma = 0, \qquad d>2(g-1)\dim\rho,
\end{equation}
where $g$ is the genus of $X$.

The same statement holds if $f_\sigma$ is a Hecke eigenfunction
corresponding to an Arthur parameter $\sigma: \SL_2 \times W(F) \to
\LG$ and the spaces of invariants and coinvariants of the
representation $\rho \circ \sigma|_{W(F)}$ of the Weil group $W(F)$
are equal to $0$.
\end{lem}

\begin{cor}    \label{odin cor}
  Let $^\la H_\sigma$ be the Zariski closure of the image of $W(F)$ in
  $^\la G_\sigma$ under $\sigma$. Then ${\mb K}_{d,\rho}$ annihilates
  $f_\sigma$ for $d>2(g-1)\dim\rho$ unless the restriction of $\rho$
  to $^\la H$ has non-zero invariants or coinvariants.
\end{cor}

We expect that the representation $\rho \circ \sigma$ is semi-simple,
and hence its invariants and coinvariants are isomorphic. We will
assume that this is the case in what follows.

\subsection{Proof of \lemref{odin}}

In order to prove the lemma, we recall the Grothendieck--Lefschetz
formula for the $L$-function.

Suppose we are given an unramified $n$-dimensional $\ell$-adic
representation of $W(F)$. We attach to it an $\ell$-adic locally
constant sheaf (local system) ${\mc L}$ on $X$.

Let $X^{(d)} = X^d/S_d$ is the $d$th symmetric power of $X$. This is a
smooth algebraic variety defined over $k$, whose $k$-points are
effective divisors on $X$ of degree $d$. We define a sheaf on $X^{(d)}$,
denoted by $\Lc^{(d)}$ and called the {\em
$d$th symmetric power of $\Lc$}, as follows:
\begin{equation}    \label{stalk}
{\Lc}^{(d)} = \left(\pi^d_*(\Lc^{\boxtimes d})\right)^{S^{(d)}},
\end{equation}
where $\pi^d: X^d \to X^{(d)}$ is the natural projection. The stalks
of $\Lc^{(d)}$ are easy to describe: they are tensor products of
symmetric powers of the stalks of $\Lc$. The stalk ${\mc L}_{d,D}$ at
a divisor $D = \sum_i n_i [x_i]$ is
$$
\Lc^{(d)}_{D} = \bigotimes_i S^{n_i}({\mc \Lc}_{x_i}),
$$
where $S^{n_i}({\mc \Lc}_{x_i})$ is the $n_i$-th symmetric power of
the vector space ${\mc \Lc}_{x_i}$.  In particular, the dimensions of
the stalks are not the same, unless $n=1$. (In the case when $n=1$ the
sheaf $\Lc^{(d)}$ is in fact a rank 1 local system on $X^{(d)}$.) For
all $n$, $\Lc^{(d)}$ is actually a perverse sheaf on $X^{(d)}$ (up to
cohomological shift), which is irreducible if and only if ${\mc L}$ is
irreducible.

Now observe that
\begin{equation}    \label{D Tr}
\on{Tr}(\on{Fr}_D,\Lc^{(d)}_D) = \prod_i
\on{Tr}(\sigma(\on{Fr}_{x_i}),\rho^{(n_i)}), \qquad D = \sum_i
n_i[x_i],
\end{equation}
where $\on{Fr}_D$ is the Frobenius automorphism corresponding to the
$k$-point $D$ of $X^{(d)}$ and ${\mc L}_D$ is the stalk of ${\mc L}_d$
at $D$.

By the Lefschetz trace formula (see, e.g., \cite{Milne,Weil}), the
trace of the Frobenius on the \'etale cohomology of an $\ell$-adic
sheaf is equal to the sum of the traces on the stalks at the
$\Fq$-points:
\begin{equation}    \label{coeff}
\on{Tr}(\on{Fr},H^\bullet(X^{(d)},{\mc
  L}^{(d)})) = \sum_{D \in X^{(d)}(\Fq)} \on{Tr}(\on{Fr}_D,{\mc
  L}^{(d)}_{D}),
\end{equation}
By formula \eqref{D Tr}, the RHS of \eqref{coeff} is $l_{d,\rho}$
given by formula \eqref{ldrho}, and also the $q^{-ds}$-coefficient of
the $L$-function $L(\sigma,\rho_{\on{def}},q^{-s})$ of the
representation $\sigma: W(F) \to GL_n$ associated to ${\mc L}$ and the
defining $n$-dimensional representation $\rho_{\on{def}}$ of $GL_n$.

Let us compute the cohomology of $X^{(d)}$ with coefficients
in ${\mc L}^{(d)}$. By the K\"unneth formula, we have
$$
H^\bullet(X^{(d)},{\mc L}^{(d)}) = \left( H^\bullet(X^{d},{\mc
    L}^{\boxtimes d}) \right)^{S_d} = \left( H^\bullet(X,{\mc
    L})^{\otimes d} \right)^{S_d},
$$
where the action of the symmetric group $S_d$ on the cohomology is as
follows: it acts by the ordinary transpositions on the even cohomology
and by signed transpositions on the odd cohomology. Thus, we find that
\begin{multline}    \label{sym power}
H^\bullet(X^{(d)},{\mc L}^{(d)}) = \\
\bigoplus_{d_0+d_1+d_2=d}
S^{d_0}(H^0(X,{\mc L})) \otimes \Lambda^{d_1}(H^1(X,{\mc L})) \otimes
S^{d_2}(H^2(X,{\mc L})).
\end{multline}
The cohomological grading is computed according to the rule that $d_0$
does not contribute to cohomological degree, $d_1$ contributes $d_1$,
and $d_2$ contributes $2d_2$. In addition, we have to take into
account the cohomological grading on ${\mc L}$.

Formulas \eqref{coeff} and \eqref{sym power} give us the following
expression for the generating function of the RHS of \eqref{coeff}:
\begin{multline} \label{LG formula} \sum_{d\geq 0} t^d \sum_{D \in
    X^{(d)}(\Fq)} \on{Tr}(\on{Fr}_D,{\mc L}_{D}) = \\
  \frac{\on{det}(1-t \on{Fr},H^1(X,{\mc L}))}{\on{det}(1-t
    \on{Fr},H^0(X,{\mc L})) \; \on{det}(1-t\on{Fr},H^2(X,{\mc L}))}.
\end{multline}
This is the Grothendieck--Lefschetz formula for the $L$-function
$L(\sigma,\rho_{\on{def}},t)$.

Now let $\Lc_{\rho \circ \sigma}$ be the $\ell$-adic local system
corresponding to the representation $\rho \circ \sigma$ of
$W(F)$. Then the trace of the Frobenius on the right hand side of
\eqref{sym power} gives us the eigenvalue $l_{d,\rho}$ of ${\mb
  K}_{d,\rho}$ on $f_\sigma$.

If $\rho \circ \sigma$ has zero spaces of invariants and coinvariants,
then
$$
H^0(X,\Lc_{\rho \circ \sigma}) = H^2(X,\Lc_{\rho \circ \sigma}) = 0
$$
and $\dim H^1(X,\Lc_{\rho \circ \sigma}) = (2g-2)\dim\rho$ (since it
is then equal to the Euler characteristic of the constant local system
of rank $\dim\rho$ on $X$). Hence we obtain that
$$
H^\bullet(X^{(d)},\Lc_{\rho\circ\sigma}^{(d)}) \simeq
\Lambda^d(H^1(X,\Lc_{\rho \circ \sigma})),
$$
and the $L$-function $L(\sigma,\rho,q^{-s})$ is a polynomial in
$q^{-s}$ of degree $2(g-1)\dim\rho$. \lemref{odin} then follows from
formula \eqref{eig Kdrho1}.

\subsection{Decomposition of the trace formula}    \label{decomp
  trace}

The goal of the program outlined in \cite{FLN} (see also
\cite{L:BE,L:ST,L:IAS}) is to use the trace formula to prove the
existence of the functoriality transfers corresponding to the
homomorphisms $\phi$ given by formula \eqref{phi} in \secref{arthur
  funct}. The basic idea is to apply the trace formula to the
operators ${\mb K}_{d,\rho}$ and use \lemref{odin}. Here is a more
precise description.

As the first step, we need to remove from the trace formula the
contributions of the non-tempered representations (those correspond to
the transfer associated to the homomorphisms $\phi$ whose restriction
to $\SL_2$ is non-trivial), because these terms dominate the trace
formula (see the calculation below). Then we want to use \lemref{odin}
to isolate in the trace of ${\mb K}_{d,\rho}$ with $d>(2g-2)\dim \rho$
the terms corresponding to the automorphic representations of
$G(\AD_F)$ that come by functoriality from the groups $H$ satisfying
the $(G,\rho)$ property.

\medskip

Recall from \secref{arthur funct} that to each Arthur parameter
$\sigma$ we attach the group $^\la H = {}^\la H_\sigma$ which is the
Zariski closure of the image of $W(F)$ in $^\la G_\sigma$ under
$\sigma$. According to \lemref{odin} and \corref{odin cor}, for large
enough $d$ the operator ${\mb K}_{d,\rho}$ acts non-trivially only on
those automorphic representations of $G(\AD_F)$ which correspond to
$\sigma$ such that $\rho \circ \sigma$ has non-zero invariants. This
happens if and only if the restriction of $\rho$ to $^\la H_\sigma$
has non-zero invariants; in other words, if and only if $^\la H$
satisfies property $(G,\rho)$.

Thus, assuming that Arthur's conjectures are true, we obtain that the
trace of ${\mb K}_{d,\rho}$ decomposes as a double sum: first, over
different homomorphisms
$$
\varphi: \SL_2 \to {}\LG,
$$
and second, for a given $\varphi$, over the subgroups $^\la H$ of the
centralizer $^\la G_\varphi$ of $\varphi$ having non-zero invariants in
$\rho$:
\begin{equation}    \label{Phi and Psi}
\on{Tr} {\mb K}_{d,\rho} = \sum_\varphi \sum_{^\la H \subset {}^\la
  G_\varphi} \Phi_{\varphi,{}^\la H}.
\end{equation}
Here $\Phi_{\varphi,{}^\la H}$ is the trace over the automorphic
representations of $G(\AD_F)$ which come from the transfer of tempered
(hadronic) representations of $H(\AD_F)$ with respect to
homomorphisms $\phi$, given by \eqref{phi}, such that $\phi|_{\SL_2} =
\varphi$.

A precise formula for these eigenvalues of ${\mb K}_{d,\rho}$ is
complicated in general, but we can compute its asymptotics as $d \to
\infty$.

Suppose first that $\varphi$ is trivial, so we are dealing with the
tempered representations.  If we divide ${\mb K}_{d,\rho}$ by $q^d$,
then the asymptotics will be very simple:
\begin{equation}    \label{asymptotics}
q^{-d} (\on{Tr} {\mb K}_{d,\rho})_{\on{temp}} \sim \sum_{^\lambda H \subset \LG}
\quad \sum_{\sigma': W(F) \to {}^\lambda H} N_\sigma \begin{pmatrix}
d+m_\sigma(\rho)-1 \\ m_\sigma(\rho)-1 \end{pmatrix},
\end{equation}
where $N_\sigma$ is the multiplicity of automorphic representations in
the corresponding $L$-packet.

Indeed, the highest power of $q$ comes from the highest cohomology,
which in this case is
$$
H^{2d}(X^{(d)},\Lc_{\rho \circ \sigma}^{(d)}) = \on{Sym}^{d}(H^2(X,\rho
\circ \sigma))
$$
($d_0=0, d_1=0$, and $d_2=d$ in the notation of formula \eqref{sym
power}). We have $\dim H^2(X,\rho \circ \sigma) = m_\sigma(\rho)$, the
multiplicity of the trivial representation in $\rho \circ
\sigma$ (we are assuming here again that this trivial representation
splits off as a direct summand in $\rho \circ \sigma$), and
\begin{equation}    \label{dim}
\dim \on{Sym}^{d}(H^2(X,\Lc_{\rho \circ \sigma}) = \begin{pmatrix}
  d+m_\sigma(\rho)-1 \\ m_\sigma(\rho)-1 \end{pmatrix}.
\end{equation}

Thus, as a function of $q^d$, the eigenvalues of $q^{-d} {\mb
  K}_{d,\rho}$ on the tempered representations grow as $O(1)$ when $d
\to \infty$.

For the non-tempered representations corresponding to non-trivial
$\varphi: \SL_2 \to {}\LG$, they grow as a higher power of $q^d$. For
instance, the eigenvalue of $q^{-d} {\mb K}_{d,\rho}$ corresponding to
the trivial representation of $G(\AD_F)$ (for which $\varphi$ is a
principal embedding) grows as $O(q^{d(\rho,\mu)})$, where $\mu$ is the
highest weight of $\rho$ (see the calculation in \secref{functor
  Kdrho} below). In general, it grows as $O(q^{da})$, where $2a$ is
the maximal possible highest weight of the image of $\SL_2 \subset
{}\LG$ under $\varphi$ acting on $\rho$. Thus, we see that the
asymptotics of the non-tempered representations dominates that of
tempered representations. This is why we wish to remove the
contribution of the non-tempered representations first.

\medskip

Note that if $\varphi$ is non-trivial, then the rank of $^\la
G_\varphi$ is less than that of $\LG$. As explained in \cite{FLN}, we
would like to use induction on the rank of $\LG$ to isolate and get
rid of the terms in \eqref{Phi and Psi} with non-trivial $\varphi$. In
\cite{FLN} it was shown how to isolate the contribution of the trivial
representation of $G(\AD_F)$ for which $\varphi$ is the principal
embedding (it is, along with all other one-dimensional representations
of $G(\AD_F)$, the most non-tempered).

\medskip

If we can do the same with other non-tempered contributions, then we
will be left with the terms $\Phi_{\on{triv},{}^\la H}$ in
\eqref{Phi and Psi} corresponding to the tempered representations of
$G(\AD_F)$. Denote their sum by $(\on{Tr} {\mb K}_{d,\rho})_{\on{temp}}$.  We
try to decompose it as a sum over $^\la H$:
\begin{equation}    \label{sum2}
(\on{Tr} {\mb K}_{d,\rho})_{\on{temp}} = \sum_{^\la H \subset \LG} (\on{Tr}
  {\mb K}^H_{d,\rho_H})_{\on{temp}}.
\end{equation}
Here the sum should be over all possible $^\lambda H \subset {}\LG$
such that $^\lambda H$ has non-zero invariant vectors in $\rho$, and
${\mb K}^H_{d,\rho_H}$ is the operator corresponding to $\rho_H =
\rho|_{^\la H}$ for the group $H(\AD_F)$ (note that different groups
$H$ may correspond to the same $^\la H$).

The ultimate goal is to prove formula \eqref{sum2} by comparing the
orbital sides of the trace formula for the operators ${\mb
  K}_{d,\rho}$ and ${\mb K}^H_{d,\rho_H}$. Of course, for any given
$\rho$, the right hand side of formula \eqref{sum2} will contain
contributions from different groups $H$. However, because we have two
parameters: $\rho$ and $d$ (sufficiently large), we expect to be able
to separate the contributions of different groups by taking linear
combinations of these formulas with different $\rho$ and $d$. In the
case of $G=GL_2$ this is explained in \cite{L:BE}.

\medskip

There are many subtleties involved in formulas \eqref{Phi and Psi} and
\eqref{sum2}. As we mentioned above, $^\la H$ may not be itself a
Langlands dual group, but it can be enlarged to one. There may exist
more than one conjugacy class of $^\la H$ assigned to a given
automorphic representation; this is expected to be related to the
multiplicities of automorphic representations. Also, comparisons of
trace formulas should always be understood as comparisons of their
{\em stabilized} versions. Therefore the traces in \eqref{sum2} should
be replaced by the corresponding stable traces

The upshot is that we want to establish formula \eqref{sum2} by
proving identities between the corresponding orbital integrals. In the
case of $G=\SL_2$ the first steps have been made in \cite{L:ST}, where
we refer the reader for more details.

One powerful tool that we hope to employ is the {\em geometrization}
of these orbital integrals. We discuss this in the next section.

\section{Geometrization of the orbital part of the trace
  formula}    \label{gos}

Our goal is to construct, in the case that the curve $X$ is defined
over a finite field $\Fq$, a vector space with a natural action of
$\on{Gal}(\ovl{\mathbb F}_q/\Fq)$ such that the trace of the Frobenius
automorphism is equal to the right hand side of the trace formula
\eqref{trace formula}. We hope that this construction will help us to
prove the decompositions \eqref{Phi and Psi} and \eqref{sum2} on the
orbital side of the trace formula. Another important aspect of the
construction is that this vector space will be defined in such a way
that it will also make sense if the curve $X$ is over $\C$. In this
section we outline this construction following \cite{FN}. (We will
discuss the geometrization of the spectral side of the trace formula
in \secref{leap}.)

This section is organized as follows. In \secref{geom Arthur} we
describe the geometric analogues of the Arthur parameters. These are
certain complexes of local systems on the curve $X$. Then in
\secref{functor Kdrho} we construct the sheaf ${\mc K}_{d,\rho}$ such
that the corresponding function is the kernel $K_{d,\rho}$ from the
previous section. We define the corresponding functor ${\mathbb
  K}_{d,\rho}$ acting on the category of sheaves on $\Bun_G$ in
\secref{action of K}.  In order to geometrize the orbital part of the
trace formula, we need to restrict ${\mc K}_{d,\rho}$ to the diagonal
and take the cohomology. This is explained in \secref{v sp}. We
interpret the resulting vector space as the cohomology of a certain
moduli stack defined in Sections \ref{orbital}--\ref{def stack} which
we call the moduli of ``$G$-pairs''. We compare it to the Hitchin
moduli stack in \secref{comparison} and define an analogue of the
Hitchin map in \secref{analogue}. Most of this is taken from
\cite{FN}. In \secref{ex gl2} we discuss in detail the example of
$G=GL_2$. Finally, we present some of the conjectures of \cite{FN} in
\secref{conj gen}.

\subsection{Geometric Arthur parameters}    \label{geom Arthur}

First, we discuss geometric analogues of the Arthur parameters
(in the unramified case).

Let $\rho$ be a representation of $\LG$ on a finite-dimensional vector
space $V$. Then we obtain a representation $\rho \circ \sigma$ of
$\SL_2 \times W(F)$ on $V$. The standard torus of $\SL_2 \subset \SL_2
\times W(F)$ defines a $\Z$-grading on ${\rho\circ\sigma}$:
$$
{\rho\circ\sigma} = \bigoplus_{i\in \Z} \; (\rho\circ\sigma)_i,
$$
where each $\rho\circ\sigma_i$ is a continuous representation of
$W(F)$. Assume that each of them is unramified. Then it gives rise to
an $\ell$-adic local system $\Lc_{(\rho\circ\sigma)_i}$ on $X$. Now we
define $\Lc_{\rho\circ\sigma}$ to be the following complex of local
systems on $X$ with the trivial differential:
$$\Lc_{\rho\circ\sigma} = \bigoplus_{i\in \Z}
\Lc_{(\rho\circ\sigma)_i}[-i].$$

We generalize the notion of Hecke eigensheaf by allowing its
eigenvalue to be an Arthur parameter $\sigma$, as in \eqref{Arthur
  parameter}, by saying that we have a collection of isomorphisms
\begin{equation}    \label{iso1}
{\mathbb H}_\rho({\mc F}) \simeq \Lc_{\rho \circ \sigma} \boxtimes
{\mc F},
\end{equation}
for $\rho \in \on{Rep} \LG$, compatible with respect to the structures
of tensor categories on both sides.

\medskip

As an example, consider the constant sheaf on $\Bun_G$, ${\mc F}_0 =
\underline\Ql|\Bun_G$. This is the geometric analogue of the trivial
representation of $G({\mathbb A}_F$. Assume that $G$ is split and let
$\rho=\rho_\mu\otimes \rho_0$, where $\rho_\mu$ is the irreducible
representation of highest weight $\mu$, $\rho_0$ is the trivial
representation of $\Gamma$. In this case $\wF_\rho=\wF_\mu$ is the
intersection cohomology complex of $\Hc_\mu$ shifted by
$-\dim(X\times\Bun_G)$.

Let us apply the Hecke functor ${\mathbb H}_{\mu,x}$ to ${\mc
  F}_0$. For any $G$-principal bundle $E$, the fiber of $p^{-1}(E)\cap
\Hc_x$ is isomorphic to $\Gr_x$, once we have chosen a trivialization
of $E$ on the formal disc $D_x$. Thus the fiber of ${\mathbb
  H}_{\rho,x}({\mc F}_0)$ at $E$ is isomorphic to
$$
{\mathbb H}_{\rho,x}({\mc
F})_E=H^\bullet(\ol{\on{Gr}}_\mu,\IC(\ol\Gr_\mu)).
$$
This isomorphism does not depend on the choice of the trivialization
of $E$ on $D_x$, so we obtain that
$$
{\mathbb H}_{\rho,x}({\mc F}_0) \simeq
H^\bullet(\ol{\on{Gr}}_\mu,\IC(\ol\Gr_\mu)) \otimes {\mc F}_0.
$$

By the geometric Satake correspondence \cite{MV},
$$
H^\bullet(\ol{\on{Gr}}_\mu,\IC(\ol\Gr_\mu)) \simeq \rho_\mu^{\on{gr}},
$$
a complex of vector spaces, which is isomorphic to the representation
$\rho_\mu$ with the cohomological grading corresponding to the
principal grading on $\rho_\mu$. One can show that as we vary $x$, the
``eigenvalue'' of the Hecke functor ${\mathbb H}_\rho$ is the complex
$\rho_\mu^{\on{gr}} \otimes \Lc_0$, where $\Lc_0$ is the trivial local
system on $X$. In other words, it is the local system $\Lc_{\rho \circ
  \sigma_0}$ as defined above, where $\sigma_0: W(F) \times \SL_2 \to
{}\LG$ is trivial on $W(F)$ and is the principal embedding on
$\SL_2$. We conclude that the constant sheaf on $\Bun_G$ is a Hecke
eigensheaf with the eigenvalue $\sigma_0$.  This is in agreement with
the fact that $\sigma_0$ is the Arthur parameter of the trivial
automorphic representation of $G({\mathbb A}_F$.

For example, if $\rho_\mu$ is the defining representation of $GL_n$,
then the corresponding Schubert variety is ${\mathbb P}^{n-1}$, and we
obtain its cohomology shifted by $(n-1)/2$, because the intersection
cohomology sheaf $\IC(\ol\Gr_\mu)$ is the constant sheaf placed in
cohomological degree $-(n-1)$, that is
$$
H^\bullet(\ol{\on{Gr}}_\mu,\IC(\ol\Gr_\mu)) = \theta^{(n-1)/2} \oplus
\theta^{(n-3)/2} \oplus \ldots \oplus \theta^{-(n-1)/2},
$$
where $\theta^{1/2} = \Ql[-1](-1/2)$. This agrees with the fact that
the principal grading takes values $(n-1)/2,\ldots,-(n-1)/2$ on the
defining representation of $GL_n$, and each of the corresponding
homogeneous components is one-dimensional.

\subsection{The sheaf ${\mc K}_{d,\rho}$}    \label{functor
  Kdrho}

Next, we define a geometric analogue of the operator ${\mb
  K}_{d,\rho}$. Recall that ${\mb K}_{d,\rho}$ is defined as the
integral operator with the kernel $K_{d,\rho}$ given by formula
\eqref{kernel of kdrho}. It is a function on the set of $\Fq$-points
of the moduli stack $\Bun_G$ of $G$-bundles on $X$. Hence, according
to the Grothendieck philosophy, we need to do replace $K_{d,\rho}$ by
an $\ell$-adic sheaf ${\mc K}_{d,\rho}$ on $\Bun_G \times \Bun_G$,
whose ``trace of Frobenius'' function is $K_{d,\rho}$. We then define
${\mathbb K}_{d,\rho}$ as the corresponding ``integral transform''
functor acting on the derived category of $\ell$-adic sheaves on
$\Bun_G$. The construction of this sheaf will work also if $X$ id
define over $\C$, in which case ${\mathbb K}_{d,\rho}$ will be a
functor on the derived category of ${\mc D}$-modules on $\Bun_G$.

Since $K_{d,\rho}$ is built from the kernels $K_{\rho^{(n)},x}$ of the
Hecke operators according to formula \eqref{kernel of kdrho}, we need
to perform the same construction with the sheaves ${\mc
  K}_{\rho^{(n)},x}$ defined in \secref{gen hecke}, which are the
geometric counterparts of the $K_{\rho^{(n)},x}$.

Introduce the algebraic stack ${\mc H}_{d}$ over the field $k$ (which
is either a finite field or $\C$) that classifies the data
\begin{equation}    \label{objects}
(D,E,E',\phi),
\end{equation}
where
\begin{equation}
D = \sum_{i=1}^r n_i [x_i]
\end{equation}
is an effective divisor on our curve $X$ of degree $d$ (equivalently,
a point of $X^{(d)}$), $E$ and $E'$ are two principal $G$-bundles on
$X$, and $\phi$ is an isomorphism between them over $X-\on{supp}(D)$.

Let ${\mc H}_{d,\mu}$ be the closed substack of ${\mc H}_d$ that
classifies the quadruples \eqref{objects} as above, satisfying the
condition $\inv_{x_i}(E,E')\leq n_i \mu$.

Consider the morphism
$$
{\mc H}_{d,\mu} \to X^{(d)}\times \Bun_G
$$
sending the quadruple \eqref{objects} to $(D,E)$. Its fiber over a
fixed $D=\sum_i n_i [x_i]$ and $E\in\Bun_G$ isomorphic to the product
\begin{equation}    \label{product}
\prod_{i=1}^r \ovl\Gr_{[n_i \mu_i]}.
\end{equation}
Our sheaf ${\mc K}_{d,\rho}$ on ${\mc H}_{d,\mu}$ will have the
property that its restriction to these fibers are isomorphic to
isomorphic to
\begin{equation}    \label{restriction}
\boxtimes_{i=1}^d \IC_{\rho^{(n_i)}},
\end{equation}
where $ \ovl\Gr_{[n_i \mu_i]}$ is the perverse sheaf on $\ovl\Gr_{[n_i
  \mu_i]}$ corresponding to the representation $\rho^{(n_i)}$ (the
$n_i$th symmetric power of $\rho$) under the geometric Satake.

The precise definition of ${\mc K}_{d,\rho}$ is given in \cite{FN},
and here we give a less formal, but more conceptual construction.

\medskip

We will use the following result from the Appendix of \cite{FGV} that
generalizes the geometric Satake correspondence to the case of
``moving points''.

For any partition ${\mathbf d} = (d^1,\ldots,d^k)$ of $d$, consider
the open subset $\ovc{X}{}^{\mathbf d}$ of $X^{(d^1)} \times \ldots
\times X^{(d^k)}$ consisting of $k$--tuples of divisors
$(D_1,\ldots,D_k)$, such that $\on{supp} D_i \cap \on{supp} D_j =
\emptyset$, if $i\neq j$. Denote the map $\ovc{X}{}^{\mathbf d} \to
X^{(d)}$ by $p_{\mathbf d}$.  We introduce an abelian category ${\mc
A}^d$ as follows.  The objects of ${\mc A}^d$ are perverse sheaves
$\F$ on $X^{(d)}$ equipped with a $\LG$--action, together with the
following extra structure: for each partition ${\mathbf d}$, the sheaf
$p_{\mathbf d}^*(\F)$ should carry an action of $k$ copies of $\LG$,
compatible with the original $\LG$--action on $\F$ with respect to
the diagonal embedding $\LG \to (\LG)^{\times k}$. For different
partitions, these actions should be compatible in the obvious
sense. In addition, it is required that whenever $d^i=d^j, i \neq j$,
the action of the $i$th and $j$th copies of $\LG$ on $p_{\mathbf
d}^*(\F)$ should be intertwined by the corresponding natural
$\Z_2$--action on $\ovc{X}{}^{\mathbf d}$.

The claim of \cite{FGV} is that the category ${\mc A}^d$ is equivalent
to a certain category of perverse sheaves on ${\mc
  H}_d$. For instance, if $d=1$, then the constant sheaf on
$X^{(1)} = X$ with the stalk $\rho = \rho_\mu$, a finite-dimensional
representation of $\LG$, goes to a perverse sheaf on ${\mc
  H}_1$, whose restriction to each fiber of the projection
${\mc H}_1 \to X \times \Bun_G$ (which is isomorphic to
$\Gr$) is the perverse sheaf $\IC_\mu$.

More generally, for each representation $\rho$ of $\LG$ we have an
object $\underline{\rho}^{(d)}$ of the category ${\mc A}^d$ defined as
follows:
$$
\underline{\rho}^{(d)} = \left(\pi^d_*(\underline{\rho}^{\boxtimes
  d})\right)^{S_d},
$$
where $\underline{\rho}$ is the constant sheaf on $X$ with the stalk
$\rho$, $\pi^d: X^d \to X^{(d)}$ is the natural projection, and $S_d$
is the symmetric group on $d$ letters. It is easy to see that it
carries the structures from the above definition. Moreover, it is an
irreducible object of the category ${\mc A}^d$ if $\rho$ is
irreducible. Note that the stalk of
$\underline{\rho}^{(d)}$ at the divisor $D \subset X^{(d)}$ is the
tensor product
$$
\bigotimes_i \rho^{(n_i)},
$$
where $\rho^{(n)}$ denotes the $n$th symmetric power of $\rho$.

Now, define the sheaf ${\mc K}_{d,\mu}$ as the irreducible perverse
sheaf on ${\mc H}_d$ corresponding to $\underline{\rho}^{(d)}$. Then
its restriction to the fiber at $(D,{\mc M}) \in X^{(d)} \times
\Bun_G$ for general $D$ is given by formula \eqref{restriction}, and
so the ``trace of Frobenius'' function corresponding to ${\mc
  K}_{d,\mu}$ is the function $K_{d,\mu}$ given by formula
\eqref{kernel of kdrho}.

Note that the highest weights of the irreducible representations in
the decomposition of $\rho^{(n)}$ are less than or equal to $n\mu$,
and so the sheaf $\IC_{\mu}^{(n)}$ is supported in
$\ol{Gr}_{n\mu}$. Therefore ${\mc K}_{d,\mu}$ is supported on the
closed substack ${\mc H}_{d,\mu}$ introduced above.

\subsection{The functor ${\mathbb K}_{d,\rho}$}    \label{action of K}

We now have a morphism
$$\Hc^\BD_d\to X^{(d)}\times \Bun_G\times\Bun_G$$ and a perverse sheaf
$\wF_{d,\rho}$ on $\Hc^\BD_d$ attached to any finite-dimensional
representation $\rho$ of $\LG$. Let us denote by $p_d$ and $p'_d$ the
two projections $\Hc^\BD_d \to \Bun_G$ mapping the quadruple
\eqref{objects} to $E$ and $E'$. We use the sheaf $\wF_{d,\rho}$ to
define an integral transform functor ${\mathbb K}_{d,\rho}$ on the
derived category $D(\Bun_G)$ of $\ell$-adic sheaves on $\Bun_G$ by the
formula
\begin{equation}
{\mathbb K}_{d,\rho}({\mc F}) =p_{d!}(p'_d{}^*({\mc
    F}) \otimes \wF_{d,\rho}).
\end{equation}

Now we compute the ``eigenvalues'' of the functors ${\mathbb
  K}_{d,\rho}$ on ${\mc F}_\sigma$. The following result, which is the
geometrization of formulas \eqref{ldrho} and \eqref{D Tr} for the
eigenvalues of ${\mb K}_{d,\rho}$ on $f_\sigma$, is Lemma 2.6 of
\cite{FLN}.

\begin{lem} \label{Lemma 1 bis}
If ${\mc F} = {\mc F}_\sigma$ is a Hecke eigensheaf with eigenvalue
$\sigma$, then for every representation $\rho$ of $\LG$ and every
positive integer $d$ we have
\begin{equation} \label{d-th coeff}
{\mathbb K}_{d,\rho}({\mc
  F}_\sigma)=H^\bullet(X^{(d)},\Lc_{\rho\circ\sigma}^{(d)}) \otimes {\mc
  F}_\sigma.
\end{equation}
\end{lem}

Taking the trace of the Frobenius on the RHS of formula \eqref{d-th
  coeff}, we obtain the eigenvalues of ${\mb K}_{d,\rho}$ on the Hecke
eigenfunctions $f_\sigma$ corresponding to both tempered and
non-tempered automorphic representations.

As an example, we compute the action of ${\mathbb K}_{d,\rho}$ on the
constant sheaf on $\Bun_G$:
$$
{\mathbb K}_{d,\rho_\mu}(\underline\Ql) \simeq
H^\bullet(X^{(d)},\Lc_{\rho_\mu \circ \sigma_0}^{(d)}) \otimes
\underline\Ql,
$$
where $\rho_\mu \circ \sigma_0$ is the complex described in
\secref{geom Arthur}. 

At the level of functions, we are multiplying the constant function on
$\Bun_G(\Fq)$ (corresponding to the trivial representation of
$G(\AD_F)$) by
\begin{equation}    \label{prod of zeta1}
\on{Tr}(\on{Fr},H^\bullet(X^{(d)},\Lc_{\rho_\mu \circ \sigma_0}^{(d)})) =
\prod_{i \in P(\rho_\mu)} \zeta(s-i)^{\dim \rho_{\mu,i}},
\end{equation}
where $P(\rho_\mu)$ is the set of possible values of the principal
grading on $\rho_\mu$ and $\rho_{\mu,i}$ is the corresponding subspace
of $\rho_\mu$.

For example, if $\rho_\mu$ is the defining representation of $GL_n$,
then formula \eqref{prod of zeta1} reads
\begin{equation}    \label{prod of zeta}
\prod_{k=0}^{n-1} \zeta(s+k-(n-1)/2).
\end{equation}

\subsection{Geometrization of the orbital side}    \label{v sp}

Now we are ready to construct a geometrization of the orbital side of
the trace formula.

\medskip

Let us form the Cartesian square
\begin{equation}    \label{Cartesian}
\begin{CD}
{\mc M}_{d} @>{{\Delta}_\Hc}>> {\mc H}_{d} \\
@VV{{\mb p}_\Delta}V @VV{{\mb p}}V \\
\Bun_G @>{\Delta}>> \Bun_G \times
\Bun_G
\end{CD}
\end{equation}
where $\Delta$ is the diagonal morphism. Thus, ${\mc
  M}_{d}$ is the fiber product of $X^{(d)} \times \Bun_G$ and ${\mc
  H}_{d}$ with respect to the two morphisms to $X^{(d)} \times \Bun_G
\times \Bun_G$.

Let ${\mc K} = {\mc K}_{d,\rho}$ be a sheaf introduced in
\secref{functor Kdrho} and ${\mathbb K} = {\mathbb K}_{d,\rho}$ the
corresponding functor on the derived category $D(\Bun_G)$ of sheaves
on $\Bun_G$. The following discussion is applicable to more general
functors that are compositions of ${\mathbb K}_{d,\rho}$ and Hecke
functors ${\mathbb H}_{\rho_i,x_i}$ at finitely many points $x_i \in
|X|$ and their kernels (see \cite{FN}).

Let
$$
\ol{\mc K} = {\mb p}_*({\mc K})
$$
(note that ${\mathbf p}$ is proper over the support of ${\mc
  K}$). This is a sheaf on $\Bun_G \times \Bun_G$ which is the kernel
of the functor $\K$. Let $\ol{K}$ be the corresponding function on
$\Bun_G(k) \times \Bun_G(k)$.

Recall that
$$\Bun_G(k) = G(F) \bs G(\AD_F)/G(\OO_F).$$
The right hand side of \eqref{tr} may be rewritten as
\begin{equation}    \label{sum b}
\sum_{P \in \Bun_G(k)} \frac{1}{|\on{Aut}(V)|} \ol{K}(P,P)
\end{equation}
Using the Lefschetz formula for algebraic stacks developed by
K. Behrend \cite{Behrend} (see also \cite{BeDh}), we find that,
formally, the alternating sum of the traces of the arithmetic
Frobenius on the graded vector space
\begin{equation}    \label{rhs-new}
H^\bullet(\Bun_G,\Delta^!(\ol{\mc K})) = H^\bullet(\Bun_G,\Delta^!
{\mb p}_* (\ol{\mc K}))
\end{equation}
is equal (up to a power of $q$) to the sum \eqref{sum b}. Therefore
the vector space $H^\bullet(\Bun_G,\Delta^! {\mb p}_* (\ol{\mc K}))$
is a geometrization of the orbital (right hand) side of the trace
formula \eqref{tr}.

By base change,
\begin{equation}    \label{adjunction}
H^\bullet(\Bun_G,\Delta^! {\mb p}_*({\mc K})) =
H^\bullet(\Bun_G,{\mb p}_{\Delta *} \Delta_{\Hc}^!({\mc K})) =
H^\bullet({\mc M}_d,\Delta_{\Hc}^!({\mc K})).
\end{equation}
We will use the space
\begin{equation}    \label{geom orb side}
H^\bullet({\mc M}_d,\Delta_{\Hc}^!({\mc K}))
\end{equation}
as the geometrization of the orbital side.

Here is another way to express it: Let ${\mathbb D}$ be the Verdier
duality on ${\mc H}_{d,\mu}$. It follows from the construction and the
fact that ${\mathbb D}(\on{IC}({\mc H}_\mu)) \simeq \on{IC}({\mc
  H}_\mu)$ that
\begin{equation} \label{verdier}
{\mathbb D}({\mc K}_{d,\rho}) \simeq {\mc K}_{d,\rho}[2(d+\dim
\Bun_G)] (d+\dim \Bun_G).
\end{equation}
Therefore, up to a shift and Tate twist, the last space in
\eqref{adjunction} is isomorphic to
\begin{equation}    \label{comp supp}
H^\bullet({\mc M}_d,{\mathbb D}(\Delta_{\Hc}^*({\mathcal K}))) \simeq
H^\bullet_c({\mc M}_d, \Delta_{\Hc}^*({\mathcal K}))^*,
\end{equation}
where $H^\bullet_c(Z,{\mc F})$ is understood as $f_!({\mc F})$, where
$f: Z \to \on{pt}$. Here we use the results of Y. Laszlo and M. Olsson
\cite{LO} on the six operations on $\ell$-adic sheaves on algebraic
stacks and formula \eqref{verdier}.

\medskip

If $X$ is a curve over $\C$, then the vector space \eqref{geom orb
  side} still makes sense if we consider ${\mc K}$ as an object of
either the derived category of constructible sheaves or of ${\mc
  D}$-modules on $\Bun_G$.

\subsection{The moduli stack of $G$-pairs}    \label{orbital}

We now give a description of the stack ${\mc M}_d$ that is reminiscent
and closely related to the Hitchin moduli stack of Higgs bundles on
the curve $X$ \cite{Hit1}. We will also conjecture that
$\Delta_{\Hc}^!({\mc K}_{d,\rho})$ is a pure perverse sheaf on ${\mc
  M}_d$.

Recall that the sheaf ${\mc K}_{d,\rho}$ is supported on the substack
$\Hc_{d,\mu}$ of $\Hc$. Let ${\mc M}_{d,\mu}$ be the fiber product of
$X^{(d)} \times \Bun_G$ and ${\mc H}_{d,\mu}$ with respect to the two
morphisms to $X^{(d)} \times \Bun_G \times \Bun_G$. In other words, we
replace $\Hc_d$ by $\Hc_{d,\mu}$ in the upper right corner of the
diagram \eqref{Cartesian}. The sheaf $\Delta_{\Hc}^!({\mc
  K}_{d,\rho})$ is supported on ${\mc M}_{d,\mu} \subset {\mc M}_d$,
and hence the vector space \eqref{geom orb side} is equal to
\begin{equation}    \label{restr1}
H^\bullet({\mc M}_{d,\mu},\Delta_{\Hc}^!({\mc K}_{d,\rho})).
\end{equation}

In this section we show, following closely \cite{FN}, that the stack
${\mc M}_{d,\mu}$ has a different interpretation as a moduli stack of
objects that are closely related to {\em Higgs bundles}. More
precisely, we will have to define ``group-like'' versions of Higgs
bundles (we call them ``$G$-pairs''). The moduli spaces of (stable)
Higgs bundles has been introduced by Hitchin \cite{Hit1} (in
characteristic $0$) and the corresponding stack (in characteristic
$p$) has been used in \cite{Ngo:FL} in the proof of the fundamental
lemma. In addition, there is an analogue ${\mc A}_{d,\mu}$ of the
Hitchin base and a morphism $h_{d,\mu}: {\mc M}_{d,\mu} \to {\mc
  A}_{d,\mu}$ analogous to the Hitchin map. We hope that this Higgs
bundle-like realization of ${\mc M}_{d,\mu}$ and the morphism
$h_{d,\mu}$ can be used to derive the decompositions \eqref{Phi and
  Psi} and \eqref{sum2}. In the next few subsections we discuss this
in more detail.

\subsection{Definition of the moduli stack}    \label{def stack}

Let us assume that $G$ is split over $X$ and $\mu$ is a fixed dominant
coweight.  The groupoid ${\mc M}_{d,\mu}(k)$ classifies the triples
$$
(D,E,\varphi),
$$
where $D = \sum_i n_i [x_i]\in X^{(d)}$ is an effective divisor of degree $d$,
$E$ is a principal $G$-bundle on a curve $X$, and $\varphi$ is a section
of the adjoint group bundle
$$
\on{Ad}(E) = E \underset{G}\times G
$$
(with $G$ acting on the right $G$ by the adjoint action) on $X {-}
\on{supp}(D)$, which satisfies the local conditions
\begin{equation}\label{invariant 4}
{\rm inv}_{x_i}(\varphi) \leq n_i\mu
\end{equation}
at $D$. Since we have defined $\Hc_{d,\mu}$ as the image of
$\Hc^d_\mu$ in $\Hc^\BD_d$, it is not immediately clear how to make
sense of these local conditions over an arbitrary base (instead of
$\on{Spec}(k)$). There is in fact a functorial description of
$\Hc_{d,\mu}$ and of ${\mc M}_{d,\mu}$ that we will now explain.

We will assume that $G$ is semi-simple and simply-connected. The
general case is not much more difficult.  Let
$\omega_1,\ldots,\omega_r$ denote the fundamental weights of $G$ and
$$\rho_{\omega_i}:G\to \GL(V_{\omega_i})$$ the Weyl modules of highest
weight $\omega_i$. Using the natural action of $G$ on ${\rm End}
(V_{\omega_i})$, we can attach to any $G$-principal bundle $E$ on $X$
the vector bundle
\begin{equation}
\on{End}_{\omega_i}(E)=E \underset{G}\times {\rm End}(V_{\omega_i}).
\end{equation}
The section $\varphi$ of $\on{Ad}(E)$ on $X{-}\on{supp}(D)$ induces a
section $\on{End}_{\omega_i}(\varphi)$ of the vector bundle
$\on{End}_{\omega_i}(E)$ on $X{-}\on{supp}(D)$. The local conditions
(\ref{invariant 4}) are equivalent to the property that for all $i$,
$\on{End}_{\omega_i}(\varphi)$ may be extended to a section
$$\varphi_i\in \on{End}_{\omega_i}(E)\otimes_{\Oc_X} \Oc_X(\langle
\mu,\omega_i\rangle D).$$ Though
the $\varphi_i$ determine $\varphi$, we will keep $\varphi$ in the
notation for convenience.

Thus, we obtain a provisional functorial description of ${\mc
M}_{d,\mu}$ as the stack classifying the data
\begin{equation}\label{phi phi}
(D,E,\varphi,\varphi_i)
\end{equation}
with $D\in X^{(d)}$, $E\in \Bun_G$, $\varphi$ is a section of $\on{Ad}(E)$
on $X{-}\on{supp}(D)$, $\varphi_i$ are sections of
$\on{End}_{\omega_i}(E)\otimes_{\Oc_X} \Oc_X(\langle
\mu,\omega_i\rangle D)$ over $X$ such that
$$\varphi_i|_{X{-}\on{supp}(D)}=\on{End}_{\omega_i}(\varphi).$$

Sometimes it will be more convenient to package the data
$(\varphi,\varphi_i)$ as a single object $\wt\varphi$ which has values
in the closure of $$(t_i \rho_{\omega_i}(g))_{i=1}^r \subset
\prod_{i=1}^r \on{End}(V_{\omega_i}).$$ where $g\in G$ and
$t_1,\ldots,t_r \in \Gm$ are invertible scalars. This way Vinberg's
semi-group \cite{Vinberg} makes its appearance naturally in the
description of ${\mc M}_{d,\mu}$.

\subsection{Comparison with the Hitchin moduli
  stack}    \label{comparison}

It is instructive to note that the stack ${\mc M}_{d,\mu}$ is very
similar to the moduli stacks of Higgs bundles (defined originally by
Hitchin \cite{Hit1} and considered, in particular, in
\cite{Ngo:Endo,Ngo:FL}). The latter stack -- we will denote it by
${\mc N}_D$ -- also depends on the choice of an effective divisor
$$
D = \sum_i n_i [x_i]
$$
on $X$ and classifies pairs $(E,\phi)$, where $E$ is again a
$G$-principal bundle on $X$ and $\phi$ is a section of the adjoint
{\em vector bundle}
$$
\on{ad}(E) = E \underset{G}\times \g
$$
(here $\g = \on{Lie}(G)$) defined on $X{-}\on{supp}(D)$, which is
allowed to have a pole of order at most $n_i$ at $x_i$. In other
words,
$$
\phi \in H^0(X,\on{ad}(E) \otimes {\mc O}_X(D)).
$$
This $\phi$ is usually referred to as a {\em Higgs field}.

In both cases, we have a section which is regular almost everywhere,
but at some (fixed, for now) points of the curve these sections are
allowed to have singularities which are controlled by a divisor. In
the first case we have a section $\varphi$ of adjoint group bundle
$\on{Ad}(E)$, and the divisor is $D \cdot \mu$, considered as an
effective divisor with values in the lattice of integral weights of
$\LG$. In the second case we have a section $\phi$ of the adjoint Lie
algebra bundle $\on{ad}(E)$, and the divisor is just the ordinary
effective divisor.

An important tool in the study of the moduli stack ${\mc N}_D$ is the
{\em Hitchin map} \cite{Hi2} from ${\mc N}_D$ to an affine space
$$
{\mc A}_D  \simeq \bigoplus_i
H^0(X,{\mc O}_X((m_i+1)D),
$$
where the $m_i$'s are the exponents of $G$. It is obtained by, roughly
speaking, picking the coefficients of the characteristic polynomial of
the Higgs field $\phi$ (this is exactly so in the case of $GL_n$; but
one constructs an obvious analogue of this morphism for a general
reductive group $G$, using invariant polynomials on its Lie
algebra). A point $a \in A_E$ then records a stable conjugacy class in
$\g(F)$, where $F$ is the function field, and the number of points in
the fiber over $a$ is related to the corresponding orbital integrals
in the Lie algebra setting (see \cite{Ngo:Endo,Ngo:FL}).

More precisely, ${\mc A}_D$ is the space of section of the bundle
$${\mathfrak t}/W \underset{\Gm}\times {\mc O}_X(D)^\times$$ obtained by
twisting ${\mathfrak t}/W = \on{Spec}(k[{\mathfrak t}]^W)$, equipped with
the $\Gm$-action inherited from ${\mathfrak t}$, by the $\Gm$-torsor
${\mc O}_X(D)^\times$ on $X$ attached to the line bundle ${\mc
O}_X(D)$. Recall that $k[\mathfrak t/W]$ is a polynomial algebra with
homogeneous generators of degrees $d_1+1,\ldots,d_r+1$.

\subsection{Analogue of the Hitchin map for
  $G$-pairs}    \label{analogue}

In our present setting, we will have to replace ${\mathfrak t}/W$ by
$T/W$.  Recall that we are under the assumption that $G$ is
semi-simple and simply-connected. First, recall the isomorphism of
algebras
$$k[G]^G=k[T]^W=k[T/W].$$ It then follows from \cite{Bourbaki},
Th. VI.3.1 and Ex. 1, that $k[G]^G$ is a polynomial
algebra generated by the functions
$$g\mapsto \on{tr}(\rho_{\omega_i}(g)),$$ where $\omega_1,\ldots,
\omega_r$ are the fundamental weights of $G$.

For a fixed divisor $D$, the analogue of the Hitchin map for ${\mc
M}_{d,\mu}(D)$ (the fiber of ${\mc M}_{d,\mu}$ over $D$) is the
following map:
\begin{equation}
{\mc M}_{d,\mu}(D) \to \bigoplus_{i=1}^r H^0(X,\Oc_X(\langle
\mu,\omega_i \rangle D ))
\end{equation}
defined by attaching to $(D,E,\varphi,\varphi_i)$ the collection of traces
$${\rm tr}(\varphi_i)\in H^0(X,\Oc_X(\langle \mu,\omega_i \rangle D )).$$
By letting $D$ vary in $X^{(d)}$, we obtain a fibration
$$h_{d,\mu}:{\mc M}_{d,\mu} \to {\mc A}_{d,\mu}$$ where ${\mc
A}_{d,\mu}$ is a vector bundle over $X^{(d)}$ with the fiber
$\bigoplus_{i=1}^r H^0(X,\Oc_X(\langle \mu,\omega_i \rangle D ))$ over
an effective divisor $D\in X^{(d)}$.

The morphism $h_{d,\mu}:\M_{d,\mu} \to \Ac_{d,\mu}$ is very similar to
the Hitchin fibration.

Recall that according to \secref{v sp}, the cohomology of $\M_{d,\mu}$
with coefficients in $\Delta^!({\mc K}_{d,\rho})$ is a geometrization
of the orbital side of the trace formula \eqref{tr}. This side of the
trace formula may be written as a sum \eqref{sum over gamma} of
orbital integrals. To obtain a geometrization of an {\em individual}
orbital integral appearing in this sum, we need to take the cohomology
of the restriction of $\Delta^!({\mc K}_{d,\rho})$ to the
corresponding {\em fiber} of the map $h_{d,\mu}$ (see \cite{FN},
Section 4.3). Thus, the picture is very similar to that described in
\cite{Ngo:FL} in the Lie algebra case.

\medskip

In \cite{FLN} we considered the trace formula in the general ramified
setting. Like in the above discussion, we obtained a useful
interpretation of the stable conjugacy classes in $G(F)$ as $F$-points
of the Hitchin base described above (in \cite{FLN} we called it
``Steinberg--Hitchin base''). In particular, if the group $G$ is
simply-connected, this base has the structure of a vector space over
$F$. The orbital side of the trace formula may therefore be written as
the sum over points of this vector space, and we can apply the Poisson
summation formula to it (actually, in order to do this we need to
overcome several technical problems). We used this to isolate the
contribution of the trivial representation to the trace formula, which
appears as the contribution of the point $0$ in the dual sum, see
\cite{FLN}. The question now is how to separate the contributions of
other non-tempered automorphic representations of $G(\AD_F)$.

\subsection{Example of $GL_2$}    \label{ex gl2}

Let us specialize now to the case when $G=GL_2$ and $\mu=(2,0)$. Here
we need to make one adjustment of the general set-up; namely, instead
of $$GL_2(F) \bs GL_2({\mathbb A}_F/GL_2({\mathcal O}_F),$$ which is a
union of infinitely many components (hence non-compact), we will
consider a compact quotient $$GL_2(F) a^{\Z} \bs GL_2({\mathbb
  A})/GL_2({\mathcal O}),$$ where $a$ is an element in the center
$Z({\mathbb A}_F)$. We choose as $a$ the element equal to $1$ at all $x
\in X$, except for a fixed point that we denote by $\infty$. We set
$a$ equal to a uniformizer $t_\infty$ at $\infty$. Thus, we do not
allow ramification at $\infty$, but rather restrict ourselves to
automorphic representations on which this $a$ acts trivially, in order
to make sure that our integrals converge.

Since we now identify the rank two bundle ${\mc M}$ with ${\mc
M}(k[\infty])$ for all $k \in \Z$, which is an operation shifting the
degree of the bundle by an even integer $2k$, this stack has two
connected components, corresponding to the degree of the bundle ${\mc
M}$ modulo $2$.

We denote them by ${\mc H}^{\Delta,0}_{d,\mu,\infty}$ and ${\mc
  H}^{\Delta,1}_{d,\mu,\infty}$. Let us focus on the first one and set
$\mu=(2,0)$ (so that $\rho$ is the symmetric square of the defining
vector representation).

\medskip

This ${\mc H}^{\Delta,0}_{d,(2,0),\infty}$ is the moduli stack of the
following data:
$$
(D,{\mc M},s: {\mc M} \hookrightarrow {\mc M}(d[\infty])),
$$
where $D$ is an effective divisor on $X$ of degree $d$,
\begin{equation}    \label{D}
D = \sum_{i=1}^\ell n_i [x_i],
\end{equation}
${\mc M}$ is a rank two vector bundle on $X$ of degree $0$, and $s$ is
an injective map such that
$$
{\mc M}(d[\infty])/s({\mc M}) \simeq \bigoplus_i {\mc T}_{x_i},
$$
where the ${\mc T}_{x_i}$ are torsion sheaves supported at $x_i$. For
instance, if all points in $D$ have multiplicity $1$, that is, $n_i=1$
for all $i$, then ${\mc T}_{x_i} = {\mc O}_{2x_i}$ or ${\mc O}_{x_i}
\oplus {\mc O}_{x_i}$. But if $n_i>1$, then it splits in one of the
following ways:
$$
{\mc T}_{x_i} = {\mc O}_{k_i x_i} \oplus {\mc O}_{(2n_i-k_i)x_i},
\qquad k_i=0,\ldots,n_i
$$
(these correspond to the strata in the affine Grassmannian which
lie in the closure of $\on{Gr}_{n_i\mu}$).

\medskip

We define an analogue of
the Hitchin map by taking the
coefficients of the characteristic polynomial of $\varphi$:
\begin{equation*}
\begin{array}{cccccc}
{\mc H}^{\Delta,0}_{d,\mu,\infty} &\to& & H^0(X,{\mc
  O}_X(d[\infty])) &\oplus& H^0(X,{\mc O}_X(2d[\infty])), \\
(D,{\mc M},\varphi) &\mapsto& &  (\on{Tr} \varphi & , & \on{det}
  \varphi).
\end{array}
\end{equation*}
We denote the first component of this map by $b_1$ and the second
component by $b_2$.

What are the allowed values of $b_1, b_2$? By construction, $b_2$ is a
section of ${\mc O}_X(2d[\infty])$, whose divisor of zeros is $2D$.

\begin{lem} Any such section $b_2$ of ${\mc O}_X(2d[\infty])$ has the
  form $b_2 = \eta^2$, where $\eta$ is a non-zero section of ${\mc L}
  \otimes {\mc O}_X(d[\infty])$, where ${\mc L}$ is a square root of
  the trivial line bundle (that is, ${\mc L}^{\otimes 2} \simeq {\mc
  O}_X$) having the divisor of zeros equal to $D$.
\end{lem}

\begin{proof} Consider the equation $y^2=b_2$, where $y$ is a section
of ${\mc O}_X(d[\infty])$. We interpret the solution as a curve in the
total space of the line bundle ${\mc O}_X(d[\infty])$ over $X$, which
is a ramified double cover $C$ of $X$, with the divisor of
ramification equal to $D$ (indeed, away from the points of $D$ the
value of $b_2$ is non-zero, hence there are two solutions for $y$, but
at the points of $D$ we have $b_2=0$, so there is only one solution
$y=0$). Moreover, near each point $x_i$ in $D$ (see formula \eqref{D})
this curve $C$ is given by the equation $y^2=t^{2n_i}$, and hence has
two branches. Thus, the normalization of $C$ is an {\em unramified}
double cover $\wt{C}$ of $X$. It gives rise to a square root ${\mc L}$
of ${\mc O}_X$.\footnote{Namely, we interpret $\wt{C}$ as a principal
$\Z_2=\{ \pm 1 \}$-bundle over $X$, and take the line bundle
associated to the non-trivial one-dimensional representation of
$\Z_2$.}

Clearly, the pull-back of ${\mc L}$ to $\wt{C}$ is canonically
trivialized. Therefore a section of ${\mc L} \otimes {\mc
O}_X(d[\infty])$ over $X$ is the same thing as a section of the
pull-back of ${\mc O}_X(d[\infty])$ to $\wt{C}$ which is
anti-invariant under the natural involution $\tau: \wt{C} \to
\wt{C}$. But $y$ gives us just such a section (indeed, it assigns to
each point $c$ of $\wt{C}$ a vector in the fiber of ${\mc
O}_X(d[\infty])$ over the image of $c$ in $X$ -- namely, the image of
$c$ in $C$, which is viewed as a point in the total space of ${\mc
O}_X(d[\infty])$; if there are two points $c_1, c_2$ in $\wt{C}$ lying
over the same point of $X$, then these vectors are obviously opposite
to each other). Hence we obtain a section $\eta$ of ${\mc L} \otimes
{\mc O}_X(d[\infty])$ over $X$. It follows that $\eta^2 = b_2$.
\end{proof}

For each square root ${\mc L}$ of the trivial line bundle on $X$,
consider the map
$$
H^0(X,{\mc L} \otimes {\mc O}_X(d[\infty]))^\times \to H^0(X,{\mc
  O}(2d[\infty]))
$$
sending
$$
\eta \mapsto \eta^2.
$$
Denote the image by $B_{\mc L}$. It is isomorphic to the quotient of
$H^0(X,{\mc L} \otimes {\mc O}_X(d[\infty]))^\times$ by the involution
acting as $\eta \mapsto -\eta$.

We conclude that $b_2$ could be any point in
$$
\bigsqcup_{\mc L} B_{\mc L} \subset H^0(X,{\mc O}(2d[\infty])).
$$
Thus, we have described the possible values of $b_2$, which is the
determinant of $\gamma$. This is a subset in a vector space, which has
components labeled by ${\mc L}$, which we view as points of order two
in $\on{Jac}(\Fq)$. We denote the set of such points by
$\on{Jac}_2$ and its subset corresponding to non-trivial ${\mc L}$ by
$\on{Jac}_2^\times$.

\bigskip

What about $b_1$, which is the trace of $\gamma$? It is easy to see
that $b_1$ may take an arbitrary value in $H^0(X,{\mc
O}_X(d[\infty]))$. Thus, we find that the image of ${\mc
H}^{\Delta,0}_{d,(2,0),\infty}$ in
$$
H^0(X,{\mc O}_X(d[\infty])) \oplus H^0(X,{\mc O}_X(2d[\infty]))
$$
under the map
$$
\gamma \mapsto (\on{tr} \gamma,\on{det} \gamma) = (b_1,b_2)
$$
is equal to
\begin{equation}    \label{base}
H^0(X,{\mc O}_X(d[\infty])) \times \left( \bigsqcup_{\mc L \in
  \on{Jac}_2} B_{\mc L} \right).
\end{equation}
In other words, {\em any} $\gamma$ that has trace and determinant of
this form corresponds to a point of ${\mc
  H}^{\Delta,0}_{d,(2,0),\infty}$.

\bigskip

What does the fiber $F_{b_1,b_2}$ look like? Using the theory of
Hitchin fibrations, we describe it as follows: consider the
characteristic polynomial of $\gamma$:
\begin{equation}    \label{wtX}
z^2 - b_1 z + b_2 = 0.
\end{equation}
This equation defines a curve $\wt{X}_{b_1,b_2}$ in the total space of
the line bundle ${\mc O}_X(d[\infty])$ over $X$, which is called the
{\em spectral curve} associated to $b_1, b_2$. If this curve is
smooth, then the fiber $F_{b_1,b_2}$ is just the Jacobian of
$\wt{X}_{b_1,b_2}$. What is the rank two bundle ${\mc M}$
corresponding to a point of the Jacobian of $\wt{X}_{b_1,b_2}$? This
point is a line bundle on $\wt{X}_{b_1,b_2}$. Take its push-forward to
$X$. This is our ${\mc M}$. It comes equipped with a map $s: {\mc M}
\to {\mc M}(d[\infty])$ whose trace and determinant are $b_1$ and
$b_2$, respectively. The claim is that any point in $F_{b_1,b_2}$ may
be obtained this way, as a push-forward of a line bundle on
$\wt{X}_{b_1,b_2}$.

If $\wt{X}_{b_1,b_2}$ is singular, then $F_{b_1,b_2}$ is the
compactification of the Jacobian of $\wt{X}_{b_1,b_2}$ known as the
moduli space of torsion sheaves on $\wt{X}_{b_1,b_2}$ of generic rank
one. It contains the Jacobian of $\wt{X}_{b_1,b_2}$ as an open dense
subset.

\subsection{General case}    \label{conj gen}

In \cite{FN}, Ng\^o and I outlined the geometric properties of the
Hitchin fibration (used in \cite{Ngo:FL} to prove the
fundamental lemma) that can be carried over to our new situation.

Our goal is to understand the cohomology \eqref{restr1}. According to
formula \eqref{verdier}, up to a shift and Tate twist, it is
isomorphic to the dual of the cohomology with compact support of
$\Delta_{\Hc}^*({\mc K}_{d,\rho})$. The following conjecture seems to
be necessary to have a chance to approach this cohomology using known
methods (such as the decomposition theorem and Ng\^o's theorem about
push-forwards of perverse sheaves \cite{Ngo:FL}, which he used in
the proof of the fundamental lemma). It probably tells us something
important about the geometric trace formula and the categorical
Langlands correspondence (the subjects discussed in \secref{leap}),
but it's not clear to me yet what it is.

\begin{conj}[\cite{FN}]    \label{perverse}
The restriction to the diagonal $\Delta_{\Hc}^*({\mc K}_{d,\rho})$ is
a pure perverse sheaf.
\end{conj}

Assume that $G$ is semisimple. As in \cite{CL}, there exists a open
substack $\M_{d,\mu}^{\rm st}$ of $\M_{d,\mu}$ that is proper over
$\Ac_{d,\mu}$.  This open substack depends on the choice of a
stability condition. However, its cohomology should be independent of
this choice. Moreover, there exists an open subset $\Ac_{d,\mu}^{\rm
  ani}$ of $\Ac_{d,\mu}$ whose $\bar k$-points are the pairs $(D,b)$
such that as an element of $(T/W)(F\otimes_{k} \bar k)$, $b$
corresponds to a regular semisimple and anisotropic conjugacy class in
$G(F\otimes_{k} \bar k)$. The preimage $\M_{d,\mu}^{\rm ani}$ of
$\Ac_{d,\mu}^{\rm ani}$ is contained in $\M_{d,\mu}^{\rm st}$ for all
stability conditions. In particular, the morphism $\M_{d,\mu}^{\rm
  ani}\to \Ac_{d,\mu}^{\rm ani}$ is proper.
 
To compute the cohomology with compact support of
$\Delta_{\Hc}^*({\mc K}_{d,\rho})$, we consider the sheaf $(h_{d,\mu}^{\rm
  st})_{!}\Delta_{\Hc}^*({\mc K}_{d,\rho}) = (h_{d,\mu}^{\rm
  st})_{*}\Delta_{\Hc}^*({\mc K}_{d,\rho})$ on ${\mc A}_{d,\mu}^{\rm st}$
(recall that $h_{d,\mu}^{\rm st}$ is proper). By Deligne's purity
theorem, \conjref{perverse} implies that $(h_{d,\mu}^{\rm
  st})_{*}\Delta_{\Hc}^*({\mc K}_{d,\rho})$ is a pure complex. Hence,
geometrically, it is isomorphic to a direct sum of shifted simple
perverse sheaves.

In \cite{FN}, Ng\^o and I presented some conjectures describing the
structure of $(h_{d,\mu}^{\rm st})_{*}\Delta_{\Hc}^*({\mc
  K}_{d,\rho})$, both in the general case and in the specific example
of $G=SL_2$ and $H$ a twisted torus.

\section{The geometric trace formula}    \label{leap}

In \secref{v sp} we have interpreted geometrically the right hand side
of the trace formula \eqref{tr}. Now we turn to the left hand
(spectral) side. We will follow closely the exposition of \cite{FN}.
We begin by rewriting the spectral side as a sum over the
homomorphisms $W_F \times \SL_2 \to {}\LG$ in \secref{lhs tr}. It is
tempting to interpret this sum using the Lefschetz fixed point
formula. After explaining the difficulties in doing so directly in
\secref{lfpf}, we try a different approach in \secref{back}. Namely,
we use the categorical Langlands correspondence of \secref{na fm} to
construct a vector space that should be isomorphic to the vector space
of \secref{v sp} that we proposed as the geometrization of the orbital
side of the trace formula. This gives us the sought-after
geometrization of the spectral side, and we declare the conjectural
isomorphism between the two (implied by the categorical Langlands
correspondence) as the ``geometric trace formula''. In \secref{abl
  fixed} we give an heuristic explanation why this vector space should
indeed be viewed as the geometrization of the spectral side.

\subsection{The left hand side of the trace formula}    \label{lhs tr}

We will assume that $G$ is split. As we discussed in \secref{decomp
  trace}, the $L$-packets of (unramified) irreducible automorphic
representations should correspond to (unramified) homomorphisms $W_F
\times \SL_2 \to {}\LG$. Assuming this conjecture and ignoring for the
moment the contribution of the continuous spectrum, we may write the
left hand side of \eqref{general} as
\begin{equation}    \label{trace1}
\on{Tr} {\mb K} = \sum_\sigma m_\sigma N_\sigma,
\end{equation}
where $\sigma$ runs over the unramified homomorphisms $W_F \times \SL_2
\to {}\LG$, $m_\sigma$ is the multiplicity of the irreducible
automorphic representation unramified with respect to $G(\OO_F)$ in the
$L$-packet corresponding to $\sigma$, and $N_\sigma$ is the eigenvalue
of the operator ${\mb K}$ on an unramified automorphic function
$f_\sigma$ on $\Bun_G(\Fq)$ corresponding to a spherical vector in
this representation:
$$
{\mb K} \cdot f_\sigma = N_\sigma f_\sigma.
$$

Thus, recalling \eqref{sum b}, the trace formula \eqref{general}
becomes
\begin{equation}    \label{classical}
\sum_{\sigma: W_F \times \SL_2 \to {}\LG} m_\sigma N_\sigma =
\sum_{P \in \Bun_G(\Fq)} \frac{1}{|\on{Aut}(V)|} \ol{K}(P,P).
\end{equation}

Consider, for example, the case of ${\mb K} = {\mb H}_{\rho,x}$, the
Hecke operator corresponding to a representation $\rho$ of $\LG$ and
$x \in |X|$. Then, according to formula \eqref{matching arthur},
\begin{equation}    \label{Nsigma}
N_\sigma = \on{Tr} \left(\sigma\left(\begin{pmatrix}
        q_x^{1/2} & 0 \\ 0 & q_x^{-1/2} \end{pmatrix} \times
      \on{Fr}_x\right),\rho \right).
\end{equation}
The ${\mb K}$ are generated by the Hecke operators ${\mb
  H}_{\rho,x}$. Therefore the eigenvalue $N_\sigma$ for such an
operator ${\mb K}$ is expressed in terms of the traces of
$\sigma(\on{Fr}_x)$ on representations of $\LG$.

\subsection{Lefschetz fixed point formula
  interpretation}    \label{lfpf}

It is tempting to try to interpret the left hand side of
\eqref{classical} as coming from the Lefschetz trace formula for the
trace of the Frobenius on the cohomology of an $\ell$-adic sheaf on a
moduli stack, whose set of $k$-points is the set of $\sigma$'s.
Unfortunately, such a stack does not exist if $k$ is a finite field
$\Fq$ (or its algebraic closure).  On the other hand, if $X$ is over
$\C$, then there is an algebraic stack $\Loc_{\LG}$ of (de Rham)
$\LG$-local systems on $X$; that is, $\LG$-bundles on $X$ with flat
connection. But in this case there is no Frobenius acting on the
cohomology whose trace would yield the desired number (the left hand
side of \eqref{classical}).  Nevertheless, we will define a certain
vector space (when $X$ is over $\C$), which we will declare to be a
geometrization of the left hand side of \eqref{classical} (we will
give an heuristic explanation for this in \secref{abl fixed}). We will
then conjecture that this space is isomorphic to \eqref{geom orb side}
-- this will be the statement of the ``geometric trace formula'' that
we propose in this paper.

\medskip

Let us first consider the simplest case of the Hecke operator ${\mb
  K}={\mb H}_{\rho,x}$. In this case the eigenvalue $N_\sigma$ is
given by formula \eqref{Nsigma}, which is essentially the trace of the
Frobenius of $x$ on the vector space which is the stalk of the local
system on $X$ corresponding to $\sigma$ and $\rho$. These vector
spaces are fibers of a natural vector bundle on $X \times \Loc_{\LG}$
(when $X$ is defined over $\C$).

Indeed, we have a tautological $\LG$-bundle ${\mc T}$ on $X \times
\on{Loc}_{\LG}$, whose restriction to $X \times \sigma$ is the
$\LG$-bundle on $X$ underlying $\sigma \in \on{Loc}_{\LG}$. For a
representation $\rho$ of $\LG$, let ${\mc T}_\rho$ be the associated
vector bundle on $X \times \on{Loc}_{\LG}$. It then has a partial flat
connection along $X$. Further, for each point $x \in |X|$, we denote
by ${\mc T}_x$ and ${\mc T}_{\rho,x}$ the restrictions of ${\mc T}$
and ${\mc T}_\rho$, respectively, to $x \times \on{Loc}_{\LG}$.

It is tempting to say that the geometrization of the left hand
side of \eqref{classical} in the case ${\mb K}={\mb H}_{\rho,x}$ is
the cohomology
$$
H^\bullet(\Loc_{\LG},{\mc T}_{\rho,x}).
$$
However, this would only make sense if the vector bundle ${\mc
  T}_{\rho,x}$ carried a flat connection (i.e., a ${\mc D}$-module
structure) and this cohomology was understood as the de Rham
cohomology (which is the analogue of the \'etale cohomology of an
$\ell$-adic sheaf that we now wish to imitate when $X$ is defined over
$\C$).

Unfortunately, ${\mc T}_{\rho,x}$ does not carry any natural
connection. Hence we search for another approach. Somewhat
surprisingly, it is provided by the categorical Langlands
correspondence of \secref{cat ver}.

\subsection{Geometrization of the spectral side of the trace
  formula}    \label{back}

The categorical Langlands correspondence \eqref{na fm} should yield
important information not only at the level of objects, but also at
the level of morphisms.  In particular, if we denote the equivalence
going from left to right in the above diagram by $C$, then we obtain
the isomorphisms
\begin{equation}    \label{RHom}
\on{RHom}_{D(\on{Loc}_{\LG})}({\mc F}_1,{\mc F}_2) \simeq 
\on{RHom}_{D(\Bun_G)}(C({\mc F}_1),C({\mc F}_2))
\end{equation}
and
\begin{equation}    \label{RHom1}
\on{RHom}({\mathbb F}_1,{\mathbb F}_2) \simeq 
\on{RHom}(C({\mathbb F}_1),C({\mathbb F}_2)),
\end{equation}
where ${\mathbb F}_1$ and ${\mathbb F}_2$ are arbitrary two functors
acting on the category $D(\on{Loc}_{\LG})$. We will use the latter to
produce the geometric trace formula.

Recall also the Wilson functors $\W_{\rho,x}$ from \secref{cat ver}
and the compatibility \eqref{WH} between the Wilson and Hecke
functors under $C$.

We now construct a functor $\W_{d,\rho}$ on the category of
$\OO$-modules on $\Loc_{\LG}$ from the Wilson functors in the same way
as we build the functor ${\mathbb K}_{d,\rho}$ from the Hecke
functors. Informally speaking, it is the integral over all effective
divisors $D = \sum_i n_i [x_i] \in X^{(d)}$ of $\prod_i
\W_{\rho^{(n_i)},x_i}$. The precise definition is given in \cite{FN},
Section 5.3. We have
$$
\W_{d,\rho}({\mc F}) = {\mc F}_{d,\rho} \otimes p_2^*({\mc F}),
$$
where ${\mc F}_{d,\rho}$ is a certain $\OO$-module on $\Loc_{\LG}$
explicitly defined in terms of the tautological vector bundles ${\mc
  T}_{\rho^{(n)},x}$, and $p_2$ is the projection $X \times \Loc_{\LG}
\to \Loc_{\LG}$.

The functor $\W_\rho$ may also be written as the ``integral
transform'' functor corresponding to the $\OO$-module $\Delta_*({\mc
  F}_{d,\rho})$ on $\on{Loc}_{\LG} \times \on{Loc}_{\LG}$,
where $\Delta$ is the diagonal embedding of $\on{Loc}_{\LG}$:
$$
{\mc F} \mapsto q_*(q'{}^*({\mc F}) \otimes \Delta_*({\mc
  F}_{d,\rho}))),
$$
where $q$ and $q'$ are the projections onto $X \times \on{Loc}_{\LG}$
(the first factor) and $\on{Loc}_{\LG}$ (the second factor),
respectively.

Consider now the isomorphism \eqref{RHom1} in the case ${\mathbb F}_1
= \on{Id}$ and ${\mathbb F}_2 = \W_{d,\rho}$.  It follows from the
construction of the functors $\W_{d,\rho}$ and $\H_{d,\rho}$ and
formula \eqref{WH}. Hence \eqref{RHom1} gives us an isomorphism
$$
\on{RHom}(\on{Id},\W_{d,\rho}) \simeq \on{RHom}(\on{Id},\K_{d,\rho}).
$$

We should also have an isomorphism of the RHoms's of the kernels
defining these functors. On the right hand side this is the RHom
$$
\on{RHom}(\Delta_!(\underline{{\mc C}}),\ol{\mc K}_{d,\rho}),
$$
where $\underline{\mc C}$ is the constant sheaf on $\Bun_G$ and
$\ol{\mc K}_{d,\rho} = {\mathbf p}_*({\mc K}_{d,\rho})$ (see \secref{v
  sp}) in the derived category of ${\mc D}$-modules on $\Bun_G \times
\Bun_G$.

On the left hand side of the categorical Langlands correspondence, the
kernel corresponding to $\W_{d,\rho}$ is just $\Delta_*({\mc
  F}_{d,\rho})$, supported on the diagonal in $\on{Loc}_{\LG} \times
\on{Loc}_{\LG}$.

Thus, we find that the categorical version of the geometric Langlands
correspondence should yield the following isomorphism
\begin{equation}    \label{geometric}
\on{RHom}(\Delta_*(\OO),\Delta_*({\mc F}_{d,\rho})) \simeq
\on{RHom}(\Delta_!(\underline{{\mc C}}),\ol{\mc K}_{d,\rho}).
\end{equation}

\medskip

By adjunction, the right hand side of \eqref{geometric} is isomorphic
to
$$
H^\bullet(\Bun_G,\Delta^! \ol{\mc K}) = H^\bullet(\Bun_G,\Delta^! {\mb
  p}_*({\mc K})) = H^\bullet({\mc M}_d,\Delta_{\Hc}^!({\mc
  K}_{d,\rho}))
$$
(see formula \eqref{adjunction}). The last space is the vector space
\eqref{geom orb side} that we proposed as a geometrization of the
orbital side of the trace formula!

Hence the left hand side of \eqref{geometric} should be the
sought-after geometrization of the spectral side of the trace formula.
Let us rewrite it using adjunction as follows:
$$
\on{RHom}(\Delta^* \Delta_*(\OO),{\mc F}_{d,\rho}).
$$
Then we obtain the following isomorphism, which we conjecture as a
{\em geometric trace formula}.

\begin{conj}    \label{gtf conj}
There is an isomorphism of vector spaces:
\begin{equation}    \label{more general}
\on{RHom}(\Delta^* \Delta_*(\OO),{\mc F}_{d,\rho}) \simeq
H^\bullet(\Bun_G,\Delta^!(\ol{\mc K}_{d,\rho})).
\end{equation}
\end{conj}

\medskip

Thus, starting with the categorical Langlands correspondence, we have
arrived at what we propose as a geometrization of the trace formula.

We have a similar conjecture for more general functors $\K$ which are
compositions of $\K_{d,\rho}$ and Hecke functors at finitely many
points of $X$ (see \cite{FN}).

\subsection{Connection to the Atiyah--Bott-Lefschetz fixed point
  formula}    \label{abl fixed}

In \cite{FN} we gave the following heuristic explanation why we should
think of the space \eqref{more general} as a geometrization of
the sum appearing on the left hand side of \eqref{classical}.

Let ${\mc F}$ be a coherent sheaf on $\on{Loc}_{\LG}$ built from the
vector bundles ${\mc T}_{\rho,x}, x \in |X|$ (like ${\mc F}_{d,\rho}$
constructed below). We would like to interpret taking the trace of the
Frobenius on the (coherent!) cohomology $H^\bullet(\on{Loc}_{\LG},{\mc
  F})$ as the sum over points $\sigma$ of $\on{Loc}_{\LG}$, which we
think of as the fixed points of the Frobenius automorphism acting on a
moduli of homomorphisms
$$
\pi_1(X \underset{k}\otimes \ol{k}) \to {}\LG.
$$

Recall the Atiyah--Bott--Lefshetz fixed point formula\footnote{It is a
  great pleasure to do this here, because all of them were AMS
  Colloquium Lecturers in the past.} \cite{AB} (see also
\cite{Illusie}, \S 6). Let $M$ be a smooth proper scheme, $V$ a vector
bundle on $M$, and $\V$ the coherent sheaf of sections of $V$. Let $u$
be an automorphism acting on $M$ with isolated fixed points, and
suppose that we have an isomorphism $\gamma: u^*(\V) \simeq \V$. Then
\begin{equation}    \label{l factor}
\on{Tr}(\gamma,H^\bullet(M,\V)) = \sum_{p \in M^u}
\frac{\on{Tr}(\gamma,\V_p)}{\on{det}(1-\gamma,T^*_p M)},
\end{equation}
where
$$
M^u = \Gamma_u \underset{M \times M}\times \Delta
$$
is the set of fixed points of $u$, the fiber product of the graph
$\Gamma_u$ of $u$ and the diagonal $\Delta$ in $M \times M$, which we
assume to be transversal to each other (as always, the left hand side
of \eqref{l factor} stands for the {\em alternating} sum of traces on
the cohomologies).

Now, if we take the cohomology not of $\V$, but of the tensor
product $\V \otimes \Omega^\bullet(M)$, where $\Omega^\bullet(M)
= \Lambda^\bullet(T^*(M))$ is the graded space of differential forms,
then the determinants in the denominators on the right hand side of
formula \eqref{l factor} will get canceled and we will obtain the
following formula:
\begin{equation}    \label{l factor1}
\on{Tr}(\gamma,H^\bullet(M,\V \otimes \Omega^\bullet(M))) = \sum_{p
  \in M^u} \on{Tr}(\gamma,\V_p).
\end{equation}

We would like to apply formula \eqref{l factor1} in our
situation. However, $\on{Loc}_{\LG}$ is not a scheme, but an
algebraic stack (unless $\LG$ is a torus), so we need an analogue of
\eqref{l factor1} for algebraic stacks (and more generally, for
derived algebraic stacks, since $\on{Loc}_{\LG}$ is not smooth as an
ordinary stack).

Let $M$ be as above and $\Delta: M \hookrightarrow M^2$ the diagonal
embedding. Observe that
\begin{equation}    \label{diagonal}
\Delta^* \Delta_*(\V) \simeq \V \otimes \Omega^\bullet(M),
\end{equation}
and hence we can rewrite \eqref{l factor1} as follows:
\begin{equation}    \label{l factor2}
\on{Tr}(\gamma,H^\bullet(M,\Delta^* \Delta_*(\V))) = \sum_{p \in M^u}
\on{Tr}(\gamma,\V_p).
\end{equation}

Now we propose formula \eqref{l factor2} as a {\em conjectural
  generalization of the Atiyah--Bott--Lefschetz fixed point formula}
to the case that $M$ is a smooth algebraic stack, and more generally,
smooth derived algebraic stack (provided that both sides are
well-defined). We may also allow here $\V$ to be a perfect complex (as
in \cite{Illusie}).

Even more generally, we drop the assumption that $u$ has fixed points
(or that the graph $\Gamma_u$ of $u$ and the diagonal $\Delta$ are
transversal) and conjecture the following general fixed point formula
of Atiyah--Bott type for (derived) algebraic stacks.

\begin{conj}    \label{conj abl formula}
\begin{equation}    \label{l factor3}
\on{Tr}(\gamma,H^\bullet(M,\Delta^* \Delta_*(\V))) =
\on{Tr}(\gamma,H^\bullet(M^u,i_u^*(\V))),
\end{equation}
where $M^u = \Gamma_u \underset{M \times M}\times \Delta$
is the fixed locus of $u$ and $i_u: M^u \to M \times M$.
\end{conj}
It is possible that \conjref{conj abl formula} may be proved using the
methods of \cite{BFN,BN}.\footnote{After these notes were posted on
  the arXiv, we were informed by A. Polishchuk that our conjecture
  could be proved using the methods of \cite{P}. We thank Polishchuk
  for a useful discussion, which helped us to correct an inaccuracy in
  formula \eqref{l factor3}.}

We want to apply \eqref{l factor2} to the left hand side of
\eqref{more general}, which is
$$
\on{RHom}(\Delta^* \Delta_*(\OO),{\mc F}),
$$
where ${\mc F} = {\mc F}_{d,\rho}$. This is not exactly in the form of
the left hand side of \eqref{l factor2}, but it is very close. Indeed,
if $M$ is a smooth scheme, then
$$
\on{RHom}(\Delta^* \Delta_*(\OO),{\mc F}) \simeq {\mc F} \otimes
\Lambda^\bullet(TM),
$$
so we obtain the exterior algebra of the tangent bundle instead of the
exterior algebra of the cotangent bundle. In our setting, we want
$\gamma$ to be the Frobenius, and so a fixed point is a homomorphism
$\pi_1(X) \to {}\LG$. The tangent space to $\sigma$ (in
the derived sense) should then be identified with the cohomology
$H^\bullet(X,\on{ad} \circ \, \sigma)[1]$, and the cotangent space with
its dual. Hence the trace of the Frobenius on the exterior algebra of
the tangent space at $\sigma$ is the $L$-function
$L(\sigma,\on{ad},s)$ evaluated at $s=0$. Poincar\'e duality implies
that the trace of the Frobenius on the exterior algebra of the
cotangent space at $\sigma$ is
$$
L(\sigma,\on{ad},1) = q^{-d_G} L(\sigma,\on{ad},0),
$$
so the ratio between the traces on the exterior algebras of the
tangent and cotangent bundles at the fixed points results in an
overall factor which is a power of $q$ (which is due to our choice of
conventions).

Hence, by following this argument and switching from $\C$ to $\Fq$, we
obtain (up to a power of $q$) the trace formula \eqref{classical} from
the isomorphism \eqref{more general}.

The isomorphism \eqref{more general} is still tentative, because there
are some unresolved issues in the definition of the two sides (see
\cite{FN} for more details).  Nevertheless, we hope that further study
of \eqref{more general} will help us to gain useful insights into the
trace formula and functoriality. I refer the reader to \cite{FN},
where in particular the abelian example is worked out and possible
applications are discussed in the general case.

\section{Relative geometric trace formula}    \label{rel trace}

In the previous section we discussed a geometrization of the trace
formula \eqref{tr}. It appears in the framework of the categorical
Langlands correspondence as the statement that the RHom's of kernels
of certain natural functors are isomorphic. These kernels are sheaves
on algebraic stacks over the squares $\Bun_G \times \Bun_G$ and
$\on{Loc}_{\LG} \times \on{Loc}_{\LG}$.

It is natural to ask what kind of statement we may obtain if we
consider instead the RHom's of sheaves on the stacks $\Bun_G$ and
$\on{Loc}_{\LG}$ themselves.

In this section we will show, following closely \cite{FN}, that this
way we obtain what may be viewed as a geometric analogue of the
so-called {\em relative trace formula}. On the spectral side of this
formula we also have a sum like \eqref{trace1}, but with one important
modification; namely, the eigenvalues $N_{\sigma}$ are weighted with
the factor $L(\sigma,\on{ad},1)^{-1}$.  The insertion of this factor
was originally suggested by Sarnak in \cite{Sarnak} and further
studied by Venkatesh \cite{Ven}, for the group $GL_2$ in the number
field context. The advantage of this formula is that the summation is
expected to be only over tempered representations and we remove the
multiplicity factors $m_\sigma$.

\subsection{Relative trace formula}

We recall the setup of the relative trace formula.

Let $G$ be a split simple algebraic group over $k=\Fq$.
In order to state the relative trace formula, we need
to choose a non-degenerate character of $N(F)\bs N(\AD_F)$, where $N$ is
a maximal unipotent subgroup of $G$. A convenient way to define it is
to consider a twist of the group $G$. Let us pick a maximal torus $T$
such that $B=TN$ is a Borel subgroup. If the maximal torus $T$ admits
the cocharacter $\crho: \Gm \to T$ equal to half-sum of all positive
roots (corresponding to $B$), then let $K_X^{\crho}$ be the $T$-bundle
on our curve $X$ which is the pushout of the $\Gm$-bundle $K_X^\times$
(the canonical line bundle on $X$ without the zero section) under
$\crho$. If $\crho$ is not a cocharacter of $T$, then its square is,
and hence this $T$-bundle is well-defined for each choice of the
square root $K_X^{1/2}$ of $K_X$. We will make that choice once and for
all.\footnote{This choice is related to the ambiguity of the
equivalence \eqref{na fm}, see the footnote on page \pageref{foot}.} 

Now set
$$
G^K = K_X^{\crho} \underset{T}\times G, \qquad N^K =
K_X^{\crho} \underset{T}\times N,
$$
where $T$ acts via the adjoint action. For instance, if $G=GL_n$, then
$GL_n^K$ is the group scheme of automorphism of the rank $n$
bundle ${\mc O} \oplus K_X \oplus \ldots \oplus
K_X^{\otimes(n-1)}$ on $X$ (rather than the trivial bundle ${\mc
O}^{\oplus n}$).

We have
\begin{equation}    \label{comm}
N^K/[N^K,N^K] = K_X^{\bigoplus \on{rank}(G)}.
\end{equation}
Now let $\psi: \Fq \to \C^\times$ be an additive
character, and define a character $\Psi$ of $N^K(\AD_F)$ as follows
$$
\Psi((u_x)_{x \in |X|}) = \prod_{x \in |X|} \prod_{i=1}^{\on{rank} G}
\psi(\on{Tr}_{k_x/k} \on{Res}_x(u_{x,i})),
$$
where $u_{x,i} \in K_X(F_x)$ is the $i$th projection of $u_x \in
N^K(F_x)$ onto $K_X(F_x)$ via the isomorphism \eqref{comm}. We
denote by $k_x$ the residue field of $x$, which is a finite extension
of the ground field $k=\Fq$.

By the residue formula, $\Psi$ is trivial on the subgroup $N^K(F)$
(this was the reason why we introduced the twist). It is also trivial
on $N^K({\mc O})$.

In what follows, in order to simplify notation, we will denote
$G^K$ and $N^K$ simply by $G$ and $N$.

Given an automorphic representation $\pi$ of $G(\AD_F)$, we have the
Whittaker functional $W: \pi \to \C$,
$$
W(f) = \int_{N(F)\bs N(\AD_F)} f(u) \Psi^{-1}(u) du,
$$
where $du$ is the Haar measure on $N(\AD_F)$ normalized so that the
volume of $N(\OO_F)$ is equal to $1$.

We choose, for each unramified automorphic representation $\pi$, a
non-zero $G(\OO_F)$-invariant function $f_\pi \in \pi$ on $G(F)\bs
G(\AD_F)$.

Let $K$ again be a kernel on the square of $\Bun_G(k) =
G(F) \bs G(\AD_F)/G(\OO_F)$ and ${\mb K}$ the
corresponding integral operator acting on unramified automorphic
functions. The simplest unramified version of the relative trace
formula reads (here we restrict the summation to cuspidal automorphic
representations $\pi$)
\begin{multline}    \label{simplest}
\sum_{\pi} \ol{W_{\Psi}(f_{\pi})} \; W_\Psi(K \cdot f_\pi) ||f_{\pi}||^{-2} = \\
\underset{N(F) \bs N(\AD_F)}\int \quad \underset{N(F) \bs N(\AD_F)}\int
K(u_1,u_2) \Psi^{-1}(u_1) \Psi(u_2) du_1 du_2,
\end{multline}
(see, e.g., \cite{Jacquet}), where
$$
||f||^2 = \underset{G(F)\bs G(\AD_F)}\int |f(g)|^2 dg,
$$
and $dg$ denotes the invariant Haar measure normalized so that the
volume of $G(\OO_F)$ is equal to $1$. Note that
$$
||f||^2 = q^{d_G} L(G) ||f||^2_T,
$$
where $||f||^2_T$ is the norm corresponding to the Tamagawa measure,
$d_G = (g-1) \dim G$,
$$
L(G) = \prod_{i=1}^\ell \zeta(m_i+1),
$$
where the $m_i$ are the exponents of $G$.

The following conjecture was communicated to us by B. Gross and
A. Ichino. In the case of $G=SL_n$ or $PGL_n$, formula \eqref{rs1}
follows from the Rankin--Selberg convolution formulas (see
\cite{FN}). Other cases have been considered in
\cite{GP,Ich1,Ich2}. For a general semi-simple group $G$ of adjoint
type formula \eqref{rs1} has been conjectured by A. Ichino and
T. Ikeda assuming that $\pi$ is square-integrable. Note that if a
square-integrable representation is tempered, then it is expected to
be cuspidal, and that is why formula \eqref{rs1} is stated only for
cuspidal representations.

Recall that an $L$-packet of automorphic representations is called
{\em generic} if each irreducible representation $\pi = \bigotimes'
\pi_x$ from this $L$-packet has the property that the local $L$-packet
of $\pi_x$ contains a generic representation (with respect to a
particular choice of non-degenerate character of $N(F_x)$).

\begin{conj}    \label{gen}
Suppose that the $L$-packet corresponding to an unramified $\sigma:
W_F \to {}\LG$ is generic. Then it contains a unique, up to an
isomorphism, irreducible representation $\pi$ such that $W_\Psi(f_\pi)
\neq 0$, with multiplicity $m_\pi = 1$. Moreover, if this $\pi$ is
in addition cuspidal, then the following formula holds:
\begin{equation}    \label{rs1}
|W_\Psi(f_\pi)|^2 ||f_\pi||^{-2} = q^{d_N-d_G} L(\sigma,\on{ad},1)^{-1}
|S_\sigma|^{-1},
\end{equation}
where $S_\sigma$ is the (finite) centralizer of the image of $\sigma$
in $\LG$, $d_N = -(g-1)(4 \langle \rho,\crho \rangle - \dim N)$ (see
formula \eqref{dim BunN}), and $d_G = (g-1) \dim G$.
\end{conj}

Furthermore, we expect that if the $L$-packet corresponding to
$\sigma: W_F \times \SL_2 \to {}\LG$ is non-generic, then
\begin{equation}    \label{non-gen}
L(\sigma,\on{ad},1)^{-1} = 0.
\end{equation}
Ichino has explained to us that according to Arthur's conjectures,
square-integrable non-tempered representations are non-generic.

Recall that we have ${\mb K} \cdot f_\pi = N_\sigma f_\pi$. Therefore
\conjref{gen} and formula \eqref{simplest} give us the following:
\begin{multline}    \label{relative}
q^{-d_G} \sum_{\sigma: W_F \to {}\LG} N_\sigma \cdot
L(\sigma,\on{ad},1)^{-1} |S_\sigma|^{-1} =
\\ q^{-d_N} \int \int K(u_1,u_2) \Psi^{-1}(u_1) \Psi(u_2) du_1 du_2.
\end{multline}
On the left hand side we sum only over unramified $\sigma$, and only
those of them contribute for which the corresponding $L$-packet of
automorphic representations $\pi$ is generic.

We expect that the left hand side of formula \eqref{relative} has the
following properties:

\begin{itemize}

\item[(1)] It does not include homomorphisms $\sigma: \SL_2 \times W_F
  \to \LG$ which are non-trivial on the Artur's $\SL_2$.

\item[(2)] The multiplicity factor $m_\sigma$ of formula
  \eqref{classical} disappears, because only one irreducible
  representation from the $L$-packet corresponding to $\sigma$ shows
  up (with multiplicity one).

\item[(3)] Since $S_\sigma$ is the group of automorphisms of $\sigma$,
  the factor $|S_\sigma|^{-1}$ makes the sum on the left hand side
  \eqref{relative} look like the Lefschetz fixed point formula for
  stacks.

\end{itemize}

\subsection{Geometric meaning: right hand side}

Now we discuss the geometric meaning of formula \eqref{relative},
starting with the right hand side. Let $\Bun^{\F_T}_N$ be the moduli
stack of $B=B^K$ bundles on $X$ such that the corresponding $T$-bundle
is $\F_T = K_X^{\crho}$. Note that
\begin{equation}    \label{dim BunN}
\dim \Bun^{\F_T}_N = d_N = -(g-1)(4 \langle \rho,\crho \rangle - \dim N).
\end{equation}

Let $\on{ev}: \Bun^{\F_T}_N \to {\mathbb G}_a$ be the map
constructed in \cite{FGV:w}.

For instance, if $G=GL_2$, then $\Bun^{\F_T}_N$ classifies
rank two vector bundles $\V$ on $X$ which fit in the exact sequence
\begin{equation}    \label{exten}
0 \to K_X^{1/2} \to \V \to K_X^{-1/2} \to 0,
\end{equation}
where $K_X^{1/2}$ is a square root of $K_X$ which we have fixed.
The map $\on{ev}$ assigns to such $\V$ its extension class in
$\on{Ext}(\OO_X,K_X) = H^1(X,K_X) \simeq {\mathbb G}_a$. For other
groups the construction is similar (see \cite{FGV:w}).

On ${\mathbb G}_a$ we have the Artin-Schreier sheaf ${\mc L}_\psi$
associated to the additive character $\psi$. We define the sheaf
$$
\wt\Psi = \on{ev}^*({\mc L}|_\psi)
$$
on $\Bun^{\F_T}_N$. Next, let $p: \Bun^{\F_T}_N \to \Bun_G$ be the
natural morphism. Let
$$
\Psi = p_!(\wt\Psi)[d_N-d_G]((d_N-d_G)/2).
$$

Then the right hand side of \eqref{relative} is equal to the trace of
the Frobenius on the vector space
\begin{equation}    \label{lhs1}
\on{RHom}(\Psi,{\mathbb K}_{d,\rho}(\Psi)).
\end{equation}
Here we use the fact that ${\mathbb D} \circ \K \simeq \K \circ
{\mathbb D}$ and ${\mathbb D}(\wt\Psi) \simeq \on{ev}^*({\mc
  L}|_{\psi^{-1}})[2d_N](d_N)$.

\subsection{Geometric meaning: left hand side}    \label{geom}

As discussed above, we don't have an algebraic stack parametrizing
homomorphisms $\sigma: W_F \to {}\LG$ if our curve $X$ is defined
over a finite field $\Fq$. But such a stack exists when $X$ is over
$\C$, though in this case there is no Frobenius operator on the
cohomology whose trace would yield the desired number (the left hand
side of \eqref{relative}). In this subsection we will define a certain
vector space (when $X$ is over $\C$) and conjecture that it is
isomorphic to the vector space \eqref{lhs1} which is the
geometrization of the right hand side of \eqref{relative} (and which
is well-defined for $X$ over both $\Fq$ and $\C$). This will be our
``relative geometric trace formula''. In the next subsection we will
show that this isomorphism is a corollary of the categorical version
of the geometric Langlands correspondence.

In order to define this vector space, we will use the coherent sheaf
${\mc F}_{d,\rho}$ on $\on{Loc}_{\LG}$ introduced in \secref{back}.
We propose that the geometrization of the left hand side of
\eqref{relative} in the case when $\K = \K_{d,\rho}$ is the cohomology
\begin{equation}   \label{rhs}
H^\bullet(\on{Loc}_{\LG},{\mc F}_{d,\rho}).
\end{equation}

The heuristic explanation for this proceeds along the lines of the
explanation given in the case of the ordinary trace formula in
\secref{abl fixed}, using the Atiyah--Bott--Lefschetz fixed point
formula.\footnote{We note that applications of the
  Atiyah--Bott--Lefschetz fixed point formula in the context of Galois
  representations have been previously considered by M. Kontsevich in
  \cite{Ko}.} We wish to apply it to the cohomology \eqref{rhs}. If
$\Loc_{\LG}$ were a smooth scheme, then we would have to multiply the
number $N_\sigma$ which corresponds to the stalk of ${\mc F}_{d,\rho}$
at $\sigma$, by the factor
\begin{equation}    \label{weighting factor}
\on{det}(1-\on{Fr},T^*_\sigma \Loc_{\LG})^{-1}.
\end{equation}
Recall that the tangent space to $\sigma$ (in the derived sense) may
be identified with the cohomology $H^\bullet(X,\on{ad} \circ \,
\sigma)[1]$. Using the Poincar\'e duality, we find that the factor
\eqref{weighting factor} is equal to
$$
L(\sigma,\on{ad},1)^{-1}.
$$

Therefore, if we could apply the Lefschetz fixed point formula to the
cohomology \eqref{rhs} and write it as a sum over all $\sigma: W_F \to
\LG$, then the result would be the left hand side of \eqref{relative}
(up to a factor that is a power of $q$). (Note however that since
$\Loc_{\LG}$ is not a scheme, but an algebraic stack, the weighting
factor should be more complicated for those $\sigma$ which admit
non-trivial automorphisms, see the conjectural fixed point formula
\eqref{l factor3} in \secref{abl fixed}.)

This leads us to the following {\em relative geometric trace formula}
(in the case of the functor $\K_{d,\rho}$).

\begin{conj}    \label{rel conj}
We have the following isomorphism of vector spaces:
\begin{equation}    \label{geom relative}
H^\bullet(\on{Loc}_{\LG},{\mc F}_{d,\rho}) \simeq
\on{RHom}_{\on{Bun_G}}(\Psi,{\mathbb K}_{d,\rho}(\Psi)).
\end{equation}
\end{conj}

Now we explain how the isomorphism \eqref{geom relative} fits in the
framework of a categorical version of the geometric Langlands
correspondence.

\subsection{Interpretation from the point of view of the categorical
  Langlands correspondence}

We start by asking what is the ${\mc D}$-module on $\Bun_G$
corresponding to the structure sheaf $\OO$ on $\on{Loc}_{\LG}$ under
the categorical Langlands correspondence of \secref{cat ver}. The
following answer was suggested by Drinfeld (see \cite{Laf}): it is the
sheaf $\Psi$ that we have discussed above.

The rationale for this proposal is the following: we have
$$
\on{RHom}_{\on{Loc}_{\LG}}(\OO,\OO_\sigma) = \C, \qquad \forall
\sigma,
$$
where $\OO_\sigma$ is again the skyscraper sheaf supported at
$\sigma$. Therefore, since $C(\OO_\sigma) = {\mc F}_\sigma$, we should
have, according to \eqref{RHom},
$$
\on{RHom}_{\Bun_G}(C(\OO),{\mc F}_\sigma) = \C, \qquad \forall
\sigma.
$$
According to the conjecture of \cite{LL}, the sheaf $\Psi$ has just
this property:
$$
\on{RHom}_{\Bun_G}(\Psi,{\mc F}_\sigma)
$$
is the one-dimensional vector space in cohomological degree $0$ (if we
use appropriate normalization for ${\mc F}_\sigma$).

This vector space should be viewed as a geometrization of the
Fourier coefficient of the automorphic function corresponding to ${\mc
  F}_\sigma$.

This provides
some justification for the assertion that\footnote{As explained in the
footnote on page \pageref{foot}, $C$ corresponds to a particular
choice of $K_X^{1/2}$. Given such a choice, $C(\OO)$ should be the
character sheaf $\Psi$ associated to that $K_X^{1/2}$.}
\begin{equation}    \label{CO}
C(\OO) = \Psi.
\end{equation}

Next, we rewrite \eqref{rhs} as
\begin{equation}    \label{rhs1}
\on{RHom}_{\on{Loc}_{\LG}}\left(\OO,\W_{d,\rho}(\OO) \right).
\end{equation}
Using the compatibility \eqref{WH} of $C$ with the
Wilson/Hecke operators and formulas \eqref{RHom} and \eqref{CO}, we
obtain that \eqref{rhs1} should be isomorphic to
\begin{equation}    \label{lhs2}
\on{RHom}_{\Bun_G}\left(\Psi,{\mathbb K}_{d,\rho}(\Psi) \right),
\end{equation}
which is the right hand side of \eqref{geom relative}.

Thus, we obtain that the relative geometric trace formula \eqref{geom
  relative} follows from the categorical version of the geometric
  Langlands correspondence. We hope that this formula may also be
  applied to the Functoriality Conjecture.


\begin{thebibliography}{FGKV}

\bi[AG]{AG} D. Arinkin and D. Gaitsgory, {\em Singular support of
  coherent sheaves, and the geometric Langlands conjecture}, Preprint
arXiv:1201.6343.

\bi[Art1]{Arthur:funct} J. Arthur, {\em The Principle Of Functoriality},
Bull. AMS {\bf 40} (2002) 39--53.

\bi[Art2]{Arthur:trace} J. Arthur, {\em An Introduction To The Trace
Formula}, Clay Mathematics Proceedings {\bf 4}, AMS, 2005.

\bi[AB]{AB} M. Atiyah and R. Bott, {\em A Lefschetz fixed point
  formula for elliptic complexes}, I. Ann. of Math. (2) {\bf 86}
  (1967) 374--407. II. Ann. of Math. (2) {\bf 88} 1968 451--491.

\bi[BL]{BL} A. Beauville and Y. Laszlo, {\em Un lemme de descente},
C.R. Acad.  Sci. Paris, S\'{e}r. I Math. {\bf 320} (1995) 335--340.

\bi[Be]{Behrend} K. Behrend, {\em Derived $l$-adic categories for
  algebraic stacks}, Mem. Amer. Math. Soc. {\bf 163} (2003), no. 774.

\bi[BeDh]{BeDh} K. Behrend and A. Dhillon, {\em Connected components
of moduli stacks of torsors via Tamagawa numbers},
Canad. J. Math. {\bf 61} (2009) 3--28.

\bi[BFN]{BFN} D. Ben-Zvi, J. Francis, and D. Nadler, {\em Integral
  Transforms and Drinfeld Centers in Derived Algebraic Geometry},
Preprint arXiv:0805.0157, to appear in Journal of AMS.

\bi[BN]{BN} D. Ben-Zvi and D. Nadler, {\em Loop Spaces and
  Connections}, Preprint arXiv:1002.3636.

\bi[BD]{BD} A. Beilinson and V. Drinfeld, {\em Quantization of
  Hitchin's integrable system and Hecke eigensheaves}, Preprint,
  available at www.math.uchicago.edu/$\sim$mitya/langlands

\bi[Bou]{Bourbaki} N. Bourbaki, {\em Groupes et alg`ebres de Lie}
Chapitres IV, V, VI, Hermann, Paris 1968.

\bi[CL]{CL} P.-H. Chaudouard and G. Laumon, {\em Le lemme fondamental
  pond\'er\'e I : constructions g\'eom\'e\-triques}, Preprint
  arXiv:0902.2684.

\bi[D1]{Dr1} V.G. Drinfeld, {\em Two-dimensional $\ell$--adic
representations of the fundamental group of a curve over a finite
field and automorphic forms on $GL(2)$}, Amer. J. Math. {\bf 105}
(1983) 85--114.

\bi[D2]{Dr2} V.G. Drinfeld, {\em Langlands conjecture for $GL(2)$ over
  function field}, Proc. of Int. Congress of Math. (Helsinki, 1978),
  pp. 565--574; {\em Moduli varieties of $F$--sheaves},
  Funct. Anal. Appl. {\bf 21} (1987) 107--122; {\em The proof of
  Petersson's conjecture for $GL(2)$ over a global field of
  characteristic $p$}, Funct. Anal. Appl. {\bf 22} (1988) 28--43.

\bi[F1]{F:houches} E.~Frenkel, {\em Lectures on the Langlands Program
  and Conformal Field Theory}, in {\em Frontiers in number Theory,
  Physics and Geometry II}, eds. P. Cartier, e.a., pp. 387--536,
Springer Verlag, 2007 (hep-th/0512172).

\bi[F2]{F:bourbaki} E.~Frenkel, {\em Gauge theory and Langlands
  duality}, S\'eminaire Bourbaki, Juin 2009
  (arXiv:0906.2747).

\bi[FGKV]{FGKV} E. Frenkel, D. Gaitsgory, D. Kazhdan and K. Vilonen,
{\em Geometric realization of Whittaker functions and the Langlands
conjecture}, Journal of AMS {\bf 11} (1998) 451--484.

\bi[FGV1]{FGV:w} E.~Frenkel, D.~Gaitsgory and K.~Vilonen, {\em
Whittaker patterns in the geometry of moduli spaces of bundles on
curves}, Annals of Math. {\bf 153} (2001) 699--748,
(arXiv:math/9907133).

\bi[FGV2]{FGV} E.~Frenkel, D.~Gaitsgory and K.~Vilonen, {\em On the
geometric Langlands conjecture}, Journal of AMS {\bf 15} (2001)
367--417 (arXiv:math/0012255).

\bi[FLN]{FLN} E.~Frenkel, R.~Langlands and B.C.~Ng\^o, {\em La formule
  des traces et la functorialit\'e. Le d\'ebut d'un Programme},
Ann. Sci. Math. Qu\'ebec {\bf 34} (2010) 199--243 (arXiv:1003.4578).

\bibitem[FN]{FN} E. Frenkel and B.C.~Ng\^o, {\em Geometrization of
    trace formulas}, Bull. Math. Sci. {\bf 1} (2011) 1--71
  (arXiv:1004.5323).

\bibitem[FW]{FW} E. Frenkel and E. Witten, {\em Geometric Endoscopy
    and Mirror Symmetry}, Communications in Number Theory and Physics,
    {\bf 2} (2008) 113--283 (arXiv:0710.5939).

\bi[FK]{Weil} E. Freitag, R. Kiehl, {\em Etale Cohomology and the
      Weil conjecture}, Springer, 1988.

\bi[G]{Ga} D. Gaitsgory, {\em On a vanishing conjecture appearing in the
geometric Langlands correspondence}, Ann. Math. {\bf 160} (2004)
617--682.

\bi[GM]{GM} S.I. Gelfand and Yu.I. Manin, {\em Homological Algebra},
Encyclopedia of Mathematical Sciences {\bf 38}, Springer, 1994.

\bibitem[GP]{GP} B.H. Gross and D. Prasad, {\em On the decomposition
of a representation of $SO_n$ when restricted to $SO_{n-1}$},
Canad. J. Math. {\bf 44} (1992) 974--1002.

\bi[Ha]{Har} G. Harder. {\em \"Uber die Galoiskohomologie halbenfacher
  algebraischer Gruppen. III}, J. Reine Angew. Math. {\bf 274/275}
(1975) 125--138.

\bi[H1]{Hit1} N. Hitchin, {\em The self-duality equations on a Riemann
  surface}, Proc. London Math. Soc. (3) {\bf 55} (1987) 59--126.

\bi[H2]{Hi2} N. Hitchin, {\em Stable Bundles And Integrable
Systems}, Duke Math. J. {\bf 54} (1987) 91--114.

\bi[Il]{Illusie} L. Illusie, {\em Formule de Lefschetz}, SGA 5,
Lect. Notes in Math. {\bf 589}, pp. 73--137, Springer Verlag, 1977.

\bi[Ic]{Ich1} A. Ichino, {\em On critical values of adjoint
  $L$-functions for $GSp(4)$}, Preprint, available at
  http://www.math.ias.edu/$\sim$ichino/ad.pdf

\bi[IcIk]{Ich2} A. Ichino and T. Ikeda, {\em On the periods of
  automorphic forms on special orthogonal groups and the Gross-Prasad
  conjecture}, Geometric And Functional Analysis {\bf 19} (2010)
1378--1425.

\bi[J]{Jacquet} H. Jacquet, {\em A guide to the relative trace
  formula}, in Automorphic Representations, L-functions and
  Applications: Progress and Prospects, Ohio State University
  Mathematical Research Institute Publications, Volume 11,
  pp. 257--272, De Gruyter, Berlin, 2005.

\bi[JL]{JL} H. Jacquet and R. Langlands, {\em Automorphic Forms on
  $GL(2)$}, Lect. Notes in Math. {\bf 114}, Springer, 1970.

\bi[KW]{KW} A. Kapustin and E. Witten, {\em Electric-magnetic duality
  and the geometric Langlands Program}, Preprint hep-th/0604151.

\bi[KS]{KS} M.~Kashiwara and P.~Schapira, {\em Sheaves on Manifolds},
Springer, 1990.

\bi[K]{Ko} M. Kontsevich, {\em Notes on motives in finite
  characteristic}, Preprint arXiv:math/0702206.

\bi[LLaf]{LLaf} L. Lafforgue, {\em Chtoucas de Drinfeld et
   correspondance de Langlands}, Invent. Math. {\bf 147} (2002)
   1--241.

\bi[VLaf]{Laf} V. Lafforgue, {\em Quelques calculs reli\'es \`a la
correspondance de Langlands g\'eom\'etrique pour ${\mathbb P}^1$},
available at http://people.math.jussieu.fr/$\sim$vlafforg/geom.pdf

\bi[LL]{LL1} J.-P. Labesse and R.P. Langlands, {\em
  L-indistinguishability for SL(2)}, Canadian J. of Math.
  {\bf 31} (1979) 726--785.

\bi[LafL]{LL} V. Lafforgue and S. Lysenko, {\em Compatibility of the
Theta correspondence with the Whittaker functors}, Preprint
arXiv:0902.0051.

\bi[L1]{L} R. Langlands, {\em Problems in the theory of automorphic
forms}, in Lect. Notes in Math. {\bf 170}, pp. 18--61, Springer Verlag,
1970.

\bi[L2]{L:BE} R. Langlands, {\em Beyond endoscopy}, in {\em
Contributions to automorphic forms, geometry, and number theory},
pp. 611--697, Johns Hopkins Univ. Press, Baltimore, MD, 2004.

\bi[L3]{L:PR} R. Langlands, {\em Un nouveau point de rep\`ere dans la
th\'eorie des formes automorphes}, Canad. Math. Bull. {\bf 50} (2007)
no. 2, 243--267.

\bi[L4]{L:ST} E. Langlands, {\em Singularit\'es et Transfert}, 2010,
available at http://publications.ias.edu/rpl

\bi[L5]{L:IAS} E. Langlands, {\em Functoriality and Reciprocity} (in
Russian), 2011, http://publications.ias.edu/rpl

\bi[LO]{LO} Y. Laszlo and M. Olsson, {\em The six operations for
  sheaves on Artin stacks. I. Finite coefficients}, Publ. Math. IHES
   {\bf 107} (2008) 109--168.

\bi[Lau1]{Laumon:const} G. Laumon, {\em Transformation de Fourier,
  constantes d'\'equations fonctionelles et conjecture de Weil},
  Publ. IHES {\bf 65} (1987) 131--210.

\bi[Lau2]{Laumon} G. Laumon, {\em Transformation de Fourier
g\'{e}n\'{e}ralis\'{e}e}, Preprint alg-geom/9603004.

\bi[Lau3]{Laumon:duke} G. Laumon, {\em Correspondance de Langlands
g\'eom\'etrique pour les corps de fonctions}, Duke Math. J. {\bf 54}
(1987) 309-359.

\bi[Ly1]{Lys1} S. Lysenko, {\em Geometric theta-lifting for the dual
  pair $SO_{2m}, Sp_{2n}$}, Ann. Sci. \'Ecole Norm. Sup. {\bf 44}
(2011) 427--493.

\bi[Ly2]{Lys2} S. Lysenko, {\em Geometric theta-lifting for the dual
  pair $GSp_{2n}, GO_{2m}$}, Preprint arXiv:0802.0457.

\bi[M]{Milne} J.S. Milne, {\em \'Etale cohomology}, Princeton
University Press, 1980.

\bi[MV]{MV} I.~Mirkovi\'c and K.~Vilonen, {\em Geometric Langlands
duality and representations of algebraic groups over commutative
rings}, Preprint math.RT/0401222.

\bi[N1]{Ngo:Endo} B.C. Ng\^o, {\em Fibration de Hitchin et endoscopie},
Invent. Math. {\bf 164} (2006) 399--453.

\bi[N2]{Ngo:FL} B.C. Ng\^o, {\em Le lemme fondamental pour les
  algebres de Lie}, Preprint arXiv:0801.0446.

\bi[Ni1]{Nis1} Ye. Nisnevich, {\em \'Etale cohomology and Arithmetic
  of Semisimple Groups}, Ph.D. Thesis. Harvard University,
1982. Available at http://proquest.umi.com/pqdlink?RQT=306

\bi[Ni2]{Nis2} Ye. Nisnevich, {\em Espaces homog\`enes principaux
  rationnellement triviaux et arithm\'etique des sch\'emas en groupes
  r\'eductifs sur les anneaux de Dedekind,} C. R. Acad. Sci. Paris
S\'er. I Math. {\bf 299} (1984) 5--8.

\bi[P]{P} A. Polishchuk, {\em Lefschetz type formulas for
  dg-categories}, Selecta Math. {\bf 20} (2014) 885--928
(arXiv:1111.0728).

\bi[R]{Rothstein} M. Rothstein, {\em Connections on the total Picard
  sheaf and the KP hierarchy}, Acta Applicandae Mathematicae {\bf 42}
(1996) 297--308.

\bi[S]{Sarnak} P. Sarnak, {\em Comments on Robert Langland's Lecture:
"Endoscopy and Beyond"}, available at
http://www.math.princeton.edu/sarnak/SarnakLectureNotes-1.pdf

\bi[Si]{Simp} C. Simpson, {\em Higgs bundles and local systems}, Publ.
Math. IHES {\bf 75} (1992) 5--95.

\bi[Ve]{Ven} A. Venkatesh, {\em ``Beyond endoscopy'' and special forms
on $GL(2)$}, J. Reine Angew. Math. {\bf 577} (2004) 23--80.

\bi[Vi]{Vinberg} E. B. Vinberg, {\em On reductive algebraic semigroups,
Lie Groups and Lie Algebras}, in E. B. Dynkin Seminar (S. Gindikin,
E. Vinberg, eds.), AMS Transl. (2) {\bf 169} (1995), 145--182.

\bi[W]{Wiles} A. Wiles, {\em Modular elliptic curves and Fermat's last
theorem}, Ann. of Math. (2) {\bf 141} (1995) 443--551.

R. Taylor and A. Wiles, {\em Ring-theoretic properties of certain
Hecke algebras}, Ann. of Math. (2) {\bf 141} (1995) 553--572.

C. Breuil, B. Conrad, F. Diamond and R. Taylor, {\em On the
modularity of elliptic curves over ${\mb Q}$: wild 3-adic
exercises}, J. Amer. Math. Soc. {\bf 14} (2001) 843--939.

\end{thebibliography}
\end{document}